\documentclass[12pt]{report} 

\usepackage{regina}
\usepackage{latexsym}
\usepackage{float}
\usepackage{amsfonts}
\usepackage{amssymb}
\usepackage{makeidx}     
\usepackage{graphicx}    
\usepackage{multicol}    
\usepackage{slashbox}

\usepackage{comment}
\usepackage{amstext,amsmath}
\usepackage{cite}
\usepackage{amsthm}
\usepackage{enumerate}
\usepackage{makeidx}
\usepackage[all]{xy}
\usepackage[vcentermath]{youngtab}
\usepackage{MnSymbol} 
\usepackage{extdash}

\newtheoremstyle{plainsl}%
       {\topsep}
       {\topsep}
       {\slshape} 
       {}
       {\normalfont\bfseries}	
       {.}
       { }
       {}	

\theoremstyle{plainsl}

\makeatletter
\newtheoremstyle{break}
   {\topsep}{\topsep}%
   {\slshape}{}%
   {\bfseries}{}%
   {  }
   {\thmname{#1}\thmnumber{\@ifnotempty{#1}{ }\@upn{#2}}%
    \thmnote{ {\bfseries(#3)}}{.}}%
\makeatother
\theoremstyle{break}

\newtheorem{theorem}{Theorem}[section]
\newtheorem{thm}[theorem]{Theorem}
\newtheorem{defn}[theorem]{Definition}

\newtheorem{example}[theorem]{Example}
\let\oldexample\example
\renewcommand{\example}{\oldexample\normalfont}

\newtheorem{lem}[theorem]{Lemma}
\newtheorem{cor}[theorem]{Corollary}
\newtheorem{obs}[theorem]{Observation}
\newtheorem{prop}[theorem]{Proposition}

\newtheorem{conj}[theorem]{Conjecture}

\newcommand\txtsl[1]{\textsl{#1}\index{#1}}
\newcommand\stackplus[1]{\makebox[0ex][l]{$+$} \raisebox{-.75ex}{\makebox[2ex]{$_{#1}$}}}
\DeclareMathOperator{\nul}{null}
\DeclareMathOperator{\T}{T}
\DeclareMathOperator{\CC}{cc}
\DeclareMathOperator{\mr}{mr}
\DeclareMathOperator{\rank}{rank}
\DeclareMathOperator{\TT}{\mathcal{T}}
\DeclareMathOperator{\PP}{\mathcal{P}}
\DeclareMathOperator{\diam}{diam}
\DeclareMathOperator{\Spec}{Spec}
\DeclareMathOperator{\supp}{Supp}
\DeclareMathOperator{\bfx}{{\mathbf x}}

\makeindex             

\title{Zero Forcing Sets for Graphs }
\author{Fatemeh Alinaghipour Taklimi}
\dept{Mathematics}
\faculty{Science}
\submitdate{August 2013}
\copyrightyear{2013}
\copyrighttrue
\figurespagefalse
\tablespagefalse

\begin{document}
\beforepreface
\addcontentsline{toc}{chapter}{Abstract}

\chapter*{Abstract}
For any simple \textsl{graph} $G$ on $n$ vertices, the (\textsl{positive semi-definite}) \textsl{minimum rank} of $G$ is defined to be  the smallest possible rank  among all (positive semi-definite) real symmetric $n\times n$ matrices whose entry in position $(i,j)$, for $i\neq j$, is non-zero  if $ij$ is an edge in $G$ and zero otherwise. Also, the (\textsl{positive semi-definite}) \textsl{maximum nullity} of $G$ is defined to be the largest possible nullity of a (positive semi-definite) matrix in the above set of matrices. 
In this thesis we study two graph parameters, namely the \textsl{zero forcing number} of $G$, $Z(G)$, and
the \textsl{positive zero forcing number} of $G$, $Z_+(G)$, which bound the maximum nullity and  the positive semi-definite maximum nullity from above, respectively. With regard to the zero forcing number, we introduce some new families of graphs for which the zero forcing number and the maximum nullity are the same. Also we establish an equality between the zero forcing number and the \textsl{path cover} number for a new family of graphs. In addition, we establish a connection between the zero forcing number and the \textsl{chromatic number} of graphs. With regard to the positive zero forcing number, we introduce the concept of \textsl{forcing trees} in a graph and we establish a connection between the positive zero forcing number and the \textsl{tree cover number}. Also we study families of graphs for which these parameters coincide.        
 In addition, we provide some new results on the connections of this parameter with other graph parameters, including  the \textsl{independence} number and the chromatic number of $G$.   

\addcontentsline{toc}{chapter}{Acknowledgments}
\chapter*{Acknowledgments}
Foremost, I would like to express my sincere gratitude to my advisors professor Shaun Fallat and  Dr. Karen Meagher for their continuous support of my PhD studies and research, for their patience, motivation, enthusiasm, and immense knowledge. I am so lucky to pursue my PhD under the supervision of such brilliant mathematicians.

I would also like to thank the other members of my PhD committee, Dr. Sandra Zilles and  Dr. Shonda Gosselin and appreciate their valuable time spent on reading my thesis.

I also appreciate the financial supports of Dr. Fallat, Dr. Meagher and the department of Mathematics and Statistics at the University of Regina, during my PhD studies. 

Last but not the least, I would like to thank my husband, Bahman, for all his helps, especially for drawing many of the graphs in my thesis. I really appreciate it.

\bigskip
\hfill {\bf Fatemeh Alinaghipour Taklimi}
\vspace{-.3cm}

\hfill {Regina, Canada, July 2013}
 
\addcontentsline{toc}{chapter}{Dedication}

\chapter*{Dedication}
\begin{center}
To  my  dear husband and best friend 

{\large\bf Bahman}

for his sincere love, endless support and encouragement, and to my lovely son 

{\large\bf Ilia}

for all the happiness he brought to us...

\vspace{2cm}
\end{center}
 \tableofcontents
\listoffigures 
 
 \afterpreface

\chapter{Introduction}\label{intro}
Pattern classes of matrices represent an important part of  the field of combinatorial matrix theory. Examples of related applications are models of problems in economics, mathematical biology (especially ecological food webs), statistical mechanics, and communication complexity.

\section{The minimum rank problem}

The minimum rank problem for a simple graph, or the minimum rank problem for short, determines the minimum rank among all real symmetric matrices whose zero-nonzero pattern of off-diagonal entries is described by a given simple graph $G$. 

 A graph or digraph ``describes'' the zero-nonzero pattern of a family of matrices; namely, the $(i,j)$-entry of the matrix is either non-zero or zero, depending on whether vertices $v_i$ and $v_j$ are adjacent or non-adjacent in the graph. Note that there is no restriction on the diagonal entries of the matrix. In particular, the ``adjacency matrix'' of a graph $G$ is described by $G$; this is the (0,1)-matrix whose rows and columns are labeled by the vertices of $G$ and whose $(i,j)$-entry is $1$ if and only if the vertices $v_i$ and $v_j$ are adjacent in $G$. The set of all real symmetric matrices which are described by a graph $G$, is denoted by $\mathcal{S}(G)$. For any graph $G$, we denote the minimum rank of $G$ by $\mr(G)$.

The \textsl{minimum rank problem} is to determine the minimum among the ranks of the matrices in such a family;  the determination of maximum nullity is then equivalent. Considerable progress has been made on the minimum rank problem for the family of real symmetric matrices described by a simple graph, although the problem is far from being solved in general. For example, the minimum rank problem has been completely solved for all types of tree patterns. The reader may refer to~\cite{MR2350678} to see a more detailed discussion on  the recent results in the  minimum~rank~problem.
 
 We can also view  the minimum rank problem as the problem of determining  the maximum nullity. Since the nullity is, by definition, the number of vectors in a basis for the null space, the core idea becomes the study of subspaces with structured bases. Let matrix $A$ be described by a graph $G$.  Essentially, if ``enough'' of the entries of a vector in the null space of $A$
are known to be zero, then, because of the pattern of $A$, we may deduce that the entire vector is zero. From this line of research, a new parameter was devised, and has become known as the \textsl{zero forcing number}. This parameter  is the minimum size of a ``zero forcing set'', which is a set of initially black-colored vertices of $G$ which, through a coloring process, can ``force'' the rest of the vertices to be black. The zero forcing number gives an upper bound on the maximum nullity of $A$. However, a zero forcing parameter is a completely combinatorial parameter, and thus has the advantage of being intimately related to the structures within a pattern, or more precisely, a graph.

It turns out, perhaps surprisingly, that the zero-nonzero pattern described by the graph has significant influence on the minimum rank. For example, the minimum rank of a path graph is always equal to the number of edges in the path.

If in the definition of the minimum rank and maximum nullity, one restricts the matrices considered to the family of positive semi-definite matrices, then the resulting quantities will be called \textsl{positive semi-definite minimum rank} and \textsl{positive semi-definite maximum nullity}, respectively. In addition to studying the zero forcing number, in this thesis we investigate a graph parameter called \textsl{positive zero forcing number}  which has the ``same nature'' as the zero forcing number. This is the minimum size of a \textsl{positive zero forcing set} for the graph which has, in turn, a similar definition as a zero forcing set with a different colour change rule. The goal for studying this parameter, similar to the conventional zero forcing number, is to bound  the positive semi-definite maximum nullity. 
 Although these two zero forcing numbers are analogous (as the two maximum nullities are), there has been more research conducted on the conventional zero forcing number.

\section{Relationship to other problems}
The  minimum rank of a family of matrices associated with a graph has also played a role in various other problems, such as the inverse eigenvalue problem,  singular graphs, biclique decompositions and the biclique partition number, and orthonormal labellings of graphs (see the survey paper~\cite{MR2350678} for details on these problems).

It is well-known that the sum of the rank and the nullity of an $n\times n$ matrix is $n$. This implies that, for any graph $G$, the sum of the minimum rank and the maximum nullity of a graph on $n$ vertices is $n$. Hence the solution of the minimum rank problem is equivalent to the determination of the maximum multiplicity of an eigenvalue among the same family of matrices in the symmetric free-diagonal case. An important related problem is then, the so-called \textsl{inverse eigenvalue problem} of a graph $G$, which can be stated as: what eigenvalues are possible for a real symmetric matrix $A$ having nonzero off-diagonal entries determined by the edges of $G$? For example, if the minimum rank is $m$, then the multiplicity of an eigenvalue can be no more than $n-m$.

In spectral graph theory, an important tool is the $(0,1)$-adjacency matrix $A(G)$ of the graph $G$, and  a well-known open problem from 1957, is to characterize the graphs $G$ whose adjacency matrix is singular (see  \cite{MR0087952} for a survey of this problem). Further, many researchers have studied the dimension of the null space of the adjacency matrix; the nullity of $A(G)$ is an upper bound on the difference of $n$ and the minimum rank of $G$.

Partitioning the edges of a graph is an important research area in graph theory (see \cite{MR0289210}, \cite{MR929670}, \cite{MR1330781}, \cite{MR816059}). A \textsl{biclique partition} is a partition of the edges of a graph into subgraphs in such a way that each subgraph is a complete bipartite graph (or a biclique). The size of the smallest such partition is called the biclique partition number of $G$ and is denoted by $\text{bp}(G)$ (see also \cite{MR0300937}). In \cite{MR0289210} it is shown that
\[
\text{bp}(G)\geq \max\{i_+(A(G)), i_-(A(G))\},
\]
also
\[
n-\alpha(G)\geq \max\{i_+(A), i_-(A)\},\quad A\in \mathcal{S}(G),
\]
where $A(G)$ is the $(0,1)$-adjacency matrix of $G$ and $i_+(X)$ and $i_-(X)$ are the number of positive and negative eigenvalues of the symmetric matrix $X$, respectively, and $\alpha(G)$ is the \textsl{independence number} of $G$; that is the maximum size of an independent set in $G$. The second bound is often called the \textsl{inertia bound} on the independence number.  In the special case when $G$ is a bipartite graph it follows easily that $\text{bp}(G)\geq \frac{1}{2} \rank(A(G))\geq \frac{1}{2}\mr(G)$.

The orthogonal labeling of the vertices is another prominent subject in graph theory that also has ties to minimum rank. If $G = (V,E)$ is a graph, then an \textsl{orthonormal labeling} of $G$ of \textsl{dimension} $d$ is a function $f \,:\, V\longrightarrow \mathbb{R}^d$ such that $f(u)\cdot f(v) = 0$ whenever vertices $u$ and $v$ are not adjacent, and $|f(u)| = 1$ for all $u\in V$ (see \cite{MR2046632} for more details about orthonormal labeling). If  $d(G)$ denotes the smallest dimension $d$ over all orthonormal labellings of $G$, then $d(G)$ is equal to the minimum rank of a positive semidefinite  matrix whose graph is given by G (see \cite{MR2046632}). Lov{\'a}sz \cite{MR514926} has also shown that
\begin{equation}\label{lovasz}
\alpha(G)\leq \vartheta(G)\leq d(G)\leq \chi(\bar{G}),
\end{equation}
where, $\vartheta(G)$ is the  Lov{\'a}sz $\vartheta$ function and $\chi(\bar{G})$ is the chromatic number of the complement of $G$ (i.e., the clique cover number of $G$). Note that (\ref{lovasz}) provides lower and upper bounds on $d(G)$; hence it approximates the minimum rank of a  positive semidefinite  matrix described by $G$.

\section{Thesis objectives}

One of the main objectives in this thesis is to determine the families of graphs for which  the maximum nullity and the zero forcing number coincide. This is, in particular, important since the zero forcing number is a graph theoretical parameter which is more ``tangible'' than the maximum nullity. The latter is an algebraic parameter,  and in the majority of the cases more difficult to deal with. This will, therefore, ease the study of this algebraic concept. 
 
It has been shown \cite[Proposition 2.10]{MR2645093} that the path cover number of a graph is a lower bound for the zero forcing number of the graph. As another objective, we will investigate families of graphs for which these two parameters are equal.      

One of the other objectives of this thesis is to develop new approaches to studying the (positive) zero forcing number and to understand its possible connections with other graph parameters. 
As the zero forcing number and the positive  zero forcing number have a  ``graph colouring characteristic'',  one of the natural directions to this goal is studying the possible connections  of these two parameters and other colouring parameters. As a result, we establish some relationships with the  traditional graph colouring parameter, namely the chromatic number. Further, some results on the connection between positive zero forcing number and the chromatic number will be presented. 

The other connection which will be established is the relationship between the positive zero forcing number and the \textsl{tree cover number}, which is the minimum number of induced trees of a graph which cover the vertices of the graph. We show, in particular, that the positive zero forcing number is an upper bound for the tree cover number of the graph. Also, we will investigate families of graphs in which the positive zero forcing number and the tree cover number agree.

\section{Outline of the thesis}

This thesis includes seven chapters. Chapter~\ref{Basics} will cover some necessary background material on graph theory and matrix theory and classical facts which we will use in later chapters. 

In Chapter~\ref{min_rank} we will introduce the algebraic parameters (positive semi-definite) minimum rank and maximum nullity,  defined for a given simple graph. We will then provide some known facts about these concepts.

Chapter~\ref{ZFS} will provide all the information we need to begin the study of zero forcing sets and zero forcing numbers. This will illustrate how the zero forcing number bounds the maximum nullity of a graph. Then we study the graphs for which this bound is tight and provide some new results.  Next we investigate graphs for which the zero forcing number meets the path cover number. In addition, the Colin de Verdi$\grave{\text{e}}$re parameters for a graph and their connection to the zero forcing number will be discussed in Section~\ref{colin de}. We conclude this chapter with a section on our new results on bounding the chromatic number with the zero forcing number. 

In Chapter~\ref{zfs_and_graph_operations}, we  study the  effect of graph operations on the zero forcing number. In particular, we calculate the zero forcing number of some graph products in terms of those of the initial graphs.

In Chapter~\ref{PZFS}, which also includes some of the main new results of this thesis, we will first introduce the parameter \textsl{positive zero forcing number} and then show how it bounds the positive semi-definite maximum nullity (and thus the positive semidefinite minimum rank). Furthermore, we will define the concept of  \textsl{forcing trees} which provides connections between the positive zero forcing number and the \textsl{tree cover number}. Next we introduce one of the most important families of graphs for which the positive zero forcing number equals the tree cover number. In addition, we study these two graph parameters in more families of graphs.  Finally, we will  compare the positive zero forcing number with  the chromatic number and  the independence number of a graph.

The thesis will be concluded with Chapter~\ref{conclusion} in which we show some of  the potential  directions for future research work on this topic. We will list some open questions and conjectures which are closely related to our results.

\chapter{Graph Theory}\label{Basics}

This chapter presents some basic definitions, concepts and facts from graph theory and matrix theory needed throughout the thesis.
 The reader may also refer to  \cite{MR2368647}  and \cite{MR1367739} for additional details.


\section{Basic concepts}\label{graph theory}
Throughout this thesis all graphs are assumed to be finite, simple and undirected. We will often use the notation $G=(V,E)$ to denote the graph\index{graph} with non-empty vertex set $V=V(G)$ and edge set $E=E(G)$. An edge of $G$ with end-points $i$ and $j$ is denoted by $ij$. The \textsl{order}\index{order of a graph} of the graph $G$ is defined to be $|V(G)|$, this may also be written as $|G|$. Vertex $v$ is a \txtsl{neighbour} of vertex $u$ if $uv\in E(G)$. The set of all neighbours of $v$ is denoted by $N(v)$. The \txtsl{degree} of a vertex $v$ is $|N(v)|$ and is denoted by $d(v)$. The minimum degree over all the vertices of a graph $G$ is the \txtsl{minimum degree} of the graph $G$ and is denoted by $\delta(G)$. 
 A graph $H$ is a \txtsl{subgraph} of $G$ (denoted $H \leq G$) if $V(H)\subseteq V(G)$ and  $E(H) \subseteq E(G)$. A subgraph $H$ of a graph $G$ is said to be \textsl{induced}\index{graph!induced} if, for any pair of vertices $x$ and $y$ of $H$, $xy$ is an edge of $H$ if and only if $xy$ is an edge of $G$. Let  $S\subseteq V(G)$; then  the subgraph $H$ of $G$ induced by $S$ is denoted by $H=G[S]$. Also the subgraph induced by $V(G)\backslash S$ is denoted by $G\backslash H$. The \textsl{degree} of a vertex $v\in V(G)$ in an induced subgraph of the graph $G$, the subgraph $H$, is denoted by $d_H(v)$. 
 A \txtsl{supergraph} of a graph $G$ is a graph of which $G$ is a subgraph. A graph $G$ is called \textsl{complete}\index{graph!complete} if $G$ has all possible edges between its vertices. The complete graph on $n$ vertices is denoted by $K_n$. An \textsl{empty}\index{graph!empty} graph is a graph with empty edge set. The complement of the graph $G$ which is denoted by $\overline{G}$ is a graph on the same set of vertices such that two vertices are adjacent in $\overline{G}$ if and only if they are not adjacent in $G$. In particular, the empty graph is denoted by $\overline{K_n}$.

A \txtsl{path} on $n$ vertices is a graph $P_n$ with vertex set  $V(P_n)=\{v_1,\ldots,v_n\}$ in which $v_i$ is adjacent only to $v_{i+1}$, for all $i\in\{1,\ldots,n-1\}$. We may also write $P_n$ as $v_1v_2\cdots v_n$.
The graph formed by adding the edge $v_1v_n$ to $P_n$ is called a \txtsl{cycle} on $n$ vertices and is denoted by $C_n$. The cycle $C_n$ is said to be an \txtsl{odd cycle} (\txtsl{even cycle}) if $n$ is odd (even). A graph $G$ is said to be \textsl{connected}\index{graph!connected} if for any pair $x,y$ of vertices in $G$, there is a path between $x$ and $y$. A \txtsl{component} of a graph is a maximal connected subgraph of the graph. It is clear that  any connected graph has only one component. Any graph $G$ with more than one component is said to be \txtsl{disconnected}\index{graph!disconnected}. If a graph contains no cycle as a subgraph, then it is called a \txtsl{forest}. A connected forest is a \txtsl{tree}.

The length of a path is the number of edges of the path, i.e. the length of $P_n$ is $n-1$. Let $u$ and $v$ be two vertices in graph $G$. Then the \textsl{distance} between $u$ and $v$ in $G$, denoted by $d_G(u,v)$, is the length of a shortest path in $G$ between $u$ and $v$. The \txtsl{diameter} of a graph $G$ is, then, defined to be the maximum distance over all vertices:  
\[
\diam(G)=\max\{ d_G(u,v)\,\,|\,\, u,v\in V(G) \}.
\]

A subset $S$ of the vertex set of a graph is said to be an \txtsl{independent} set if no pair of elements of $S$ are adjacent.
A \textsl{multipartite}\index{graph!multipartite}  graph is a graph whose vertex set can be partitioned into subsets $X_1,\ldots, X_k$, with $k\geq2$, such that each $X_i$ is an independent set. The subsets $X_i$ are called the \txtsl{parts} of the graph.  A multipartite graph with two parts  is called a \textsl{bipartite}\index{graph!bipartite} graph.

\begin{thm}[{{see \cite{MR1367739}}}]\label{bipartite}
A graph is bipartite if and only if it does not contain an odd cycle. \qed
\end{thm}

 It is obvious from Theorem~\ref{bipartite} that all trees are bipartite. A \textsl{complete bipartite}\index{graph!complete bipartite} graph is a bipartite graph with bipartition $(X,Y)$ in which each vertex of $X$ is adjacent to every vertex of $Y$; if $|X|=m$ and $|Y|=n$, such a graph is denoted by $K_{m,n}$. A multipartite graph which has all possible edges in its edge set is called a \txtsl{complete multipartite graph}.

A \txtsl{clique} in a graph is a subset of its vertices such that every pair of vertices in the subset are adjacent. A clique covering of a graph is a family of cliques in the graph that cover all edges of the graph. The \txtsl{clique cover number} of a graph $G$, denoted by  $\CC(G)$, is the minimum number of cliques in the graph needed to ``cover'' all edges of $G$; i.e. each edge of $G$ lies in at least one of those cliques. Figure~\ref{M(G)=M(G-v)-1} has an example of a graph with $\CC(G)=4$. The reader may refer to \cite{MR2368647} for more facts about the clique cover number.

There are similar concepts for covering the vertices of a graph. For instance, a \txtsl{path covering} of a graph is a family of induced vertex-disjoint paths in the graph that cover all vertices of the graph. The minimum number of such paths that cover the vertices of a graph $G$ is the \txtsl{path cover number} of $G$ and is denoted by $P(G)$.  
Also a \txtsl{tree covering} of a graph is a  family of induced vertex disjoint trees in the graph that cover all vertices of the graph. The minimum number of such trees that cover the vertices of a graph $G$ is the \txtsl{tree cover number} of $G$ and is denoted by $T(G)$.

Next, we define the \txtsl{chromatic number} of  a graph. To this end, recall that a k-vertex colouring of G assigns to each vertex one colour from a set of $k$ colours; the colouring is
proper if no two adjacent vertices have the same colour. A graph $G$ is \textsl{$k$-colourable}\index{k-colourable@$k$-colourable} if $G$ has a proper $k$-vertex colouring. The chromatic number, $\chi(G)$, of  $G$ is the minimum $k$ for which $G$ is $k$-colourable. A $k$-chromatic graph $G$ is a graph with $\chi(G)=k$.

Identifying two vertices is a way to form new graphs. Let $v_1,v_2\in V(G)$. The simple graph formed by identifying $v_1$ and $v_2$ has vertex set $\left(V(G)\backslash \{v_1,v_2\}\right)\cup \{v\}$ and $N(v)=N(v_1)\cup N(v_2)$ with  all possible multiple edges changed to single edges. Similarly we can identify vertices from different graphs. If $v_1\in G$ and $v_2\in H$, then identifying the vertices $v_1$ and $v_2$ produces the \txtsl{vertex-sum} of the graphs $G$ and $H$ at vertex $v$, which is denoted by $G\stackplus{v}H$. This simple graph has $V(G)\cup \left(V(H)\backslash \{v_1,v_2\}\right)\cup \{v\}$  as its vertex set and $E(G)\cup E(H)$ as its edge set and with $N(v)=N(v_1)\cup N(v_2)$.    
 The graph formed by connecting $G$ and $H$ by adding the edge $v_1v_2$  is the \txtsl{edge-sum} of the graphs $G$ and $H$ and is denoted by $G=G\,\,\stackplus{v_1v_2}\,\,H$. This graph has $V(G)\cup V(H)$ as its vertex set and $E(G)\cup E(H)\cup\{v_1v_2\}$ as its edge set.

The graph resulting from the removal of a vertex $v$ and all edges incident with $v$ from the graph $G$ is denoted by $G-v$. Also the graph resulting from removing an edge $e$ from the graph $G$ is denoted by $G-e$. The \txtsl{subdivision} of an edge $e=uv$ in a graph $G$ yields a graph with vertex set $V(G)\cup\{w\}$ and edge set $E(G)\backslash\{uv\}\cup\{uw\}\cup\{wv\}$.

A \txtsl{cut vertex} of a graph $G$ is a vertex whose deletion disconnects $G$.  A \txtsl{cut set} of the graph $G$ is a set $V' \subseteq V(G)$ whose deletion disconnects $G$.
 The \txtsl{vertex connectivity} of the graph $G$, which is denoted by $\kappa(G)$, is the minimum size over all cut sets of $G$. 

Recall that a subgraph of a graph $G$ is a graph ``smaller'' than $G$ that is obtained from $G$ by removing some vertices and/or edges of $G$.  There are other ``smaller'' graphs, which are not necessarily  subgraphs of $G$, yet are obtained from $G$ by a series of edge and vertex removing and identifying operations. The next definition introduces an example of these ``smaller" graphs. If $e=xy$ is an edge of $G$, then the \txtsl{contraction} of $e$, which is denoted by $G/e$, is the operation that  deletes the edge $e$ and identifies $x$ and $y$ and removing all multiples edges.  The resulting graph is a simple graph with fewer number of vertices and edges. 
A \txtsl{minor} of a graph $G$ arises by performing a sequence of the following operations:
\begin{enumerate}[(a)]
\item deletion of edges;
\item deletion of isolated vertices;
\item contraction of edges.
\end{enumerate}
Note that a graph obtained from a graph $G$ by deleting some  of its vertices is also a minor of $G$; therefore any subgraph is a minor. If a graph $G$ has a graph $H$ as a minor, we say that ``$G$ has an $H$ minor''.

A graph parameter  $\lambda$ is said to be  \txtsl{minor-monotone}, if for any graph $G$ and any minor  $H$ of $G$, we have $\lambda(H)\leq \lambda(G)$, or if for any graph $G$ and any minor  $H$ of $G$, we have $\lambda(H)\geq \lambda(G)$. For example, $\kappa(G)$ is a minor-monotone graph parameter. Based on the definition, a graph has lots of different minors but the existence of certain minors in the graph can give information about its minor monotone  parameters.

A graph is said to be \txtsl{embeddable} in the plane, or \txtsl{planar}, if it can be drawn in the plane so that its edges intersect only at their incident vertices.  Such a drawing of a planar graph $G$ is called a \txtsl{planar embedding} or \txtsl{crossing-free embedding} of $G$. We refer the reader to the books~\cite{MR2159259}  and~\cite{MR1367739} for  detailed discussions about planar graphs.  A graph is planar if and only if it has neither a $K_5$ minor nor a $K_{3,3}$ minor (see \cite[Theorem 6.2.2]{MR1367739}). In a planar graph a \txtsl{face} is a region bounded by edges that does not have any edges going through them. The area outside the planar graph is also a face, called the \textsl{outer face}\index{face!outer}, and the other faces are called the \textsl{inner faces}\index{face!inner}.   Similarly a graph is \txtsl{outerplanar} if it has a planar embedding in which the outer face contains all the vertices. The edges touching the outer face are called \txtsl{outer edges} and other edges are called \txtsl{inner edges}. In the graph depicted in Figure~\ref{pendant_tree_construction}, $wu$ is an outer edge while $wv$ is an inner edge and the region bounded by edges $wu$, $wv$ and $uv$ is an inner face. A graph is outerplanar if and only if it has neither a $K_4$ minor nor a $K_{2,3}$ minor~\cite{MR2159259}.

The \txtsl{corona} of $G$ with $H$, denoted by  $G\circ H$, is the graph of order $|V(G)||V(H)|+|V(G)|$ obtained by taking one copy of $G$ and $|V(G)|$ copies of $H$, and joining all vertices in the $i$-th copy of $H$ to the  $i$-th vertex of $G$. See Figure~\ref{corona} for an illustration of $C_5\circ K_2$.  In Section~\ref{graph_with_Z=M}  we give a generalization of the corona, in keeping with the notation of this generalization, the corona of $G$ with $H$ will be denoted by $G \prec H,\ldots,H\succ$.

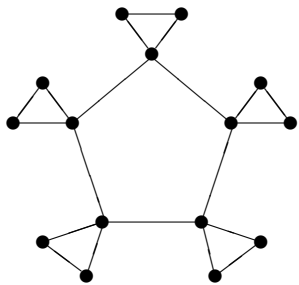
\begin{figure}[ht!]
\begin{center}
\unitlength=.75pt
\begin{picture}(85,75)
\put(0,0){\circle*{6}}
\put(-30,0){\circle*{6}}
\put(-15,20){\circle*{6}}
\put(-30,0){\line(1,0){27}}
\put(-30,0){\line(3,4){13}}
\put(-15,20){\line(3,-4){13}}

\put(80,0){\circle*{6}}
\put(110,0){\circle*{6}}
\put(95,20){\circle*{6}}
\put(82.3,2.3){\line(3,4){11.2}}
\put(97,17.5){\line(3,-4){11.2}}
\put(83,0){\line(1,0){25}}

\put(65,-50){\circle*{6}}
\put(95,-60){\circle*{6}}
\put(72,-77){\circle*{6}}
\put(66,-53){\line(1,-4){6}}
\put(68,-51){\line(3,-1){24}}
\put(72,-77){\line(4,3){20.5}}

\put(15,-50){\circle*{6}}
\put(-15,-60){\circle*{6}}
\put(7,-77){\circle*{6}}
\put(-15,-60){\line(3,1){27}}
\put(-15,-60){\line(4,-3){19.5}}
\put(8,-74){\line(1,3){7}}

\put(40,35){\circle*{6}}
\put(25,55){\circle*{6}}
\put(55,55){\circle*{6}}
\put(25,55){\line(1,0){27}}
\put(25,55){\line(3,-4){13}}
\put(55,55){\line(-3,-4){13}}

\put(0,0){\line(6,5){40}}
\put(0,0){\line(1,-3){17}}
\put(15,-50){\line(1,0){47}}
\put(65,-50){\line(1,3){16}}
\put(80,0){\line(-6,5){40}}
\end{picture}\unitlength=1pt
\\
\vspace{3cm}
\end{center}
\caption{Corona of $C_5\circ K_2$}
\label{corona}
\end{figure}


\section{Spectral graph theory}\label{matrix theory}
This section represents a brief introduction to an important part of algebraic graph theory, namely  spectral graph theory. Here we introduce some  linear algebraic tools which are used to establish some  properties of graphs. For more details on spectral graph theory, the reader may refer to \cite{MR1271140} or  \cite{MR1829620} and for more details on matrix theory to \cite{MR0276251} or \cite{MR832183}.

First we define the concept of an \txtsl{adjacency matrix}.
Assume that $X$ is a graph with vertex set $V(X)=\{v_1,\ldots,v_n\}$. Then the \txtsl{adjacency matrix} of $X$, which is denoted by $A(X)$, is a square (0,1)-matrix of size $n$, whose $(i,j)$-th entry is $1$ if and only if $v_i$ is adjacent to $v_j$. Since there are no loops in the graph, the diagonal entries of the adjacency matrix are zero. Also, changing the order of the vertices produces a distinct, but similar adjacency matrix.

Recall from linear algebra that for a matrix $A=[a_{i,j}]$, the \txtsl{transpose of a matrix} of $A$, denoted by  $A^{\top}$, is  defined to be the matrix $A^\top=[a'_{i,j}]$, where $a'_{i,j}=a_{j,i}$, for all $i$ and $j$. An $n\times n$ matrix $A$ is said to be a \txtsl{symmetric matrix} if $A^\top=A$. The set of all square $n\times n$ matrices with entries in $\mathbb{C}$ is denoted by $M_n$. The \txtsl{conjugate} of a matrix $A=[a_{ij}] \in M_n$ is denoted by $\bar{A}=[\bar{a}_{ij}]$, where $\bar{a}_{ij}$ is the conjugate of $a_{ij}$.
The \txtsl{conjugate transpose} of a matrix $A$, denoted by $A^*$, is defined to be 
\[
A^*\equiv \left(\bar{A}\right)^\top= [\bar{a}_{ji}].
\]
A matrix $A=[a_{ij}] \in M_n$ is said to be \txtsl{Hermitian} if $A=A^*$. The set of real symmetric $n\times n$ matrices is denoted by $\mathcal{S}_n$ and the set of Hermitian $n\times n$ matrices is denoted by $H_n$. 

The \txtsl{null space} (or the \textsl{kernel}) of a matrix $A$ is the vector space consisting of all vectors $v$ for which $Av=0$; we denote the null space of $A$ by $\ker(A)$. The matrix $A$ is said to be \txtsl{nonsingular} (or \txtsl{invertible}) if there is no nonzero vector in its null space, otherwise, it is \txtsl{singular}. The dimension of the null space of a matrix $A$ is called the \txtsl{nullity} of $A$ and is denoted by $\nul(A)$. The \txtsl{rank} of a matrix is defined to be the dimension of the row or column space of the matrix; that is, the dimension of the vector space spanned by the rows or columns of the matrix. The rank of $A$ is denoted by $\rank(A)$. The following fact is well-known.

\begin{prop}[{{see \cite{MR0276251}}}]\label{rank+null=n}
For an $n\times n$ matrix $A$,
\[
\rank(A)+\nul(A)=n.\qed
\]
\end{prop}

For a square matrix $A$, the \txtsl{eigenvalues} are defined to be the roots $\lambda$ of the \txtsl{characteristic polynomial} $\phi(A;x)=$det$(xI-A)$. Equivalently, a complex number $\lambda$ is an eigenvalue of $A$, if the determinant of the matrix $\lambda I-A$ is zero. This is, in turn, equivalent to the fact that there is a non-zero vector in the null space of $\lambda I-A$.  If $\lambda$ is an eigenvalue of $A$, then the null space of $\lambda I -A$ is  called the \txtsl{eigenspace} of $A$ corresponding to the eigenvalue $\lambda$, and its non-zero elements are called the \txtsl{eigenvector}s of $A$ corresponding to  $\lambda$. The dimension of the eigenspace corresponding to the eigenvalue $\lambda$ is called the \txtsl {geometric multiplicity} of the eigenvalue $\lambda$. The \txtsl{algebraic multiplicity}  of the eigenvalue $\lambda$ of a matrix $A$, is the maximum power of the factor $x-\lambda$ in the characteristic polynomial $\phi(A;x)$ and is denoted by $m(\lambda)$.

\begin{prop}
If a square matrix $A$ is symmetric, then all the eigenvalues of $A$  are real and  the \txtsl {algebraic multiplicity} of each eigenvalue is equal to its \txtsl{geometric multiplicity}.\qed
\end{prop}
 Note that in this thesis, we call the the eigenvalues of the  adjacency matrix of the graph the eigenvalues of the graph. So we have the following simple corollary of the previous results. 
\begin{cor}\label{evalues are real}
All the eigenvalues of a graph are real.\qed
\end{cor}

 Therefore, we usually order the eigenvalues of a graph on $n$ vertices as $\lambda_n\geq \lambda_{n-1}\geq \cdots\geq \lambda_1$.  The \txtsl{spectrum} of a graph $X$ is the set of its eigenvalues with their multiplicities. It is usually written as the array:
\[
\Spec(X) =
\left( {\begin{array}{cccc}
 \lambda_1 & \lambda_2 & \cdots & \lambda_s   \\
 m(\lambda_1) & m(\lambda_2) & \cdots & m(\lambda_s)\\
 \end{array} } \right),
\]
\medskip
where $\lambda_s>\cdots>\lambda_1$ are the distinct eigenvalues of $X$. 
For example, the spectra of the complete graph $K_n$ and the complete bipartite graph $K_{m,n}$ are as follows:
\[
\Spec(K_n) =
\left( {\begin{array}{cccc}
 -1 & n-1   \\
 n-1 & 1\\
 \end{array} } \right),
\]
\[
\Spec(K_{m,n}) =
\left( {\begin{array}{cccc}
 -\sqrt{mn}& 0 & \sqrt{mn}   \\
 1 & m+n-2& 1\\
 \end{array} } \right).
\]
A symmetric matrix $A$ is \txtsl{positive semi-definite} if all the eigenvalues of $A$ are non-negative. For example $A(K_n)$ is not a positive semi-definite matrix while $A(K_n)+I$ is. The matrix $A$ is said to be a \txtsl{tridiagonal matrix} if it has non-zero entries only on the diagonal, the super-diagonal and the sub-diagonal. For example $A(P_n)$ is a tridiagonal matrix.  The following is another example of a  tridiagonal matrix :
\[
A =
\left( {\begin{array}{cccc}
 4 & -1 & 0 & 0   \\
 -1 & 3 & 2 & 0  \\
 0 & 2 & 1 & 1 \\
 0 & 0 & 1 & -5
 \end{array} } \right).
\]
\begin{lem}\label{A^i of a path} Let $A$ be an $n\times n$ tridiagonal matrix in which all the entries of the super-diagonal and sub-diagonal are non-zero. Then the $(1,n)$-entry of $A^r$ is zero, for any integer  $0\leq r\leq n-2$, and non-zero for $r=n-1$.
\end{lem}
\begin{proof}
We assume $n\geq 3$ as the lemma is trivial for $n=1,2$. First we claim that for any $0\leq r\leq n-1$ and for any $1\leq i< n-r$, the $(i,n)$-entry of $A^r$ is zero. The proof is by induction on $r$. If $r=0$ or $1$, the claim is, obviously, true. Thus, assume $2\leq r\leq n-1$. Then for any $1\leq i<n-r$, since $A$ is tridiagonal, we have
\begin{align*}
(A^r)_{i,n} &= (A\, A^{r-1})_{i,n} = \sum_{k=1}^n A_{i,k} (A^{r-1})_{k,n}\\
            &= A_{i,i-1}  (A^{r-1})_{i-1,n} + A_{i,i}  (A^{r-1})_{i,n} + A_{i,i+1}  (A^{r-1})_{i+1,n},
\end{align*}
where  $A_{i,i-1}$ is considered to be zero if $i=1$. Then, using the induction hypothesis, we have $(A^{r-1})_{i-1,n} = (A^{r-1})_{i,n} =  (A^{r-1})_{i+1,n}=0$; thus the claim is proved. Now it follows from the claim that, for any $0\leq r\leq n-2$, the $(1,n)$-entry of $A^r$ is zero. To complete the proof of lemma, it is enough to note that the $(1,n)$-entry of $A^{n-1}$ is 
\[
A_{1,2} \,A_{2,3}\,\cdots \,A_{n-1,n}\neq 0.\qedhere
\]
\end{proof}


\chapter{The Minimum Rank of Graphs}\label{min_rank}
In this chapter we introduce the  ``minimum rank'' and the ``maximum nullity'' of a graph. We will also consider the positive semi-definite minimum rank and  the positive semi-definite maximum nullity of graphs. The problem of evaluating or approximating these parameters is of great importance in this thesis. In the future chapters we will approach  these problems using some  combinatorial tools.


\section{Definition}\label{min_rank_def}
In this chapter we introduce the concept of minimum rank of graphs and related concepts, and provide some examples.
To a given graph $G$ with vertex set $\{1, \ldots, n\}$, we associate a class of real, symmetric matrices as follows.
\[
\mathcal{S}(G)=\{ A=[a_{ij}]  \,\,|\,\, A\in \mathcal{S}_n,\,\text{ for } \,\,i \neq j, \,\,\, a_{ij} \neq 0 \iff  ij \in E(G)\},
\]
Note that there is no restriction on the value of  $a_{ii}$, with $i=1,2,\dots,n$ and  the adjacency matrix $A(G)$ belongs to $\mathcal{S}(G)$.

On the other hand,  the \textsl{graph}\index{graph of  a matrix} of an $n\times n$ Hermitian matrix $A$, denoted by $\mathcal{G}(A)$, is the graph with vertices $\{1,\ldots,n\}$ and the edge set
\[
\{ij\,\,|\, a_{ij}\neq 0,\,\,1\leq i \neq j\leq n\}.
\]
The \txtsl{minimum rank} of $G$ is defined to be
\[
\mr(G)=\text{min} \{\rank(A) \,\,|\,\, A \in \mathcal{S}(G)\},
\]
while the \txtsl{maximum nullity} of $G$ is defined as
\[
M(G)=\text{max} \{\nul(A) \,\,|\,\, A \in \mathcal{S}(G)\}.
\]
By Proposition~\ref{rank+null=n} we have
\begin{equation}\label{mr&M}
\mr(G)+M(G)=|V(G)|.
\end{equation}

Note that for any non-empty graph each matrix in the class of all symmetric matrices associated with a non-empty graph is non-zero, so the minimum rank of any graph is bounded below by $1$. Also, as the diagonal entries of the matrices are free to be chosen, they can always be adjusted  in such a way that the matrix admits a zero eigenvalue; this implies that $\mr(G)$ is bounded above by $|V(G)|-1$. Thus we have the following lemma.
\begin{lem}
For any non-empty graph $G$
\[
 1\leq \mr(G)\leq |V(G)|-1.
 \]
 \end{lem}
  \begin{example}
Consider the complete graph $K_3$. Then, for example
\[
\left(
\begin{array}{ccc}
  0 & -3 & 1   \\
 -3 & 0 & 7   \\
  1 & 7 & -4   \\
\end{array}
\right) \in \mathcal{S}(K_3);
\]
and also
\[
\left(
\begin{array}{cccc}
1  &1   &1    \\
1 & 1  & 1    \\
1 & 1  & 1    \\
\end{array}
\right) \in \mathcal{S}(K_3);
\]
Thus since there is a matrix of rank $1$ in $\mathcal{S}(K_3)$, $\mr(K_3)=1$.
\end{example}
\begin{example}
As another illustrative example, consider any matrix $A$ associated with a path  $v_1 v_2 \cdots v_n$. Up to labelling of the vertices, $A$ is a symmetric tridiagonal matrix with nonzero sub-diagonal and super-diagonal. From Lemma~\ref{A^i of a path} the $(1,n)$-entry of $A^i$ is zero, for any $0\leq i\leq n-2$, and non-zero for $i=n-1$. We deduce that the set $\{I=A^0, A^1, A^2,\dots, A^{n-1} \}$ is linearly independent. Thus the minimal polynomial of $A$ must be of degree at least $n$. This implies that $A$ must have $n$ distinct eigenvalues. This, in turn, implies that the number of non-zero eigenvalues of $A$ is at least $n-1$; that is $\mr(G)\geq n-1$  and, therefore, $\mr(G)=n-1$.  In fact a path on $n$ vertices is the only graph on $n$ vertices whose minimum rank is equal to $n-1$ (see \cite[Theorem 1.4]{MR2350678} and \cite{fiedler1969characterization}). 
\end{example}

\section{Basic facts about the minimum rank}\label{min_rank_basics} 
In this section we provide some basic results about the minimum rank of graphs. The set of real positive semidefinite matrices described by a graph $G$ and the set of Hermitian positive semidefinite matrices described by $G$ are
\[
\mathcal{S}_+(G)=\{A\in \mathcal{S}_n\,\,:\,\, \mathcal{G}(A)=G\,\, \text{and}\,\, A\,\, \text{is positive semidefinite} \}
\]
and
\[
\mathcal{H}_+(G)=\{A\in H_n\,\,:\,\, \mathcal{G}(A)=G\,\, \text{and}\,\, A\,\, \text{is positive semidefinite} \},
\]
respectively. The \textsl{minimum positive semidefinite rank}\index{minimum rank!positive semidefinite} of $G$ and \textsl{minimum Hermitian positive semidefinite rank}\index{minimum rank!Hermitian positive semidefinite} of $G$ are
\[
\mr^\mathbb{R}_+(G)=\text{min} \{\rank(A) \,\,|\,\, A \in \mathcal{S}_+(G)\} 
\]
and
\[
\mr^\mathbb{C}_+(G)=\text{min} \{\rank(A) \,\,|\,\, A \in \mathcal{H}_+(G)\},
\]
respectively. The \textsl{maximum positive semidefinite nullity}\index{maximum nullity!positive semidefinite} of $G$ and \textsl{maximum Hermitian positive semidefinite nullity}\index{maximum nullity!Hermitian positive semidefinite} of $G$ are
\[
M^\mathbb{R}_+(G)=\text{max} \{\nul(A) \,\,|\,\, A \in \mathcal{S}_+(G)\} 
\]
and 
\[
M^\mathbb{C}_+(G)=\text{max} \{\nul(A) \,\,|\,\, A \in \mathcal{H}_+(G)\},
\]
respectively. 

Clearly $\mr^\mathbb{R}_+(G)+M^\mathbb{R}_+(G)=|V(G)|$ and $\mr^\mathbb{C}_+(G)+M^\mathbb{C}_+(G)=|V(G)|$.

The following two observations are well known and straightforward to prove.
\begin{obs}\label{M of disconnected graphs}
If the connected components of the graph $G$ are $G_1,\ldots, G_t$, then
\[
\mr(G)=\sum^t_{i=1}\mr(G_i)\,\,\,\,\, \text{and}\,\,\,\,\,M(G)=\sum^t_{i=1}M(G_i).\qed
\]
\end{obs}

The following fact has been shown in \cite{MR1416462}. 
\begin{prop}\label{mr-spread}
For any vertex $v$ of $G$, $0\leq \mr(G)-\mr(G-v)\leq 2$. \qed
\end{prop}

Clearly if $A'$ is a submatrix of $A$ then $\rank(A')\leq \rank(A)$. Based on this fact we have the following corollary. 
\begin{cor}
If $G'$ is an induced subgraph of $G$ then 
\[
\mr(G')\leq \mr(G).\qed
\]
\end{cor}

Note that a similar result does not hold for maximum nullity. To show this note that $M(C_n)=2$ and $P_{n-1}$ is an induced subgraph of it with  $M(P_{n-1})=1$. Also the graph shown in Figure~\ref{M(G)=M(G-v)-1} has $M(G)=2$ while $M(G-v)=3$.

Since a connected graph $G$ contains an induced path on $\diam(G)+1$ vertices we observe the following.
\begin{obs}
For a connected graph $G$ we have $\diam(G)\leq \mr(G)$.\qed
\end{obs}

The \textsl{union}\index{union of graphs} of the graphs $G_i=(V_i,E_i)$, $i=1,\ldots, h$, is defined to be 
\[
\bigcup^h_{i=1} G_i=\left(\bigcup^h_{i=1}V_i\,,\, \bigcup^h_{i=1}E_i\right).
\]
If $G=\bigcup^h_{i=1} G_i$, a matrix $A$ of rank at most $\sum^h_{i=1} \mr(G_i)$ having $\mathcal{G}(A)=G$ can be obtained by choosing (for each $i=1, \ldots, h$) a matrix $A_i$ that realizes $\mr(G_i)$, embedding $A_i$ in a matrix $\tilde{A}_i$ of size $|V(G)|$, choosing $a_i\in \mathbb{R}$ such that no cancelation of nonzero entries occurs, and letting $A=\sum^h_{i=1}a_i\tilde{A}_i$. Thus we have the following.
\begin{obs}\label{union}
If $G=\bigcup^h_{i=1}G_i$, then $\mr(G)\leq \sum^h_{i=1}\mr(G_i)$.\qed
\end{obs}  

The inequality in Observation~\ref{union} can be strict. For example $\mr(K_3)=1$ and $K_3$ can be considered as the union of $3$ copies of $K_2$ with $\mr(K_2)=1$. Since $1<1+1+1$, the inequality in this example is strict.   

If in the definition of the union of graphs we have $V_i\cup V_j=\emptyset$ for $i\neq j$, then we have the \txtsl{disjoint union} of graphs. Under these assumptions the matrices $\tilde{A}_i$, for $1\leq i\leq h$, don't overlap on the position of nonzero entries. Thus in this case we have the same result for the minimum rank of disjoint union of graphs as Observation~\ref{M of disconnected graphs}, i.e. $\mr(\bigcup^h_{i=1}G_i)= \sum^h_{i=1}\mr(G_i)$. 

The following well-known fact is a special case of Observation~\ref{union}.
\begin{cor}\label{mr&cc}
If $G$ is a graph, $\mr(G)\leq \CC(G)$. \qed
\end{cor} 
 The problem of the minimum rank of trees has  been completely solved (see \cite{MR1712856}).
\begin{thm}
$M(T)=P(T)=|V(T)|-\mr(T)$.\qed
\end{thm}

The \txtsl{rank-spread} of a graph $G$ at  a vertex $v$, denoted by $r_v(G)$, is the amount of change in the minimum rank of $G$ when $v$ is deleted. More precisely 
\[
r_v(G) = \mr(G)-\mr(G-v).
\]
From Proposition~\ref{mr-spread}, for any vertex $v$ of $G$, we have $0\leq r_v(G) \leq 2$. It has been shown in \cite{MR2702337, MR2095918} that if a graph has a cut-vertex, then the problem of finding the minimum rank of the graph can be reduced to the problem of determining the minimum rank of several graphs of smaller orders. 
 
\begin{thm}[{{Cut-vertex reduction}}]\label{reduction} Let $v$ be a cut-vertex of a graph $G$. For $i =1,\ldots, h$, let $W_i\subseteq V (G)$ be the vertices of the $i$-th component of $G\backslash v$ and let $G_i$ be the subgraph of $G$ induced by $\{v\} \cup W_i$. Then
\[
r_v(G)=\text{min}\,\{ \sum_{i=1}^h r_v(G_i),\,2 \}
\]
and hence
\[
\mr(G)=\sum_{i=1}^h\mr(G_i-v)+\text{min}\,\{ \sum_{i=1}^h r_v(G_i),\,2 \}.\qed
\]
\end{thm}

\begin{example}
Consider the graph $K_3$ with a vertex labeled $v$. The vertex $v$ is a cut-vertex of the graph $G=K_3\stackplus{v}K_3$. Using Theorem~\ref{reduction}, we calculate the minimum rank of the graph $G$.
 The removal of the vertex $v$ from $G$ results in two components $G_1-v=G_2-v=K_2$ with $G_1=G_2=K_3$.  We have $r_v(G_1)=r_v(G_2)=0$ therefore, $r_v(G)=\text{min}\,\{ \sum_{i=1}^h r_v(G_i),\,2 \}=0$. Also $\mr(G_1-v)=\mr(G_2-v)=1$ thus $\mr(G)=1+1=2$.
\end{example}

\chapter{Zero Forcing Sets}\label{ZFS}

In this chapter we will introduce a new type of graph colouring which defines a graph parameter called the zero forcing number,  denoted by $Z(G)$, which is the minimum size of a zero forcing set. This parameter was first introduced and defined at the workshop ``Spectra of Families of Matrices described by Graphs, Digraphs, and Sign Patterns'', which was held at the American Institute of Mathematics on October, 2006 (see \cite{MR2388646}). Also, in that workshop it was shown that  $Z(G)$ is an upper bound for $M(G)$. Somewhat surprisingly, $M(G)=Z(G)$, for most of the graphs for which $M(G)$ is known. For instance, these two parameters are equal for all graphs with fewer than seven vertices and for some families of chordal graphs. We will establish this equality for more families of graphs and we show the equality $Z(G)=P(G)$ in some families of graphs as well.  In addition, we will establish a relationship between $Z(G)$ and the chromatic number $\chi(G)$ of a graph $G$.

\section{Definition}\label{ZFS_basics}
 This section includes basic definitions and facts about the zero forcing number of  a graph $G$, which will be used throughout this thesis. The zero forcing number is a graph parameter that arises from a type of graph colouring, therefore, first we turn our attention to the rules of this new graph colouring. This graph colouring is based on a colour-change rule that describes how to colour the vertices of the graph.

Let $G$ be a graph with all vertices initially coloured either black or white. If $u$ is a black vertex of $G$ and $u$ has exactly one neighbour that is white, say $v$, then we change the colour of $v$ to black; this rule is called the \txtsl{colour change rule}. In this case we say ``$u$ forces $v$'' which is denoted by {$u\rightarrow v$}. The procedure of colouring a graph using the colour rule is called the zero forcing process (shortly the forcing process). Note that each vertex will force at most one other vertex. 
Given an initial colouring of $G$, the \txtsl{derived set} is the set of all black vertices resulting from repeatedly applying the colour-change rule until no more changes are possible.
A \txtsl{zero forcing set} $Z$, is a subset of vertices of $G$ such that if initially the vertices in $Z$ are coloured black and the remaining vertices are coloured white, then the derived set of $G$ is $V(G)$. The \txtsl{zero forcing number} of a graph $G$, denoted by $Z(G)$, is the smallest size of a zero forcing set of $G$.
We abbreviate the term zero forcing set as ZFS. A  zero forcing process is called \textsl{minimal}\index{minimal zero forcing process} if the initial set of black vertices is a minimal ZFS.  
Note that for any nonempty graph $G$
\[
1\leq Z(G) \leq |V(G)|-1.
\]
\section{Examples and basic results}
This section presents examples of zero forcing sets of several graphs in order to illustrate this concept.  As an example, each endpoint of a path is a zero forcing set for the path, thus $Z(P_n)=1$ (note that no other vertex of the path is a zero forcing set). In a cycle, any set of two adjacent vertices is a zero forcing set and there is no ZFS of size one for the cycle, thus $Z(C_n)=2$. It is, also, easy to see that $Z(K_n)=n-1$.
\begin{example}\label{forcing-figures}
Let $G$ be the graph in Figure~\ref{start_graph_for_zfs}. Then, Figure~\ref{finding_a_zfs} illustrates why the set $Z=\{v_1,v_2\}$ is a ZFS for $G$.

\begin{figure}[ht!]
\begin{center}
\unitlength=1pt
\begin{picture}(100,25)
\multiput(0,0)(50,0){3}{\circle{6}}
\multiput(0,-50)(50,0){2}{\circle{6}}

\put(0,-3){\line(0,-1){44}}
\put(3,-50){\line(1,0){44}}
\put(2,-48){\line(1,1){45.5}}
\put(50,-47){\line(0,1){44}}
\put(53,0){\line(1,0){44}}

\put(0,8){\makebox(0,0){$v_1$}}
\put(50,8){\makebox(0,0){$v_2$}}
\put(100,8){\makebox(0,0){$v_5$}}
\put(0,-58){\makebox(0,0){$v_3$}}
\put(50,-58){\makebox(0,0){$v_4$}}
\end{picture}
\vspace{2cm}
\end{center}
\caption{The graph $G$ for which we want to find a ZFS}
\label{start_graph_for_zfs}
\end{figure}
\vspace{.5cm}

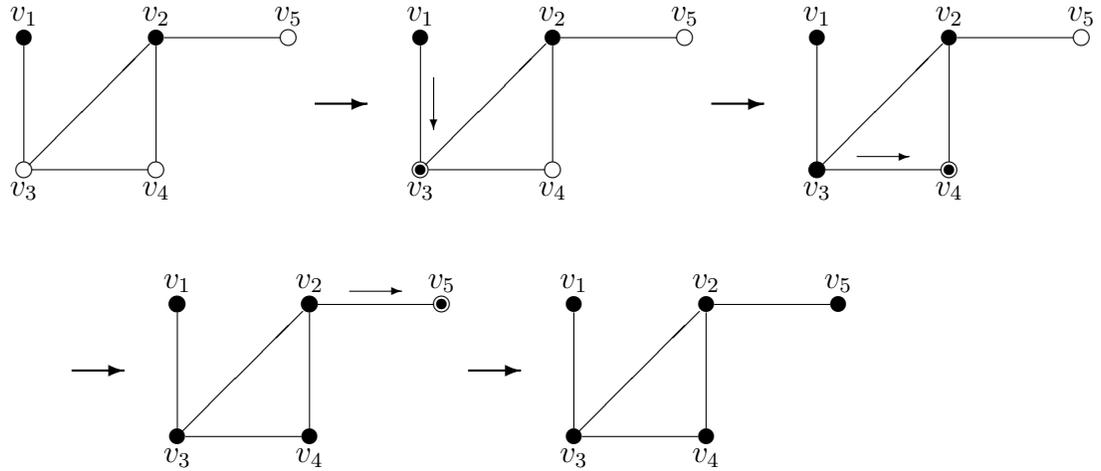
\begin{figure}[ht!]
\begin{center}
\unitlength=1pt
\begin{picture}(360,30)
\multiput(0,0)(50,0){2}{\circle*{6}}
\put(100,0){\circle{6}}
\multiput(0,-50)(50,0){2}{\circle{6}}

\put(0,-3){\line(0,-1){44}}
\put(3,-50){\line(1,0){44}}
\put(2,-48){\line(1,1){45.5}}
\put(50,-47){\line(0,1){44}}
\put(53,0){\line(1,0){44}}

\put(0,8){\makebox(0,0){$v_1$}}
\put(50,8){\makebox(0,0){$v_2$}}
\put(100,8){\makebox(0,0){$v_5$}}
\put(0,-58){\makebox(0,0){$v_3$}}
\put(50,-58){\makebox(0,0){$v_4$}}

\thicklines
\put(110,-25){\vector(1,0){20}}
\thinlines
\unitlength=1pt

\multiput(150,0)(50,0){2}{\circle*{6}}
\put(250,0){\circle{6}}
\multiput(150,-50)(50,0){2}{\circle{6}}
\put(150,-50){\circle*{4}}

\put(150,-3){\line(0,-1){44}}
\put(153,-50){\line(1,0){44}}
\put(152,-48){\line(1,1){45.5}}
\put(200,-47){\line(0,1){44}}
\put(203,0){\line(1,0){44}}

\put(155,-15){\vector(0,-1){20}}

\put(150,8){\makebox(0,0){$v_1$}}
\put(200,8){\makebox(0,0){$v_2$}}
\put(250,8){\makebox(0,0){$v_5$}}
\put(150,-58){\makebox(0,0){$v_3$}}
\put(200,-58){\makebox(0,0){$v_4$}}

\thicklines
\put(260,-25){\vector(1,0){20}}
\thinlines

\multiput(300,0)(50,0){2}{\circle*{6}}
\put(400,0){\circle{6}}
\multiput(300,-50)(50,0){2}{\circle{6}}
\put(300,-50){\circle*{6}}
\put(350,-50){\circle*{4}}

\put(300,-3){\line(0,-1){44}}
\put(303,-50){\line(1,0){44}}
\put(302,-48){\line(1,1){45.5}}
\put(350,-47){\line(0,1){44}}
\put(353,0){\line(1,0){44}}

\put(315,-45){\vector(1,0){20}}

\put(300,8){\makebox(0,0){$v_1$}}
\put(350,8){\makebox(0,0){$v_2$}}
\put(400,8){\makebox(0,0){$v_5$}}
\put(300,-58){\makebox(0,0){$v_3$}}
\put(350,-58){\makebox(0,0){$v_4$}}
\end{picture}


\begin{picture}(240,100)
\thicklines
\put(-40,-25){\vector(1,0){20}}
\thinlines

\multiput(0,0)(50,0){3}{\circle{6}}
\multiput(0,-50)(50,0){2}{\circle*{6}}
\multiput(0,0)(50,0){2}{\circle*{6}}
\put(100,0){\circle*{4}}

\put(0,-3){\line(0,-1){44}}
\put(3,-50){\line(1,0){44}}
\put(2,-48){\line(1,1){45.5}}
\put(50,-47){\line(0,1){44}}
\put(53,0){\line(1,0){44}}

\put(65,5){\vector(1,0){20}}

\put(0,8){\makebox(0,0){$v_1$}}
\put(50,8){\makebox(0,0){$v_2$}}
\put(100,8){\makebox(0,0){$v_5$}}
\put(0,-58){\makebox(0,0){$v_3$}}
\put(50,-58){\makebox(0,0){$v_4$}}

\thicklines
\put(110,-25){\vector(1,0){20}}
\thinlines
\multiput(150,0)(50,0){3}{\circle*{6}}
\multiput(150,-50)(50,0){2}{\circle*{6}}

\put(150,-3){\line(0,-1){44}}
\put(153,-50){\line(1,0){44}}
\put(152,-48){\line(1,1){45.5}}
\put(200,-47){\line(0,1){44}}
\put(203,0){\line(1,0){44}}

\put(150,8){\makebox(0,0){$v_1$}}
\put(200,8){\makebox(0,0){$v_2$}}
\put(250,8){\makebox(0,0){$v_5$}}
\put(150,-58){\makebox(0,0){$v_3$}}
\put(200,-58){\makebox(0,0){$v_4$}}
\end{picture}
\vspace{2.5cm}
\end{center}
\caption{Finding a zero forcing set}
\label{finding_a_zfs}
\end{figure}
\end{example}
 Any three pairwise adjacent vertices of a wheel form a ZFS for it. Any set of pendant vertices of size $n-1$ of a star ($K_{1,n-1}$) is a ZFS for it. As another example we have the following proposition. 
 
\begin{prop}\label{Z of complete multipartite graph}
 Let $K_{n_1, \dots ,n_m}$ be a complete multipartite graph with at least one $n_i> 1$ for $1\leq i\leq m$, 
then
\[
Z(K_{n_1, \dots ,n_m}) = (n_1+n_2+\dots+n_m)-2.
\]
\end{prop}
\begin{proof}
Let $G=K_{n_1, \dots ,n_m}$ and $Z$ be the set of all vertices of the graph except for two vertices say, $u$ and $v$, which lie in different parts, $A$ and $B$. Since the graph is a complete multipartite graph, any vertex in $A$ can force $v$ and any vertex in $B$ can force $u$. Thus Z is a ZFS for $G$ and
\[
Z(G) \leq |G|-2.
\]
Next we show that $Z(G) \geq |G|-2$. Suppose that $Z(G)<|G|-2$, therefore we have a set of size at most  $|G|-3$ as the set of initial black vertices and this means we initially have at least 3 white vertices, which leads to the following three different cases:
\begin{enumerate}[(a)]
\item All of them are in the same part.
\item They are all in different parts.
\item Exactly two of them are in the same part.
\end{enumerate}
In the first and second cases there is no black vertex with a single white neighbour in the graph thus no vertex can perform a force. In the third case neither of the two white vertices that are in the same part can be turned black since any existing black vertex is either adjacent to both of them or adjacent to neither of them. Therefore the initial set of black vertices can not be a zero forcing set for the graph. Thus $Z(G) \geq|G|-2$.
\end{proof}
\begin{obs}
For every graph $G$ we have
\[
 Z(G)\geq \delta(G).\qed
 \]
\end{obs}
This bound is tight when $G$ is a complete graph. But for trees with a large path cover number, such as stars,  this  can be a very bad bound (see Proposition~\ref{P(T)=Z(T)}).

Since the neighbours of the first black vertex which is performing a force in a zero forcing process are all black, it is easy to observe the following.
The $i$-th {\it level}\index{levels of a forcing process} of a forcing process is the colouring of $G$ after applying the colour-change rule in a zero forcing process $i$ times. As noted in \cite{MR2388646}, since any vertex that turns black under one sequence of application of the colour-change rule can always be turned black, regardless of the order of colour changes, we have the following.
\begin{prop}\label{uniqueness}
Let $G$ be a graph. The derived set of any forcing process in the graph $G$, starting with a specific initial set of black vertices, is unique.
\end{prop}

\section{Connection to the maximum nullity}
In this section we show how the zero forcing number of a graph bounds the maximum nullity of the graph. To see this we need some additional definitions and theorems from linear algebra. The \txtsl{support} of a vector $\bfx=(x_i)$, denoted by $\supp(\bfx)$, is the set of indices  $i$ such that $x_i\neq0$. The following proposition relates the support of the vectors in the null space of a matrix to its nullity, (from \cite[Proposition 2.2]{MR2388646}).

\begin{prop}\label{vanishing_at_k_positions}
Let $A$ be an $n\times n$ matrix and suppose $\nul(A)>k$. Then there is a nonzero vector $\bfx\in \nul(A)$ vanishing at any $k$ specified positions.\qed 
\end{prop}
In other words, if $W$ is a set of $k$ indices, then there is a nonzero vector $\bfx \in \nul(A)$ such that $\supp(\bfx)\cap \it W=\emptyset$.

\begin{prop}[{{see \cite[Proposition 2.3]{MR2388646}}}]\label{support}
 Let $Z$ be a zero forcing set of $G=(V,E)$ and $A\in \mathcal{S}(G)$. If $\bfx \in\nul(A)$ and $\supp(\bfx)\cap Z=\emptyset$, then $\bfx=0$.
\end{prop}
\begin{proof}
Assume $\bf x \in\nul(A)$ and $\supp(\bfx)\cap Z=\emptyset$. If $Z=V$, the statement clearly holds, so suppose $Z\neq V$. If $v\in Z$ then $x_v=0$. Since $Z$ is a zero forcing set we must be able to perform a colour change. That is, there exists a vertex $u$ coloured black ($x_u$ is required to be $0$) with exactly one neighbour $v$ coloured white (so $x_v$ is not yet required to be $0$). Upon examination, the equation $(A\bfx)_u=0$ reduces to $a_{uv}x_v=0$, which implies that $x_v=0$. Similarly each colour change corresponds to requiring another entry in $\bfx$ to be zero. Thus $\bfx=0$.
\end{proof}

\begin{thm}[{{see \cite[Proposition 2.4]{MR2388646}}}]\label{M(G)<=Z(G)}
If $G$ is a graph, then $M(G)\leq Z(G)$.
\end{thm}
\begin{proof}
Let $Z$ be a ZFS for $G$. Assume $M(G)> |Z|$, and let $A\in \mathcal{S}(G)$ with $\nul(A) > |Z|$. By Proposition~\ref{vanishing_at_k_positions}, there is a nonzero vector ${\bfx} \in  \ker(A)$ that vanishes on all vertices in $Z$. By Proposition~\ref{support}, $\bfx=0$, which is a contradiction. 
\end{proof}
Note that Theorem~\ref{M(G)<=Z(G)} displays a nice relationship between a linear algebraic quantity, $M(G)$, and a purely graph theoretical parameter, $Z(G)$. Note also that the inequality in this theorem can be tight or strict. For example, for paths, it holds with equality while the corona  $C_5 \prec K_1,\ldots,K_1\succ$ (also called the penta-sun) has a zero forcing number equal to three (this follows from the fact that is proven in Corollary~\ref{unicycles}) but the maximum nullity of the penta-sun is equal to two, see \cite[Example 4.1]{MR2388646}.

\section{Zero forcing chains}
In this section we study the \textsl{zero forcing chains} produced by a zero forcing process in a graph and how this concept relates the zero forcing number of a graph to the path cover number of the graph.  

Let $Z$ be a zero forcing set of a graph $G$. Construct the derived set, making a list of the forces in the order in which they are performed. This list is called the \txtsl{chronological list of forces}.
A \txtsl{forcing chain} (for a particular chronological list of forces) is a sequence of vertices $(v_1,v_2,\ldots,v_k)$ such that $v_i\rightarrow v_{i+1}$, for $i=1,\ldots,k-1$.   Not that a minimal zero forcing process produces a minimal collection of forcing chains. 
For the graph $G$ in Example~\ref{forcing-figures}, we have the forcing chains $(v_1,v_3,v_4)$ and $(v_2,v_5)$.
 A \txtsl{maximal forcing chain} is a forcing chain that is not a proper subsequence of another zero forcing chain (the previous example has two maximal forcing chains). Note that a zero forcing chain can consist of a single vertex $(v_1)$ and such a chain is maximal if $v_1\in Z$ and $v_1$ does not perform a force. In each step of a forcing process, each vertex can force at most one other vertex and can be forced by at most one other vertex, therefore the maximal forcing chains are disjoint. Thus the vertices of the zero forcing set partition the vertices of the graph into disjoint paths (this will be discussed in Proposition~\ref{P(G)<=Z(G)}). As we showed in Theorem~\ref{uniqueness}, the derived set of a given set of black vertices is unique; however, a chronological list of forces and the forcing chains of a particular zero forcing set usually is not. 

 The number of chains in a zero forcing process starting with a zero forcing set $Z$ is equal to the size of $Z$ and the elements of $Z$ are the initial vertices of the forcing chains. Let $Z$ be a zero forcing set of a graph $G$. A \txtsl{reversal} of $Z$ is the set of last vertices of the maximal zero forcing chains of a chronological list of forces. Thus the cardinality of a reversal of $Z$ is the same as the cardinality of $Z$.
\begin{thm}[{{see \cite[Theorem 2.6]{MR2645093}}}]\label{Reversal is a ZFS}
If $Z$ is a zero forcing set of $G$, then so is any reversal of $Z$.
\end{thm}

\begin{proof}
We prove this by induction on the number of vertices. For the base case consider the graph $K_2$. Either of the vertices of the graph is a reversal for the other one and also is a ZFS for the graph. Let $G$ be a graph and assume that this is true for all graphs $G'$  with $|V(G')|<|V(G)|$. Now in graph $G$ write the chronological list of forces in reverse order. Reversing each force, this is the reverse chronological list of forces. Let the reversal of $Z$ for this list be denoted by $W$. We show the reverse chronological list of forces is a valid list of forces for $W$. Consider the first force, $u\rightarrow v$ on the reverse chronological list. We need to show that all neighbours of $u$ except $v$ must be in $W$, so that $u$ can force $v$. Since then the last force  in the original chronological list of forces was $v\rightarrow u$, each of the neighbours of $u$ had $u$ as a white neighbour and thus could not have forced any vertex previously (in the original chronological list of forces). Then $\left(Z\backslash \{u\}\right)\cup \{v\}$ is a ZFS for the graph $G-u$. Then according to the induction hypothesis, the rest of forces in  the reverse chronological list of forces is a valid list of forces for the graph $G-u$. Thus $W$ is a zero forcing set of $G$.
\end{proof}
Since we can always reverse a ZFS, every connected graph (except $K_1$) has multiple ZFS. Therefore, we have the following.
\begin{cor}[see{{\cite[Corollary 2.7]{MR2645093}}}]
No connected graph of order greater than one has a unique minimum zero forcing set.\qed
\end{cor}

The following theorem shows that there is no connected graph $G$ with a vertex $v\in V(G)$ such that $v$ is in every minimum zero forcing set. To see a proof of it refer to \cite[Theorem 2.9]{MR2645093}. Let ZFS(G) be the set of all minimum zero forcing sets of $G$.
\begin{thm}
If $G$ is a connected graph of order greater than one, then
\[
\bigcap_{Z\in ZFS(G)}Z=\emptyset. \qed
\]
\end{thm}
The next proposition shows that the path cover number is a lower bound for the zero forcing number (see \cite[Proposition 2.10]{MR2645093}).

\begin{prop}\label{P(G)<=Z(G)}
For any graph $G$, $P(G)\leq Z(G)$.\qed
\end{prop}

The most famous family of graphs for which the path cover number agrees with the zero forcing number is trees (see \cite[Proposition 4.2]{MR2388646}). 
Two forcing chains $P_1$ and $P_2$ are called \textsl{adjacent}\index{adjacent forcing chains} if there are two vertices $v\in P_1$ and $u\in P_2$ such that $uv\in E(G)$.

\begin{prop}\label{P(T)=Z(T)}
For any tree $T$, $P(T)=Z(T)$. Moreover, any minimal path covering of a tree $\PP(T)$ coincides with a collection of forcing chains with $|\PP(T)|=Z (G)$ and  the set consisting of one end-point from each path in $\PP(T)$ is a ZFS for $T$.
 \end{prop}
 
\begin{proof}
We prove this by induction on the path cover number. For any tree with $\PP(T)=1$ (a path), the theorem applies. To perform induction step we need to show the following claim. 
\newline {\bf Claim.} In any minimal path covering of a tree there always is a path that is connected (through an edge) to only one other path in the path covering. We call such path a \txtsl{pendant path}. To observe this, suppose there is no such a path in a minimal path covering of a tree $T$. Thus any path is connected to at least two other paths in the path covering. This means the graph has a cycle as a subgraph which contradicts $T$ being a tree.

Assume the theorem holds for all trees $T'$ with $P(T')<P(T)$. Let $\PP(T)$ be a path covering of $T$ with $|\PP(T)|=P(T)$. Let $Z$ be the set consisting of one end-point of each path in $\PP(T)$ and $P_1$ be a pendant path in $\PP(T)$ that is joined to the rest of $T$ by only one edge $uv$ with $v\in V(P_1)$ and $u\not\in V(P_1)$. Then by repeatedly applying the colour-change rule starting at the black end-point of $P_1$, all vertices from the black end-point through to $v$ are coloured black. Now the path $P_1$ is irrelevant to the analysis of the tree $T-V(P_1)$, thus by the induction hypothesis, the black end-points of the remaining paths are a zero forcing set for $T-V(P_1)$, and all vertices not in $P_1$, including $u$, can be coloured black. Hence the remainder of path $P_1$ can also be coloured black and $Z$ is a zero forcing set for $T$. Moreover all the forces are performed along the paths in $\PP(T)$ which completes the proof.   
\end{proof}

We will basically follow a similar idea as in the proof of Proposition~\ref{P(T)=Z(T)} in order to prove Theorem~\ref{For block-cycle Z(G)=P(G)} and Theorem~\ref{outerplanars satisfy Z_+=T} in the following sections.   
In \cite{MR1712856} it has been shown that:
\begin{thm}\label{M(T)=P(T)}
For any tree $T$, $M(T)=P(T)$.
\end{thm}
 Combining Proposition~\ref{P(T)=Z(T)} and Theorem~\ref{M(T)=P(T)} we obtain the following.
 \begin{cor}
 For any tree $T$, $M(T)=Z(T)$. 
 \end{cor}

\section{Graphs with $Z(G)=P(G)$}\label{graphs_with_Z=P}

 The most famous family of graphs satisfying $Z(G)=P(G)$ are trees (see Proposition~\ref{P(T)=Z(T)}). In this section we establish this equality for another family of graphs namely the block-cycle graphs. We will, also, try to give some evidence to show that not so many families of graphs satisfy this equality.
\subsection{Block-cycle graphs} 
 A graph is called \textsl{non-separable}\index{non-separable graph} if it is connected and has no cut-vertices. 
A \txtsl{block} of a graph is a maximal non-separable induced subgraph. A \txtsl{block-cycle} graph is a graph in which every block is either an edge or a cycle (see Figure~\ref{block-cycle}). A block-cycle graph with only one cycle is a \txtsl{unicycle} graph. 

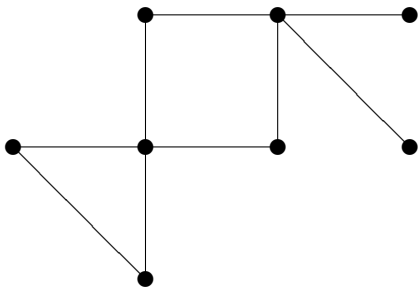
\begin{figure}[h!]
\begin{center}
\begin{picture}(50,0)
\multiput(0,0)(50,0){3}{\circle*{6}}
\multiput(-50,-50)(50,0){4}{\circle*{6}}
\put(0,0){\line(1,0){100}}
\put(-50,-50){\line(1,0){100}}

\put(0,0){\line(0,-1){100}}
\put(50,0){\line(0,-1){50}}
\put(50,0){\line(1,-1){50}}

\put(0,-100){\circle*{6}}
\put(-50,-50){\line(1,-1){50}}

\end{picture}
\vspace{3.5cm}
\end{center}
\caption{A block-cycle graph}
\label{block-cycle}
\end{figure}

Let $b(G)$ be the number of blocks in a block-cycle graph $G$.
 According to the definition, the only block-cycle graph with no cut vertex is either a cycle or an edge. 
In a block-cycle graph each pair of cycles can intersect in at most one vertex, otherwise there will exist a block in the graph which is neither a cycle nor an edge.  
Two blocks are called \txtsl{adjacent} if they have one vertex in common. 
A block in a block-cycle graph is {\it  pendant}\index{pendant block} if it shares only one of its vertices with the other blocks.   
\begin{lem}\label{pendant block}
 Any block-cycle graph has at least two pendant blocks.
\end{lem}
\begin{proof}
Assume $G$ is a given block-cycle graph. To prove this lemma, we will construct a minor $G'$ of the graph $G$  and we will show that the end-points of the longest induced path in $G'$ are associated to the pendant blocks in the original graph.  Let $B_1,\ldots, B_N$ be the blocks in $G$. Note that $B_i$ is either an edge or a cycle, for any $1\leq i\leq N$. If $B_i$ is a cycle, contract all edges in $B_i$ until all that remains is a single vertex, call this $v_i$. In this case we say the vertex $v_i$ in $G'$ is associated to the cycle $B_i$ in $G$. 
Also we say $v_i \in G'$ is associated to the edge $B_i$ in $G$, if $v_i$ is an end-point of the edge $B_i$ in $G$. 
 Note that all the edges adjacent to $B_i$, are adjacent to $v_i$ after this operation. If two cycles $B_i$ and $B_j$ share a vertex in $G$, draw an edge between the associated vertices $v_i$ and $v_j$ in $G'$. 
 
Let $P=\{u_1,u_2,\cdots, u_k\}$ be the longest induced path in $G'$.  Let $u_1=v_i$ and $u_2=v_j$  for some $i$ and $j$ and $v_i$ and $v_j$ correspond to blocks $B_i$ and $B_j$. We first consider the case where  $B_i$ is a cycle while $B_j$ is either an edge or a cycle. We claim that the cycle $B_i$ is a pendant block in $G$. Otherwise there is another block, $B$, which can be an edge or a cycle that shares a vertex with $B_i$ which is different from the vertex that $B_i$ shares with $B_j$.  First, assume that $B$ is a cycle in $G$. 
Let $z$ be the vertex associated to $B$. The only reason $z$  is not in $P$ is that there is an edge between $z$ and some other vertices, say $u_i$, in $P$. The translation of this in the graph $G$ is that there is either an edge or a cycle associated with $u_i$, with $i\neq 1$, which $B$ shares a vertex with. Thus, there exists a cycle in $G$ which shares more than one vertex with at least one other cycle in $G$. This means that there is a block in $G$ that is neither a cycle nor an edge which contradicts with the fact that $G$ is a block-cycle graph. A similar argument applies when $B$ is an edge in $G$. 

Using a similar reasoning  when $B_i$ is an edge in $G$, we get the same contradiction which proves the theorem. 
\end{proof}
The following lemma is straightforward to prove.
\begin{lem}\label{P(G-B), G is a block-cycle}
If $B$ is a pendant block in a block-cycle graph, then 
\[
P(G\backslash B)\leq P(G).\qed
\]
\end{lem} 

\begin{thm}\label{For block-cycle Z(G)=P(G)}
Let $G$ be a block-cycle graph. Then 
\[
Z(G)=P(G).
\]
Furthermore, the paths in any minimal path covering of $G$ are precisely the forcing chains in a minimal zero forcing process initiated by a proper selection of the end-points of the paths in this collection.  
\end{thm}
\begin{proof}
We prove the equation by induction on the number of blocks in $G$ and applying Lemma~\ref{pendant block}. The only block-cycle graph with $b(G)=1$ is either an edge or a cycle and the theorem is clearly true for these graphs. Assume that it is true for all graphs $G'$ with $b(G')< n$. Let $b(G)=n$. According to Lemma~\ref{pendant block}, there is a pendant block, $B$, in $G$ which is connected to the other blocks through a single vertex, $u$. Let $G=G'\,\,\stackplus{u}\,B$. The induction hypothesis holds for $G'$, that is $Z(G')=P(G')$ and a proper selection of the end-points of the paths in a minimal path covering $\PP$ of $G'$ constructs a zero forcing set for it. By Lemma~\ref{P(G-B), G is a block-cycle}, we have $P(G')\leq P(G)$. 
Two cases are then possible: 
\begin{enumerate}[(a)]
\item There is a path-cover $\PP$ for  the graph $G'=G\backslash B$ in which there is a path $P$ of length more than one such that $u$ is an end-point of $P$. 

First assume that $B$ is the edge $uv$. Then $G\backslash B$ is the graph obtained from $G$ by removing the pendant vertex $v$ of $G$. Since $u$ is an end-point of $P$ and $v$ is only connected to $u$, returning $B$ to $G'$ doesn't change the path cover number of the graph. By the induction hypothesis, the paths in the path-cover $\PP$ are the forcing chains of the forcing process initiated by the end-points of the paths in $\PP$. Also since $u$ is an end-point of $P$, we can assume that it doesn't perform any force. Therefore, the zero forcing process will be continued by using $u$ to force $v$. Thus, 
\[
P(G')=P(G)\leq Z(G)\leq Z(G')=P(G'),
\]
which implies $Z(G)=P(G)$.

If $B$ is a cycle, then since it is a pendant cycle, returning it to $G'$ increases the path cover number of $G'$ at least by one, since we will need at least two paths to cover the vertices of a cycle. Assume that $v$ and $w$ are two neighbours of $u$ in $B$. Then, since $u$ is an end-point of $P$, $P\cup \left(V(B) \backslash \{v\}\right)$ can cover all vertices in the cycle $B$ except $v$. Thus the path cover for $G'$ along with $\{v\}$ is a path cover for $G$ and $P(G)=P(G')+1$. Also we need at least two initial black vertices in the cycle to colour its vertices. By assigning colour black to the vertex $v$, all vertices in $B$ will be coloured by continuing the forcing process in $P$ through $u$. Thus,
\[
P(G')+1=P(G)\leq Z(G)\leq Z(G')+1=P(G')+1,
\]
which implies $Z(G)=P(G)$.
\item In every minimal path covering $\PP$ of $G'$, $u$ is an inner vertex of a path in $P$.

 If $B$ is an edge $uv$, then $P(G)\neq P(G')$. Otherwise $v$ is covered in the same path as $u$ is, which contradicts the fact that $u$ is not an end-point of any path in any path covering of $G$. Thus $P(G)=P(G')+1$ since $\PP\cup \{v\}$ covers all vertices of $G$. Also by assigning a black colour to the vertex $v$ we are able to colour the graph $G$ following the same forcing process which we followed to colour the graph $G'$. Thus,
\[
P(G')+1=P(G)\leq Z(G)\leq Z(G')+1=P(G')+1,
\]
which implies $Z(G)=P(G)$.

If $B$ is a cycle, then $P(G)= P(G')+1$, since $\PP$ along with a path covering all vertices of $B$, except $u$, covers all vertices of $G$. Also, by assigning a black colour to one of the neighbours of $u$ in $B$, $w$, we can colour $G$ using the initial set of black vertices in the zero forcing set of $G'$ along with $w$. Thus,
\[
P(G')+1=P(G)\leq Z(G)\leq Z(G')+1=P(G')+1,
\]
which implies $Z(G)=P(G)$.\qedhere
\end{enumerate} 
\end{proof}

The following corollary is obtained from the fact that any unicycle graph is a block-cycle graph. 
\begin{cor}\label{unicycles}
 If $G$ is a unicycle graph, then $Z(G)=P(G)$.\qed
\end{cor}

\subsection{Graphs for which the path cover number is two}
It seems that for a general graph it is rare to satisfy the equality $Z(G)=P(G)$. To show this along with the fact that the discrepancy between $Z(G)$ and $P(G)$ can be arbitrarily large we focus on the family of graphs with $P(G)=2$.  
\begin{prop}
Let $G$ be a graph with $P(G)=2$ and two covering paths $P_1$ and $P_2$ with $|P_1|=m$ and $|P_2|=n$.  Then 
\[
2 \leq Z(G)\leq \text{min}\{n,m\}+1.
\]
Moreover for any number $k$ in this interval, there is a graph $G$ satisfying $P(G)=2$ with $Z(G)=k$.
\end{prop}

\begin{proof}
The claim that $Z(G)\geq 2$ is trivial. Suppose $m\leq n$ and let $B$ be the set consisting of $V(P_1)$ and an end-point of $P_2$.  Obviously $B$ is a ZFS for $G$. Thus $Z(G)\leq |B|=m+1= \text{min}\{n,m\}+1$.
For the given number $k$ in the interval $[2,\text{min}\{m,n\}+1]$ let $G$ be the following  graph. Set two paths $P_1$ and $P_2$ with $|P_1|=m$ and $|P_2|=n$ and assume that $m\leq n$. Starting with an end-point of $P_1$ make each of $k$ consecutive vertices of $P_1$ adjacent to all of the vertices of $P_2$.  Then, it is easy to observe that $Z(G)=k$.      
\end{proof}

Among all the graphs with $P(G)=2$, only those that are also outerplanar satisfy $Z(G)=2$ (see \cite[Theorem 5.1]{johnson2009graphs}).

\section{Graphs with $Z(G)=M(G)$}\label{graph_with_Z=M}
One of the most important questions here, which is motivated by Theorem~\ref{M(G)<=Z(G)},  is to characterize the graphs for which $M(G)=Z(G)$ holds.
Although, all families of graphs with $M(G)=Z(G)$ have not yet  been fully characterized, there are some families of graphs for which this equality has been established. For instance, all the graphs with less than seven vertices (see \cite[Proposition 4.4]{MR2388646}), and the families of graphs are listed in \cite[Table 1]{MR2388646}. In this chapter we expand this list.

A graph $G$ is said to be \txtsl{chordal}, if it has no induced cycle of length more than three. In \cite[Proposition 3.23]{MR2388646} as well as in \cite[Theorem 7]{MR2639259}, it has been shown that the equality $M(G)=Z(G)$ also holds for some families of chordal graphs. In this section, also, we prove that this equality holds for some new families of chordal graphs.

For any graph $G$, assume that $Z_G$ is a minimum zero forcing set of $G$. Given graph $G$ with  $|G|=n$ and the graphs $H_1,\ldots, H_n$, let
\[
G \prec H_1,\ldots,H_n\succ
\]
be the graph obtained by joining all the vertices of the graph $H_i$ to the $i$-th vertex of $G$, where $i=1,\ldots, n$. Note that  $|H_i|$ can be zero, in which case no extra vertices will be joined to the vertex $i$.  We call the graph 
\[
G \prec H_1,\ldots,H_n\succ
\]
a \txtsl{generalized corona} of $G$ with $H_1,\ldots,H_n$. In the next theorem, we use the convention that $Z(H_i)=0$ if $|H_i|=0$.
\begin{thm}\label{generalized corona} For graph $G$ with $|G|=n$ and graphs $H_1,\ldots, H_n$ we have
\[
Z(G \prec H_1,\ldots,H_n\succ)\leq Z(G)+\sum_{i=1}^nZ(H_i).
\]
\end{thm}
\begin{proof}
Let $G=\{v_1,\ldots,v_n\}$. Without loss of generality, we may assume that $Z_G=\{v_1,\ldots,v_k\}$ is a zero forcing set for $G$, and that one can colour all the other vertices of $G$ by a sequence of forces in the following order:
\begin{equation}\label{coloring_order}
v_{k+1},\, v_{k+2},\ldots, v_n.
\end{equation}
Define
\[
\Delta=Z_G\cup Z_{H_1}\cup\ldots\cup Z_{H_n}.
\]
We claim that $\Delta$ is a zero forcing set for $G \prec H_1,\ldots,H_n\succ$, which is of size
\[
|\Delta|=Z(G)+\sum_{i=1}^nZ(H_i).
\]
To prove the claim,  first assume that $\Delta$ is initially coloured black and we colour all the vertices of the graphs $H_i$, with $i=1,\dots,k$, which are associated to the vertices of $Z_G$ using the corresponding set $Z_{H_i}$'s. Now note that  there is a vertex $v_i\in Z_G$ whose only white neighbour in $G$ is $v_{k+1}$; because we were able to colour the vertices of $G$ using $Z_G$ in the order (\ref{coloring_order}). Now, since all the vertices in the graph $H_i$ associated with $v_i$  are black, $v_{k+1}$ is, still, the only white neighbour of $v_i$ in $G \prec H_1,\ldots,H_n\succ$. Thus $v_{k+1}$ can be forced by $v_i$. Then we colour all the vertices in the $H_{k+1}$  using the black vertices in $Z_{H_{k+1}}$. Continuing this process, therefore, we can colour all the vertices of $G \prec H_1,\ldots,H_n\succ$ which proves the claim.
\end{proof}
The following is an immediate consequence of  Theorem~\ref{generalized corona}.
\begin{cor}\label{improve}  For graphs $G$ and $H$  we have
\[
Z(G \prec H,\ldots,H\succ)\leq Z(G)+|G|Z(H).\qed
\]
\end{cor}
The  bound in Corollary~\ref{improve} is, in fact, an improvement of the bound in \cite[Proposition 2.12]{MR2388646}.
Using Theorem~\ref{generalized corona}, Theorem~\ref{M(G)<=Z(G)} and Corollary~\ref{mr&cc},  we can prove that equality in Theorem~\ref{M(G)<=Z(G)} holds for the following families of chordal graphs.

\begin{thm}\label{Z=M Generalized corona}
Let $G$ be a graph satisfying the equalities $Z(G)=M(G)$ and $\mr(G)=\CC(G)$ and $H_1,\ldots, H_n$ be graphs satisfying the equalities $Z(H_i)=M(H_i)$ and $\mr(H_i)=\CC(H_i)$ with $i=1,\dots,n$. Then 
\[
Z(G \prec H_1,\ldots,H_n\succ)=M(G \prec H_1,\ldots,H_n\succ),
\] 
and
\[
\mr(G \prec H_1,\ldots,H_n\succ)=\CC(G \prec H_1,\ldots,H_n\succ).
\]
\end{thm}
\begin{proof}
By (\ref{mr&M}) and Corollary~\ref{mr&cc}  we have 
\begin{equation}\label{eq_1}
M(G \prec H_1,\ldots,H_n\succ)\geq |V(G \prec H_1,\ldots,H_n\succ)|-\CC(G \prec H_1,\ldots,H_n\succ).
\end{equation}
Note that
\[
|V(G \prec H_1,\ldots,H_n\succ)|=|V(G)|+ \sum_{i=1}^{n} |V(H_i)|.
\]
On the other hand we have 
\begin{equation}\label{cc(Generalized corona)}
\CC(G \prec H_1,\ldots,H_n\succ)\leq \CC(G)+ \sum_{i=1}^{n} \CC(H_i). 
\end{equation}
In order to show (\ref{cc(Generalized corona)}) note that each clique $K_{j_i}$ in a minimal clique covering of $H_j$ joined to the vertex $j$ of $G$ constructs a clique isomorphic to $K_{j_i+1}$. Thus by using the same number of cliques we use to cover the edges of $G$ and $H_i$ with $i=1,\dots,n$ we can cover the edges of $G \prec H_1,\ldots,H_n\succ$.  Also by Theorem~\ref{generalized corona} and (\ref{mr&M}),  and the assumptions of the theorem we have
\begin{equation}\label{eq_2}
\begin{tabular}{ll}
$M(G \prec H_1,\ldots,H_n\succ)$&\\
$=|V(G \prec H_1,\ldots,H_n\succ)|-\mr(G \prec H_1,\ldots,H_n\succ)$&\\
$\leq Z(G \prec H_1,\ldots,H_n\succ)$&\\
$\leq Z(G)+\sum_{i=1}^nZ(H_i)$&\\
$=M(G)+\sum_{i=1}^n M(H_i)$&\\
$=|V(G)|+ \sum_{i=1}^{n} |V(H_i)|-\left[\CC(G)+ \sum_{i=1}^{n} \CC(H_i)\right]$&\\
$\leq |V(G \prec H_1,\ldots,H_n\succ)|- \CC(G \prec H_1,\ldots,H_n\succ).$&\\
\end{tabular}
\end{equation}
The theorem follows from  (\ref{eq_1}) and (\ref{eq_2}).
\end{proof}
Theorem~\ref {Z=M Generalized corona} provides a recursive construction  to generate an infinite family of graphs that satisfy the equality $Z(G)=M(G)$. The following is a consequence of Theorem~\ref{Z=M Generalized corona}.    

\begin{cor}\label{Z=M for P_s o K_t}
For the following graphs, $Z(G)=M(G)$:
\begin{enumerate}[(a)]
\item $G_1=K_t\prec P_{s_1},\ldots,P_{s_r}, K_{q_1},\ldots,K_{q_\ell}\succ$, where $s_i,q_j\geq 2$;
\item $G_2=P_t\prec P_{s_1},\ldots,P_{s_r}, K_{q_1},\ldots,K_{q_\ell}\succ$, where $s_i,q_j\geq 2$.
\end{enumerate}
Moreover 
\begin{enumerate}[(1)]
\item $M(G_1)= Z(G_1)=t-1-\ell+q_1+\cdots+q_\ell+r;$
\item $\mr(G_1)=\CC(G_1)= 1+\ell+s_1+\cdots+s_r-r;$
\item $M(G_2)= Z(G_2)=1-\ell+q_1+\cdots+q_\ell+r;$
\item $\mr(G_2)=\CC(G_2)= t-1+\ell+s_1+\cdots+s_r-r.$
\end{enumerate}
\end{cor}
\begin{proof}
Statements $(1)$, $(2)$, $(3)$ and $(4)$ follow from the facts  $Z(K_t)=t-1$, $Z(P_t)=1$, $\CC(K_t)=1$ and $\CC(P_t)=t-1$, for any $t$. 
\end{proof}

The next theorem demonstrates the existence of another family of chordal graphs for which  the equality holds in Theorem~\ref{M(G)<=Z(G)}.
\begin{thm}\label{another M=Z}
Let $G$ be a chordal graph consisting of $N\geq 1$ cliques, $K_{n_1},\ldots, K_{n_N}$, such that the intersection of any two cliques is a complete graph and no vertex is contained in more than two cliques. Also, let $k_{i,j}$ be the size of the clique in the intersection of $K_{n_i}$ and $K_{n_j}$ then
\[
M(G)=Z(G)=\sum_{i=1}^{N}n_i-\sum_{i<j}k_{ij}-N,
\]
and
\[
\mr(G)=\CC(G)=N.
\]
\end{thm}
\begin{proof}
 The proof is by induction on the number of cliques. It is clearly true for $N=1$. Since no vertex of $G$ is contained in more than two cliques, it has a pendant clique,  that is, a clique which intersects only one other clique in the graph. Otherwise we have an induced cycle of order more than three in $G$ which contradicts with $G$ being chordal. Without loss of generality, assume that, this clique is  $K_{n_1}$ and that it only intersects with $K_{n_2}$. Let $H=K_{n_1}\backslash \left(K_{n_1}\cap K_{n_2}\right)$ and $G'=G\backslash H$. By the induction hypothesis we have,
\[
M(G')= Z(G')= \sum_{i=2}^{N}n_i-\sum_{1<i<j}k_{ij}-(N-1).
\]
It is clear that
\[
\CC(G')\leq N-1.
\]
Then it follows that a ZFS from $G'$ with an additional $n_1-k_{1,2}-1$  black vertices from $H$ forms a ZFS for $G$. Thus,
\[
M(G)\leq Z(G)\leq \sum_{i=2}^{N}n_i-\sum_{1<i<j}k_{ij}-N+1+ n_1-k_{1,2}-1=\sum_{i=1}^{N}n_i-\sum_{i<j}k_{ij}-N,
\]
and
\[
\CC(G)\leq N.
\]
 It is clear that
\[
|G|=\sum_{i=1}^{N}n_i-\sum_{i<j}k_{ij}.
\]
Thus, using (\ref{mr&M}) and Corollary~\ref{mr&cc}, we obtain the following lower bound for $M(G)$:
\[
M(G)\geq \sum_{i=1}^{N}n_i-\sum_{i<j}k_{ij}-N,
\]
which completes the proof.
\end{proof}
Note that the equality in Theorem~\ref{M(G)<=Z(G)} does not hold for all chordal graphs. See for instance \cite[Example 2.11]{MR2645093}.

\section{The Colin de Verdi$\grave{\text{e}}$re graph parameter}\label{colin de} 
In 1990, Colin de Verdi$\grave{\text{e}}$re \cite{MR1224700} introduced an interesting new parameter $\mu(G)$ for any undirected graph $G$. The parameter $\mu(G)$ can be fully described in terms of properties of matrices related to $G$.
One of the interesting applications of $\mu$ is that certain topological properties of a graph $G$ can be characterized by spectral properties of matrices associated with $G$ including values of $\mu$, see \cite{MR1224700,MR1673503}. Another interesting graph theoretical property of $\mu$  is that it is monotone on graph minors. Before defining this parameter, we need  to introduce the Strong Arnold Property, which will be abbreviated by SAP throughout the thesis.

The \txtsl{Hadamard product} of two matrices $A=[a_{ij}]$ and $B=[b_{ij}]$ of the same size is just their element-wise product $A\circ B\equiv [a_{ij}b_{ij}]$.
Let $A$ and $X$ be symmetric $n\times n$ matrices. We say that $X$ \textsl{fully annihilates}\index{fully annihilation} $A$ if
\begin{enumerate}
\item $AX=0$,
\item $A\circ X=0$, and
\item $I_n \circ X=0$.
\end{enumerate}
The matrix $A$ has the Strong Arnold Property (SAP) if the zero matrix is the only symmetric matrix that fully annihilates $A$.
The following is a basic, yet useful, observation concerning symmetric matrices with small nullity. In particular, it demonstrates that there are matrices in $\mathcal{S}(G)$ that satisfy SAP, for any graph $G$.
\begin{lem}\label{null<=1 has SAP}
If $\nul(A)\leq1$, then $A$ has SAP.
\end{lem}
\begin{proof}
If null$(A)=0$, then $A$ is nonsingular and the only matrix $X$ that fully annihilates $A$ is the zero matrix. Suppose, now, that null$(A)=1$ and let $X$ fully annihilate $A$. Thus by condition ($3$), the diagonal of $X$ is $0$. Since $X$ is symmetric, this implies that $X$ is not a rank $1$ matrix. Hence,  if $X\neq 0$, then rank$(X)\geq 2$ and $AX=0$ would imply null$(A)\geq 2$. Therefore $X=0$ and $A$ has SAP.
\end{proof}

Now we can define the Colin de Verdi$\grave{\text{e}}$re parameter. For a given graph $G$, $\mu(G)$ is defined to be the maximum multiplicity of $0$ as an eigenvalue of $L$, where $L=[l_{i,j}]$ satisfies all of the following conditions:
\begin{enumerate}
\item  $L\in \mathcal {S}(G)$ and $l_{i,j}\leq 0$, for $i\neq j$;
\item $L$ has exactly one negative eigenvalue (with multiplicity one);
\item  $L$ has SAP.
\end{enumerate}
In other words $\mu(G)$ is the maximum nullity among the matrices satisfying (1)-(3) above. Further, observe that
\begin{equation}\label{mu&M}
\mu(G)\leq M (G) = n-\mr(G)\leq Z(G).
\end{equation}
Hence there is an obvious relationship between $\mu(G)$ and mr$(G)$.

the \txtsl{linking number} is a numerical invariant that describes the linking of two closed curves in three-dimensional space. Intuitively, the linking number represents the number of times that each curve winds around the other. A \txtsl{linkless embedding} of an undirected graph is an embedding of the graph into Euclidean space in such a way that no two cycles of the graph have nonzero linking number.

Colin de Verdi$\grave{\text{e}}$re et al \cite{MR1224700,MR1673503}, through some  sophisticated theorems, have shown that :
\begin{itemize}
\item $\mu(G)\leq 1$ if and only if $G$ is a disjoint union of paths;
\item $\mu(G)\leq 2$ if and only if $G$ is outerplanar;
\item $\mu(G)\leq 3$ if and only if $G$ is planar;
\item $\mu(G)\leq 4$ if and only if $G$ is linklessly embeddable.
\end{itemize}

A related parameter, also introduced by Colin de Verdi$\grave{\text{e}}$re (see \cite{MR1654157}), is denoted by $\nu(G)$ and is defined to be the maximum nullity among matrices $A$ that satisfy the following conditions:
\begin{enumerate}
\item $A\in\mathcal{S}(G)$;
\item  $A$ is positive semidefinite;
\item $A$ has SAP.
\end{enumerate}
Note, also, that
\begin{equation}\label{nu&M}
\nu(G)\leq M_+ (G) = n-\mr_+(G).
\end{equation}
Thus, there is an obvious relationship between $\nu(G)$ and $\mr_+(G)$.

Properties analogous to $\mu(G)$ have been established for $\nu(G)$. For example, $\nu(G)\leq 2$ if the dual of $G$ is outerplanar, see \cite{MR1654157}. Furthermore, $\nu(G)$, like $\mu(G)$ is a minor-monotone graph parameter.

Consequently, in order to learn more about the minimum rank of graphs, another related parameter which is denoted by $\xi(G)$ is been introduced by Fallat et al. \cite{barioli2005variant} with the following definition.   
\begin{defn}
 For a graph $G$, $\xi(G)$ is the maximum nullity among matrices $A \in \mathcal{S}(G)$ having SAP.
 \end{defn}
 For example, based on Lemma~\ref{null<=1 has SAP} we have :
 \begin{equation}\label{xi(K_n)=n-1} 
 \xi(K_n)=n-1.
 \end{equation}
 One of the most important properties of $\xi(G)$ analogous to $\mu(G)$ is its minor-monotonicity \cite[Corollary 2.5]{barioli2005variant} . Also, note that
 \begin{equation}\label{zi&mu}
\mu(G)\leq \xi(G)\leq M(G) \leq Z(G) ,
\end{equation}
and
 
 \begin{equation}\label{zi&M}
\nu(G)\leq \xi(G)\leq M(G) \leq Z(G) .
\end{equation}
 
Two of the most interesting open questions in this area that have been proposed by Colin de Verdi$\grave{\text{e}}$re in 1998,  are the following.
\begin{conj}\label{chi&mu} For any graph $G$,
$\chi(G)\leq\mu(G)+1.$
\end{conj}
\begin{conj}\label{chi&nu} For any graph $G$,
$\chi(G)\leq\nu(G)+1.$
\end{conj}

According to Conjecture~\ref{chi&mu} and Equations (\ref{M(G)<=Z(G)}) and (\ref{mu&M}), a weaker comparison can be stated as follows:
\newline{\bf Question.} Does the following inequality hold for any graph $G$;
$$\chi(G)\leq Z(G)+1?$$
The next section is devoted to answering this question.

\section{The zero forcing number and the chromatic number}
Since the zero forcing number of a graph is a graph parameter associated with a new type of graph colouring, it seems natural  to ask if there is any relationship between this parameter and the traditional graph colouring parameter called the chromatic number of the graph. This, along with the second last question of  Section~\ref{colin de} motivates us to find a relationship between the zero forcing number and the chromatic number of a graph.

\begin{thm}\label{Z&delta}
Let $G$  be a graph. Then we have
\[
Z(G) \geq  \max \{ \delta(H') \,\, | \,\, H'\,\, \text{is any induced subgraph of} \,\,G \}.
\]
\end{thm}
\begin{proof}
Let $H$ be an induced subgraph of $G$ such that
\[
\delta(H)=\max \{ \delta(H')  \,\,| \,\, H'\,\, \text{is any induced subgraph of}\,\, G \}.
\]
Let $Z$ be a zero forcing set of $G$ with $|Z|=Z(G)$. If $Z$ includes all vertices of $H$, then there is nothing to prove. If not, and no vertex of $H$ performs a force in the zero forcing process with the initial black vertices in $Z$, then there is a zero forcing set $Z'$ (the reversal of $Z$) with $|Z'|=|Z|$ that includes all vertices in $H$ and again the result follows.
Now assume that  there is at least one vertex in $H$ performing a force and let $v$ be the first vertex of $H$ which performs a force in the process (this vertex is either in $Z$ or it is forced by a vertex not in $H$).
Since $v$ is the first one which performs a force, all the neighbours of $v$ in $H$, except one, should already have been forced in some distinct forcing chains. The initial vertices of these chains are in $Z$. Thus,
\[
Z(G) \geq d_H(v)-1+1 \geq \delta(H),
\]
which completes the proof.
\end{proof}
The following inequality shows the connection between the chromatic number of a graph and the minimum degree among all the induced subgraphs of the graph (see \cite[Section 4.2]{MR2368647}).
\begin{lem}\label{chi&delta} For any graph $G$,
\[
\chi(G)  \leq 1 + \max  \{ \delta(H')\,\,  |\,\,  H\,\,'\text{is any induced subgraph of}\,\,G \}.\qed
\]
\end{lem}

\begin{cor}\label{Z& chi}
For any graph $G$ we have
\[
\chi(G) \leq Z(G)+1.
\]
\end{cor}
\begin{proof}
The inequality follows directly from Theorem~\ref{Z&delta} and Lemma~\ref{chi&delta}.
\end{proof}
Note that the bound in the above inequality is tight for some graphs such as  paths and complete graphs.
The following corollary is a direct result of Theorem~\ref{Z&delta}.
\begin{cor}\label{subgraphs K_n & K_(p,q)}
Let $G$ be a graph.
\begin{enumerate}
\item If $K_n$ is a subgraph of $G$ then $Z(G)\geq n-1$.
\item If $K_{p,q}$ is a subgraph of $G$ then $Z(G)\geq \min\{p,q\}$.\qed
\end{enumerate}
\end{cor} 
In fact there is a stronger result than the first statement of Corollary~\ref {subgraphs K_n & K_(p,q)} based on the minor monotonicity property of $\xi(G)$ and (\ref{zi&M}) and (\ref {xi(K_n)=n-1}) which is as follows:
\begin{prop}
If $G$ has a $K_n$ minor then 
\[
Z(G)\geq n-1.   
\]
\end{prop}

Also we can improve the bound in the second statement of Corollary~\ref {subgraphs K_n & K_(p,q)} as follows.
\begin{prop}
If the complete bipartite graph $K_{p,q}$ is a  subgraph of $G$ then,
\[Z(G)\geq \min\{p,q\}+1,\] 
provided that $(p,q)\neq (1,1), (1,2), (2,1), (2,2)$.
\end{prop}
\begin{proof}
Let the graph $K_{p,q}$ be a subgraph of the graph $G$. Assume that $(X,Y)$ are the partitions of $K_{p,q}$ with $|X|=p$ and $|Y|=q$ where $p\leq q$. 
Let $Z$ be a zero forcing set of $G$ with $|Z|=Z(G)$. If $Z$ includes all vertices of $K_{p,q}$, there is nothing to prove. If not and no vertex of $K_{p,q}$ performs a force in the zero forcing process with the initial black vertices in $Z$, then there is a zero forcing set $Z'$ (the reversal of $Z$) with $|Z'|=|Z|$ that includes all vertices in $K_{p,q}$ and the inequality follows.

Now assume that  there is at least one vertex in $K_{p,q}$ performing a force. Let $v$ be the first vertex of $K_{p,q}$ which is forcing a vertex $u$ in a zero forcing process starting with the vertices in $Z$ ($v$ is either in $Z$ or it is forced by a vertex not in $K_{p,q}$). First assume that $v\in X$. 
If $u$ is not a vertex of $K_{p,q}$, since $v$ is the first vertex of $K_{p,q}$ which performs a force, all $q$ neighbours of $v$ in $K_{p,q}$ should already have been forced in some distinct forcing chains and the result follows. If $u$ is in $K_{p,q}$, then there are $q-1$ distinct forcing chains having the neighbours of $v$ in $K_{p,q}$ except $u$ as their end-points. If $u$ has no other neighbours in $X$, then $\min\{p,q\}=1$ and the inequality follows from the fact that $Z(G)\geq q\geq 2=\min\{p,q\}+1$. If not and $u$ is the second vertex of $K_{p,q}$ that performs a force, this requires at least all the neighbours of $u$ in $K_{p,q}$, except one, are already coloured black in some distinct forcing chains. Therefore, $Z(G)\geq q+1\geq \min\{p,q\}+1$. If not and the forcing chain containing $v$ ends with $u$, then there is at least one more forcing chain to colour the rest of the vertices in $X$. Thus $Z(G)\geq q+1\geq \min\{p,q\}+1$. A similar argument applies when $v\in Y$.         
\end{proof}


\chapter{Zero forcing number and graph operations}\label{zfs_and_graph_operations}

In this chapter we study effects of some graph operations, including operations on a single graph or operations on several graphs, on the zero forcing number of the resulting graph.    
The impact of the graph operations on zero forcing number shows this parameter, similar to maximum nullity, is not monotone on subgraphs.
 \section{Simple operations}\label{Simple operations}
 Deleting a vertex from a graph can increase or decrease the zero forcing number, but by at most one in either direction. It is somehow predictable that vertex  deletion could either decrease the zero forcing number or  keep it the same, but surprisingly, there are examples in which vertex deletion increases the zero forcing number.   
\begin{prop}[{{see \cite[Theorem 6.4]{owens2009properties}}}]\label{G-v}
Let $v$ be a vertex of the graph $G$. Then $Z(G - v) - 1\leq Z(G) \leq Z(G - v) + 1$.\qed
\end{prop}
The following are examples of all three possible cases in Proposition~\ref{G-v}.
\begin{itemize}
\item Decrease:
\bigskip
\begin{figure}[H]
\centering
\unitlength=1pt
\begin{picture}(150,0)
\put(0,0){\circle*{6}}
\put(-25,-50){\circle*{6}}
\put(25,-50){\circle{6}}

\put(1.7,-2.7){\line(1,-2){22.5}}
\put(-1.7,-2.7){\line(-1,-2){22.5}}
\put(-22.2,-50){\line(1,0){44}}

\put(0,8){\makebox(0,0){$v$}}
\put(0,-70){\makebox(0,0){$Z(G)=2$}}


\put(125,-25){\circle*{6}}
\put(175,-25){\circle{6}}
\put(125,-25){\line(1,0){47}}

\put(150,-70){\makebox(0,0){$Z(G-v)=1$}}

\end{picture}
\vspace{2cm}
\caption{An example of a graph for which removing a vertex decreases the zero forcing number}
\end{figure}

\item No change:
\bigskip

\begin{figure}[H]
\centering
\unitlength=1pt
\begin{picture}(150,0)
\put(-25,0){\circle*{6}}
\put(25,0){\circle{6}}

\put(-22.2,0){\line(1,0){44}}

\put(-25,8){\makebox(0,0){$v$}}
\put(0,-20){\makebox(0,0){$Z(G)=1$}}


\put(125,0){\circle*{6}}
\put(125,-20){\makebox(0,0){$Z(G-v)=1$}}

\end{picture}
\vspace{.5cm}
\caption{An example of a graph for which removing a vertex does not change the zero forcing number}
\end{figure}

\item Increase: 
\medskip

\begin{figure}[H]
\centering
\unitlength=1pt
\begin{picture}(150,50)

\put(-25,50){\circle{6}}
\put(25,50){\circle*{6}}
\put(-22.2,50){\line(1,0){44}}

\put(-50,0){\circle{6}}
\put(0,0){\circle*{6}}
\put(-47.2,0){\line(1,0){44}}

\multiput(-25,-50)(50,0){2}{\circle{6}}
\put(-22.2,-50){\line(1,0){44}}

\put(1.7,2.7){\line(1,2){22.2}}
\put(-1.7,2.7){\line(-1,2){22.2}}

\put(1.7,-2.7){\line(1,-2){22.2}}
\put(-1.7,-2.7){\line(-1,-2){22.2}}

\put(-48.3,2.7){\line(1,2){22.2}}
\put(-48.3,-2.7){\line(1,-2){22.2}}

\put(-50,8){\makebox(0,0){$v$}}
\put(0,-70){\makebox(0,0){$Z(G)=2$}}


\put(125,50){\circle{6}}
\put(175,50){\circle*{6}}
\put(127.8,50){\line(1,0){47}}
\put(150,0){\circle*{6}}
\put(125,-50){\circle*{6}}
\put(175,-50){\circle{6}}

\put(125,-50){\line(1,0){47}}

\put(150,0){\line(1,2){25}}
\put(148.3,2.7){\line(-1,2){22.2}}

\put(150,0){\line(1,-2){23.7}}
\put(150,0){\line(-1,-2){25}}

\put(150,-70){\makebox(0,0){$Z(G-v)=3$}}

\end{picture}
\vspace{2.3cm}
\caption{An example of a graph for which removing a vertex increases the zero forcing number}
\label{M(G)=M(G-v)-1}
\end{figure}
\end{itemize}

It is not hard to observe the following.
\begin{prop}
Let $G$ be a graph with a vertex labeled $v$. 
\begin{enumerate}[(a)]
\item If $v\in Z$ where $Z$ is a minimal ZFS of $G$, then $Z(G)-1\leq Z(G-v)\leq Z(G)$.
\item If $Z(G-v)=Z(G)+1$, then there is no minimal ZFS of $G$ including $v$. \qed
\end{enumerate}
\end{prop} 

 
 Similarly edge deletion can change the zero forcing number by at most $1$.  
\begin{prop}[{{see \cite[Theorem 5.2]{owens2009properties}}}]\label{Z(G-e)}
 Let $G$ be a connected graph. If $e = uv$ is an edge of $G$, then
\[Z(G)-1 \leq Z(G-e) \leq Z(G) + 1. \qed\] 
\end{prop}
The following are examples of all three possible cases in Proposition~\ref{Z(G-e)}.
\begin{itemize}
\item Decrease:
\bigskip

\begin{figure}[H]
\centering
\unitlength=1pt
\begin{picture}(150,0)
\put(0,0){\circle*{6}}
\put(-25,-50){\circle*{6}}
\put(25,-50){\circle{6}}

\put(0,0){\line(1,-2){23.6}}
\put(0,0){\line(-1,-2){25}}
\put(-22.2,-50){\line(1,0){44}}

\put(20,-23){\makebox(0,0){$e$}}
\put(0,-70){\makebox(0,0){$Z(G)=2$}}


\put(150,0){\circle{6}}
\put(125,-50){\circle{6}}
\put(175,-50){\circle*{6}}

\put(148.3,-2.7){\line(-1,-2){22.2}}
\put(127.8,-50){\line(1,0){45}}

\put(20,-23){\makebox(0,0){$e$}}
\put(150,-70){\makebox(0,0){$Z(G-e)=1$}}

\end{picture}
\vspace{2.3cm}
\caption{An example of a graph for which removing an edge decreases the zero forcing number}
\end{figure}

\newpage
\item No change:
\bigskip

\begin{figure}[H]
\centering
\unitlength=1pt
\begin{picture}(200,0)
\put(0,0){\circle{6}}
\put(50,0){\circle*{6}}
\put(0,-50){\circle{6}}
\put(50,-50){\circle*{6}}

\put(2.2,-2.2){\line(1,-1){45.5}}

\put(2.8,0){\line(1,0){44}}
\put(2.8,-50){\line(1,0){44}}

\put(0,-2.8){\line(0,-1){44}}
\put(50,-2.8){\line(0,-1){44}}

\put(30,-22){\makebox(0,0){$e$}}
\put(25,-70){\makebox(0,0){$Z(G)=2$}}


\put(150,0){\circle{6}}
\put(200,0){\circle*{6}}
\put(150,-50){\circle{6}}
\put(200,-50){\circle*{6}}

\put(152.8,0){\line(1,0){44}}
\put(152.8,-50){\line(1,0){44}}

\put(150,-2.8){\line(0,-1){44}}
\put(200,-2.8){\line(0,-1){44}}

\put(175,-70){\makebox(0,0){$Z(G-e)=2$}}

\end{picture}
\vspace{2.3cm}
\caption{An example of a graph for which 
removing an edge does not change the zero forcing number}
\end{figure}

\item Increase:
\bigskip

\begin{figure}[H]
\centering
\unitlength=1pt
\begin{picture}(150,50)

\multiput(-25,50)(50,0){2}{\circle{6}}
\put(-22.2,50){\line(1,0){44}}

\put(0,0){\circle*{6}}

\put(-25,-50){\circle*{6}}
\put(25,-50){\circle{6}}
\put(-22.2,-50){\line(1,0){44}}

\put(1.6,2.6){\line(1,2){22.3}}
\put(-1.6,2.6){\line(-1,2){22.3}}

\put(1.6,-2.6){\line(1,-2){22.3}}
\put(-1.6,-2.6){\line(-1,-2){22.3}}

\put(25,47.2){\line(0,-1){94}}

\put(33,0){\makebox(0,0){$e$}}
\put(0,-70){\makebox(0,0){$Z(G)=2$}}


\put(125,50){\circle{6}}
\put(175,50){\circle*{6}}
\put(127.8,50){\line(1,0){44.1}}

\put(150,0){\circle*{6}}

\put(125,-50){\circle*{6}}
\put(175,-50){\circle{6}}
\put(127.8,-50){\line(1,0){44}}

\put(151.6,2.6){\line(1,2){22.3}}
\put(148.4,2.6){\line(-1,2){22.3}}

\put(151.6,-2.6){\line(1,-2){22.3}}
\put(148.4,-2.6){\line(-1,-2){22.3}}

\put(150,-70){\makebox(0,0){$Z(G-e)=3$}}

\end{picture}
\vspace{2.3cm}
\caption{An example of a graph for which 
removing an edge increases the zero forcing number}
\end{figure}
\end{itemize}

It is not hard to observe the following
\begin{prop}
Let $e=uv$ be an edge in a graph $G$. If there is a minimal zero forcing process for $G$ in which neither $u$ forces $v$ nor $v$ forces $u$, then $Z(G)-1\leq Z(G-e)\leq Z(G)$. Otherwise $Z(G)\leq Z(G-e)\leq Z(G)+1$. 
\end{prop}
\begin{proof}
In order to prove the case ``otherwise'', assume that $\PP$ is a minimal collection of forcing chains of $G$ in which $u$ forces $v$ in the chain $P$. Let $P\backslash \{e\}=P_1\cup P_2$. Then by assigning the black colour to $v$, $\left(\PP\backslash \{P\}\right)\cup\{P_1,P_2\}$ is a minimal collection of forcing chains for $G-e$. Thus $Z(G-e)\leq Z(G)+1$. To show that $Z(G)\leq Z(G-e)$, suppose that $Z(G-e)\leq Z(G)-1$. Then there is a collection of forcing chains $\PP'$ of $G-e$ of size $Z(G)-1$. First suppose that $u$ and $v$ are in the same chain $P'\in\PP'$ and $u$ is in a  lower level than $v$. Assume $u$ forces $w$ and $w'$ forces $v$ in this chain. Let $P'\backslash \{uw,w'v\}=P'_1\cup P'_2\cup P'_3$ in which $P'_1$  ends with $u$ and $P'_2$ starts with $w$ and ends with $w'$ and $P'_3$ starts with $v$. Let $P''$ be the chain obtained by merging two chains $P'_1$ and $P'_3$ by adding the edge $uv$. Then by assigning the black colour to $w$, $\left(\PP\backslash \{P\}\right)\cup\{P'_3,P''\}$ is a minimal collection of forcing chains of $G$ in which $u$ forces $v$. A similar argument applies when $u$ and $v$ are in different chains.
\end{proof}

Edge contraction affects the zero forcing number in the following ways.
 \begin{prop}[{{see \cite[Theorem 5.1]{owens2009properties}}}]\label{Z(G/e)}
 Let $e = uv$ be an edge in a graph $G$. Then
\[
Z(G)-1\leq Z(G/e) \leq Z(G) + 1. \qed
\]
\end{prop}  
The following are examples of all three possible cases in Proposition~\ref{Z(G/e)}.
\begin{itemize}

\newpage
\item Decrease:
\medskip

\begin{figure}[H]
\centering
\unitlength=1pt
\begin{picture}(150,0)
\put(0,0){\circle*{6}}
\put(-50,-50){\circle*{6}}
\multiput(0,-50)(50,0){2}{\circle{6}}

\put(0,0){\line(0,-1){47}}
\put(-50,-50){\line(1,0){47}}
\put(2.8,-50){\line(1,0){44.1}}

\put(8,-25){\makebox(0,0){$e$}}
\put(0,-70){\makebox(0,0){$Z(G)=2$}}


\put(150,-25){\circle*{6}}
\multiput(200,-25)(50,0){2}{\circle{6}}

\put(150,-25){\line(1,0){47}}
\put(202.8,-25){\line(1,0){44.1}}

\put(200,-70){\makebox(0,0){$Z(G/e)=1$}}

\end{picture}
\vspace{2cm}
\caption{An example of a graph for which contracting an edge decreases the zero forcing number}
\end{figure}

\item No change:
\medskip

\begin{figure}[H]
\centering
\unitlength=1pt
\begin{picture}(150,0)
\put(-25,0){\circle*{6}}
\put(25,0){\circle{6}}

\put(-22.2,0){\line(1,0){44}}

\put(0,8){\makebox(0,0){$e$}}
\put(0,-20){\makebox(0,0){$Z(G)=1$}}

\put(125,0){\circle*{6}}
\put(125,-20){\makebox(0,0){$Z(G/e)=1$}}

\end{picture}
\vspace{.5cm}
\caption{An example of a graph for which 
contracting an edge does not change the zero forcing number}
\end{figure}

\item Increase: 

\begin{figure}[H]
\centering
\unitlength=1pt
\begin{picture}(150,50)

\put(-25,50){\circle{6}}
\put(25,50){\circle*{6}}

\multiput(0,0)(0,-50){2}{\circle{6}}
\put(0,-2.8){\line(0,-1){44}}
\put(-25,-100){\circle{6}}
\put(25,-100){\circle*{6}}
\put(1.7,2.7){\line(1,2){22.2}}
\put(-1.7,2.7){\line(-1,2){22.2}}
\put(1.7,-52.7){\line(1,-2){22.2}}
\put(-1.7,-52.7){\line(-1,-2){22.2}}

\put(-8,-25){\makebox(0,0){$e$}}
\put(0,-120){\makebox(0,0){$Z(G)=2$}}


\put(125,25){\circle*{6}}
\put(175,25){\circle*{6}}

\put(150,-25){\circle{6}}

\put(125,-75){\circle{6}}
\put(175,-75){\circle*{6}}
\put(151.7,-22.3){\line(1,2){22.2}}
\put(148.3,-22.3){\line(-1,2){22.2}}
\put(151.7,-27.7){\line(1,-2){22.2}}
\put(148.3,-27.7){\line(-1,-2){22.2}}

\put(150,-120){\makebox(0,0){$Z(G/e)=3$}}

\end{picture}
\vspace{3.9cm}
\caption{An example of a graph for which 
contracting an edge increases the zero forcing number}
\end{figure}
\end{itemize}

Subdividing an edge can only increase the zero forcing number by at most one.  
 \begin{prop}[{{see \cite[Theorem 5.4]{owens2009properties}}}]\label{Z(G.e)} 
If $H$ is obtained from $G$ by subdividing an edge $e = uv$, then
\[
Z(G)\leq Z(H) \leq Z(G) + 1.
\]
\end{prop}
The following are examples of the two possible cases in Proposition~\ref{Z(G.e)}.
\begin{itemize}
\item No change:
\bigskip

\begin{figure}[H]
\centering
\unitlength=1pt
\begin{picture}(150,0)
\put(-25,0){\circle*{6}}
\put(25,0){\circle{6}}
\put(-22.2,0){\line(1,0){44}}

\put(0,8){\makebox(0,0){$e$}}
\put(0,-20){\makebox(0,0){$Z(G)=1$}}


\put(125,0){\circle*{6}}
\put(175,0){\circle{6}}
\put(225,0){\circle{6}}
\put(127.8,0){\line(1,0){44}}
\put(177.8,0){\line(1,0){44}}

\put(175,-20){\makebox(0,0){$Z(H)=1$}}

\end{picture}
\vspace{.5cm}
\caption{An example of a graph for which 
subdividing an edge does not change the zero forcing number}
\end{figure}

\item Increase:  
\bigskip

\begin{figure}[H]
\centering
\unitlength=1pt
\begin{picture}(200,0)
\multiput(0,0)(50,0){2}{\circle*{6}}
\multiput(0,-50)(50,0){2}{\circle{6}}

\put(2.2,-2.2){\line(1,-1){45.5}}

\put(2.8,0){\line(1,0){44}}
\put(2.8,-50){\line(1,0){44}}

\put(0,-2.8){\line(0,-1){44}}
\put(50,-2.8){\line(0,-1){44}}

\put(30,-22){\makebox(0,0){$e$}}
\put(25,-70){\makebox(0,0){$Z(G)=2$}}


\multiput(150,0)(50,0){2}{\circle*{6}}
\multiput(150,-50)(50,0){2}{\circle{6}}
\put(152.2,-2.2){\line(1,-1){45.5}}
\put(175,-25){\circle*{6}}

\put(152.8,0){\line(1,0){44}}
\put(152.8,-50){\line(1,0){44}}

\put(150,-2.8){\line(0,-1){44}}
\put(200,-2.8){\line(0,-1){44}}

\put(175,-70){\makebox(0,0){$Z(H)=3$}}

\end{picture}
\vspace{2.3cm}
\caption{An example of a graph for which 
subdividing an edge increases the zero forcing number}
\end{figure}
\end{itemize}

In the next proposition we bound the variation of the zero forcing number of a specific type of graphs after adding edges and vertices to it. 
 A graph $G$ is called a \txtsl{semi-complete} graph if all of its forcing chains in any minimal zero forcing process are pairwise adjacent. Note that such a graph is \textsl{contractable} to a $K_{Z(G)}$ and any complete graph is semi-complete. 

\begin{prop}\label{Adding edge to a semi-complete} 
Let $G$ be a semi-complete graph and $u$ and $v$ be two nonadjacent vertices of $G$.
Let $G_1=G+uv$ and $G_2$ be the graph obtained from $G$ by adding a new vertex $z$ to $V(G)$ that is connected to the vertices of $G$ in any fashion. Then 
\begin{enumerate}
\item $Z(G)\leq Z(G_1)\leq Z(G)+1.$
\item $Z(G)\leq Z(G_2)\leq Z(G)+1.$
\end{enumerate}
\end{prop}
\begin{proof}
Suppose that $Z(G_1)<Z(G)$. Let $\PP$ be a collection of forcing chains of $G_1$ with $|\PP|=Z(G_1)$. If in this zero forcing process neither $u$ forces $v$ nor $v$ forces $u$, then $\PP$ is a ZFS for $G$. Thus  $Z(G)=Z(G_1-uv)\leq Z(G_1)$ that implies $Z(G)<Z(G)$ which is a contradiction. Assume that  there is a forcing chain in $\PP$, say $P$, that has a force through $uv$, then $P-uv=P_1 \bigcup P_2$ where $P_1$ and $P_2$ are two non-adjacent forcing chains one of them starting with $u$ (or $v$). Therefore by assigning a black colour to $u$ (or $v$), $\left(\PP\backslash \{P\}\right) \cup\{P_1, P_2\}$ is a minimal collection of forcing chains of $G$ with two non-adjacent chains $P_1$ and $P_2$. This contradicts the fact that $G$ is a semi-complete graph and therefore, proves the claim. 

A similar argument applies for the second statement.
\end{proof}
 Proposition~\ref{Adding edge to a semi-complete} and testing several examples motivates us to propose the following conjecture.
\begin{conj}
If the graph $G$ has a $K_n$ minor, then $G$ has a semi-complete subgraph, $H$, with $Z(H)\geq n-1$. 
\end{conj} 

\section{Vertex-sum of graphs} 
One of the most common graph operations is the vertex-sum of two graphs. The following theorem calculates the zero forcing number of the vertex-sum of two graphs.

\begin{thm}\label{zero forcing of vertex-sum}
Let $G$ and $H$ be two graphs each with a vertex labeled $v$. Then the following hold.
\begin{enumerate} 
\item If $v$ is in some minimal ZFS of $H$ but not in any minimal ZFS of $G$, then
\[
Z(G\,\,\stackplus{v}\,H)=Z(G)+Z(H).
\]
\item If $v$ is in some minimal ZFS of $H$ and some minimal ZFS of $G$, then
\[
Z(G\,\,\stackplus{v}\,H)=Z(G)+Z(H)-1.
\]
\item If $v$ is neither in any minimal ZFS of $H$ nor in any minimal ZFS of $G$, then
\[
Z(G\,\,\stackplus{v}\,H)=Z(G)+Z(H)+1.
\]
\end{enumerate}
\end{thm}
\begin{proof}
Let $Z_G$ be a minimal ZFS of $G$ and $Z_H$ be a minimal ZFS of $H$. 

To verify the first case  assume that $v\in Z_H$. By Theorem~\ref{Reversal is a ZFS}, if $Z'_H$ is the reversal of $Z_H$, then $Z'_H$ is a minimal ZFS in which $v$ doesn't perform a force. Therefore starting with the initial black vertices in $Z'_H\cup Z_G$ and colouring $H$ before $G$ we can colour the entire graph $G\,\stackplus{v}\,H$. Thus 
\[
Z(G\,\stackplus{v}\,H)\leq Z(G)+Z(H).
\]
To prove the equality in the equation suppose that 
\[
Z(G\,\stackplus{v}\,H)< Z(G)+Z(H).
\]
Hence we can colour $G\,\stackplus{v}\,H$ using $Z(G)+Z(H)-1$ initial black vertices. No vertex of $G$, except $v$,  can force any vertex of $H$ in any zero forcing process. In addition, since $v$ is not a member of any minimal ZFS of $G$, $v$ is not an end-point of any zero forcing chain in $G$. Therefore $v$ can not perform any force in $H$. Thus we need at least $Z(H)$ initial black vertices from $V(H)$ to colour $H$. Hence we have at most $Z(G)-1$ black vertices of $G$ along with $v$ (that is already turned to black) to colour $G$. This means there is a minimal ZFS for $G$ that contains $v$ which contradicts the fact that $v$ is not in any minimal ZFS of $G$.

To verify the second case we use a similar reasoning as in the first case. First using $Z'_H$ we can colour all of $H$.  If $Z_G$ is a minimal ZFS for $G$ that contains $v$ then $Z'_H\cup \left(Z_G\backslash \{v\}\right)$ forms a ZFS for $G\,\stackplus{v}\,H$. Thus 
\[
Z(G\,\,\stackplus{v}\,H)\leq Z(G)+Z(H)-1.
\]
In order to prove the equality, suppose that  
\[
Z(G\,\,\stackplus{v}\,H)< Z(G)+Z(H)-1.
\]
Therefore we can colour  $G\,\,\stackplus{v}\,H$ with an initial set of black vertices of size $Z(G)+Z(H)-2$. Since at most one of the vertices of $H$, i.e. $v$, can be forced by a vertex of $G$, we need at least $Z(H)-1$ black vertices from $V(H)$ to colour $H$. Therefore there are needed at most $Z(G)-1$ initial black vertices from $V(G)$ to finish colouring of $G$. This contradicts the fact that $Z(G)$ is the zero forcing number of $G$.

To prove the third case let $u$ be a vertex that is forced by $v$ in a zero forcing process started with the vertices of $Z_G$. Then  $Z_H\cup Z_G\cup \{u\}$ is a ZFS for $G\,\,\stackplus{v}\,H$. Thus  
\[
Z(G\,\,\stackplus{v}\,H)\leq Z(G)+Z(H)+1.
\]
Since  $v$ is neither in any minimal ZFS of $H$ nor in any minimal ZFS of $G$, in all  zero forcing processes starting with any minimal set of black vertices of $V(Z(G\,\,\stackplus{v}\,H))$, $v$ must perform a force both in $G$ and in $H$. Therefore in any zero forcing process, $v$ will have two white neighbours, one in $G$ and one in $H$.  Thus we can not complete the colouring of graph  $G\,\,\stackplus{v}\,H$ using at most $Z(G)+Z(H)$ initial black vertices and this completes the proof.        
\end{proof}
In order to find a connection between the zero forcing number and the path cover number of the vertex-sum of two graphs we try to calculate the path cover number of the vertex-sum of two graphs $G$ and $H$ in terms of those of  $G$ and $H$.
\begin{prop}\label{path covering of vertex-sum}
Let $G$ and $H$ be two graphs each with a vertex labeled $v$. Then the following hold.
\begin{enumerate}
\item If there is a minimal path covering of $G$ in which $v$ is a path of length zero and $v$ in not an end-point in any minimal path covering of $H$, then
\[
P(G\,\,\stackplus{v}\,H)=P(G)+P(H)-1.
 \]
\item If $v$ is an end-point of some minimal path covering of $G$ but no minimal path covering contains $v$ as a path of length $1$ and $v$ is not an end-point in any minimal path covering of $H$, then
\[
P(G\,\,\stackplus{v}\,H)=P(G)+P(H).
 \]
\item If there are minimal path coverings of $H$ and $G$ in which $v$ is an end-point of a path in both, then 
\[
P(G\,\,\stackplus{v}\,H)=P(G)+P(H)-1.
 \]

\item  If $v$ is neither an end-point of any minimal path covering of $H$ nor the end-point in any minimal path covering of $G$, then
\[
P(G\,\,\stackplus{v}\,H)=P(G)+P(H)+1.
 \]
\end{enumerate}
\end{prop}
\begin{proof}
 To prove the second case, let $\PP(G)$ and $\PP(H)$ be two minimal path coverings of $G$ and $H$ respectively. Let $P\in\PP(G)$ be the path for which $v$ is an end-point and $P'=P-v$. Then $\left(\PP(G)\backslash\{P\}\right)\cup\{P'\}\cup\PP(H)$ is a path covering for $G\,\,\stackplus{v}\,H$. Thus $P(G\,\,\stackplus{v}\,H)\leq P(G)+P(H)$. To show the equality, suppose that $P(G\,\,\stackplus{v}\,H)<P(G)+P(H)$. Hence there is a path covering for $G\,\,\stackplus{v}\,H$ with size $P(G)+P(H)-1$. Let $P_v$ be the path in a path covering of $G\,\,\stackplus{v}\,H$ with size $P(G)+P(H)-1$ that covers $v$. Three following cases are possible:
\begin{enumerate}
\item $P_v$ is entirely in $G$.
\item $P_v$ is entirely in $H$.
\item $P_v=P_1\,\,\stackplus{v}\,P_2$ such that $P_1$ and $P_2$ are two paths in $G$ and $H$ respectively.
\end{enumerate}
In the first case the vertices of $G\backslash P_v$ are covered by at least $P(G)-1$ paths. On the other hand, since the number of paths in this covering is at most $P(G)+P(H)-1$, the vertices of $H-v$ are covered by at most $P(H)-1$ paths. Therefore there is a path covering of $H$ of size $P(H)$ in which $v$ is an end-point. In the second case  the vertices of $H\backslash P_v$ are covered by at least $P(H)-1$ paths. Since the number of paths in this covering is at most $P(G)+P(H)-1$, the vertices of $G-v$ are covered by at most $P(G)-1$ paths. Thus there is a path covering of $G$ of size $P(G)$ in which there is a path of length $1$ consisting of $v$. Similarly for the third case the vertices of $G\backslash P_v$ are covered by at least $P(G)-1$ paths. On the other hand since the number of paths in this covering is at most $P(G)+P(H)-1$, the vertices of $H-P_v$ are covered by at most $P(H)-1$ paths. Therefore there is a path covering of $H$ of size $P(H)$ in which $v$ is an end-point. We have shown that in all cases we have a contradiction with the conditions in the original statement. Thus $P(G\,\,\stackplus{v}\,H)=P(G)+P(H)$.     

Using similar arguments, it is not hard to prove the other cases.  
\end{proof}

 In Proposition~\ref{P(T)=Z(T)}, it was shown that for trees, $P(T)=Z(T)$. Also for any cycle $C_n$ we have $P(C_n)=Z(C_n)=2$. In Corollary~\ref{unicycles}, we see that for any \textsl{unicycle} graph $G$, $P(G)=Z(G)$. On the other hand, any unicycle graph can be obtained by a sequence of vertex-sums of a cycle and a number of trees. Also it was shown that for block-cycle graphs $G$, $P(G)=Z(G)$. Furthermore, any block-cycle graph can be obtained by vertex-sums of a number of unicycle graphs. These observations along with Theorem~\ref{zero forcing of vertex-sum}, Proposition~\ref{path covering of vertex-sum} and Proposition~\ref{P(T)=Z(T)} lead us to the following conjecture:  
 
\begin{conj}\label{conj_Z&P_vertex_sum}
Let $G$ and $H$ be two graphs with $Z(G)=P(G)$ and $Z(H)=P(H)$. Then
\[
 Z(G\,\,\stackplus{v}\,H)=P(G\,\stackplus{v}\,H).
\]
\end{conj} 

Note that if we can show that the conditions in the cases of Theorem~\ref{zero forcing of vertex-sum} and Proposition~\ref{path covering of vertex-sum} coincide properly then we will prove the conjecture.  
\section{Join of graphs}
Next we turn our attention to the zero forcing number of the \txtsl{join} of two graphs and we give a formula to obtain the zero forcing number of the join of two graphs based on the zero forcing number of each of the initial graphs and their sizes. 
The join of two graphs is their union with all the edges $gh$ where $g\in V(G)$ and $h\in V(H)$. The join of graphs $G$ and $H$ is denoted by $G\vee H$. 

\begin{thm}\label{Z of join}
Let $G$ and $H$ be two connected graphs. Then
\[
Z(G\vee H)=\min\{|H|+Z(G), |G|+Z(H)\}.
\]
\end{thm}
\begin{proof}
Without loss of generality, assume that 
\[
\min\{|H|+Z(G), |G|+Z(H)\}=|H|+Z(G).
\]
Let $Z$ be a minimum zero forcing set for $G$. Clearly $V(H)\cup Z$ is a zero forcing set for $Z(G\vee H)$. Thus
\[
Z(G\vee H)\leq \min\{|H|+Z(G), |G|+Z(H)\}.
\]
We show that, in fact, it is not possible to colour $G\vee H$ with fewer vertices. Let $B$ be a minimal ZFS for $G\vee H$. Let $u\rightarrow v$ be the first force in the chronological list of forces. If $u,v\in G$, then $V(H)\subset B$ otherwise, $u$ can not perform any force. Also, since all the vertices of $H$ are adjacent to all the vertices of $G$, a vertex of $H$ can only force a vertex $w$ of $G$ if $w$ is the last white vertex of $G$ and no vertex of $G$ forces $w$, which contradicts with $G$ being connected.  Therefore, in order to colour the vertices of $G$ there are at least $Z(G)$ zero forcing chains in this zero forcing process starting with the vertices of a zero forcing set of $G$. Thus
\[
Z(G\vee H)\geq |H|+Z(G).
\]
If $u\in G$ and $v\in H$, then $V(H)\backslash \{v\}\subset B$. With a similar reasoning as above, no vertex of $H$ forces any vertex of $G$. Therefore, in order to colour the vertices of $G$ there are at least $Z(G)$ zero forcing chains in this zero forcing process starting with the vertices of a zero forcing set of $G$. In addition, all neighbours of $u$ in $G$ are black. This implies that $B\cap V(G)$ cannot be a minimal zero forcing set of $G$. Hence $|B\cap V(G)|> Z(G)$. Therefore
\[
Z(G\vee H)\geq |H|+Z(G).
\]
A similar argument applies when $u\in H$.
\end{proof}

{\bf Example.} $K_1\vee C_k$ with $k\geq 3$ is a wheel on $k+1$ vertices. According to Theorem~\ref{Z of join}, 
\[
Z(K_1\vee C_k)=\min\{1+2,k+1\}=3. 
\]  

\chapter{Positive Zero Forcing Sets}\label{PZFS}
\section{Positive zero forcing number}

 In this chapter we study the positive semidefinite zero forcing number which was introduced and its connection with the maximum positive semidefinite nullity was established  in \cite{MR2645093}. Also, we will establish some relations between the positive semidefinite zero forcing number and some graph parameters such as vertex connectivity and chromatic number. 
 
The \textsl{positive semidefinite colour change rule}\index{colour change rule!positive semidefinite} is as follows:  let $G$ be a graph and $B$ be the set consisting of all the black vertices of $G$. Let $W_1,\dots,W_k$ be the set of vertices of the $k$ components of $G\backslash B$ with $k\geq1$.  If $u\in B$ and $w$ is the only white neighbour of $u$ in the graph induced by $V(W_i\cup B)$, then change the colour of $w$ to black. 
The definitions and terminology for the positive zero forcing process, such as, colouring, derived set, etc., are similar to those for zero forcing number, but the colour change rule is different. 

A \txtsl{positive zero forcing set} $Z_p$ of $G$, is a subset of vertices of $G$ such that if initially the vertices in $Z_p$ are coloured black and the remaining vertices are coloured white, then the derived set of $G$ is $V(G)$. The \txtsl{positive zero forcing number} of $G$, denoted by $Z_+(G)$, is then defined to be the smallest size of a positive zero forcing set of $G$.
We abbreviate the term positive zero forcing set as PZFS. Note that for a non-empty graph $G$,
\[
1\leq Z_+(G) \leq |V(G)|-1.
\]
\begin{example}\label{pzfs-figures}
Let $G$ be the graph in Figure~\ref{start_graph_for_pzsf}. Then, Figure~\ref{finding_a_pzfs} illustrates why the set $Z_p=\{v_1,v_3,v_{10}\}$ is a PZFS for $G$. It is clear that this set is not a ZFS for $G$.

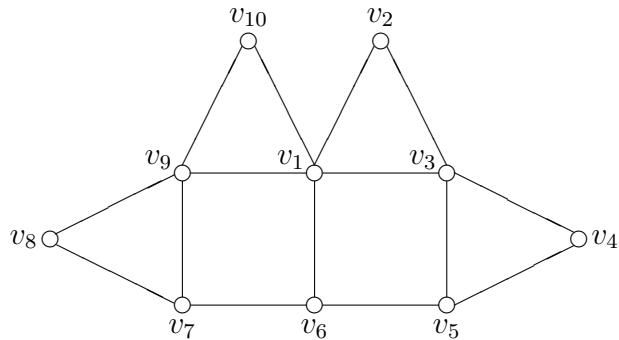
\begin{figure}[ht!]
\begin{center}
\begin{picture}(200,72)
\multiput(75,50)(50,0){2}{\circle{6}}
\multiput(50,0)(50,0){3}{\circle{6}}
\multiput(0,-25)(200,0){2}{\circle{6}}
\multiput(50,-50)(50,0){3}{\circle{6}}

\put(75,59){\makebox(0,0){$v_{10}$}}
\put(125,59){\makebox(0,0){$v_2$}}
\put(41,5){\makebox(0,0){$v_9$}}
\put(91,5){\makebox(0,0){$v_1$}}
\put(141,5){\makebox(0,0){$v_3$}}
\put(-10,-25){\makebox(0,0){$v_8$}}
\put(210,-25){\makebox(0,0){$v_4$}}
\put(50,-59){\makebox(0,0){$v_7$}}
\put(100,-59){\makebox(0,0){$v_6$}}
\put(150,-59){\makebox(0,0){$v_5$}}

\put(2,-23){\line(2,1){45}}
\put(2.5,-27){\line(2,-1){44.3}}
\put(53,0){\line(1,0){44}}
\put(53,-50){\line(1,0){44}}
\put(103,0){\line(1,0){44}}
\put(103,-50){\line(1,0){44}}
\put(50,-3){\line(0,-1){44}}
\put(100,-3){\line(0,-1){44}}
\put(150,-3){\line(0,-1){44}}
\put(153,0){\line(2,-1){45}}
\put(153,-50){\line(2,1){45}}
\put(50,3){\line(1,2){22.5}}
\put(100,3){\line(-1,2){22.5}}
\put(100,3){\line(1,2){22.5}}
\put(150,3){\line(-1,2){22.5}}
\end{picture}
\vspace{2cm}
\end{center}
\caption{A graph for which we want to find a PZFS}
\label{start_graph_for_pzsf}
\end{figure}

\begin{figure}
\begin{center}
\begin{picture}(500,0)
\put(75,50){\circle*{6}}
\put(125,50){\circle{6}}

\put(50,0){\circle{6}}
\multiput(100,0)(50,0){2}{\circle*{6}}
\multiput(0,-25)(200,0){2}{\circle{6}}
\multiput(50,-50)(50,0){3}{\circle{6}}

\put(75,59){\makebox(0,0){$v_{10}$}}
\put(125,59){\makebox(0,0){$v_2$}}
\put(41,5){\makebox(0,0){$v_9$}}
\put(91,5){\makebox(0,0){$v_1$}}
\put(141,5){\makebox(0,0){$v_3$}}
\put(-10,-25){\makebox(0,0){$v_8$}}
\put(210,-25){\makebox(0,0){$v_4$}}
\put(50,-59){\makebox(0,0){$v_7$}}
\put(100,-59){\makebox(0,0){$v_6$}}
\put(150,-59){\makebox(0,0){$v_5$}}

\put(2,-23){\line(2,1){45}}
\put(2.5,-27){\line(2,-1){44.3}}
\put(53,0){\line(1,0){44}}
\put(53,-50){\line(1,0){44}}
\put(103,0){\line(1,0){44}}
\put(103,-50){\line(1,0){44}}
\put(50,-3){\line(0,-1){44}}
\put(100,-3){\line(0,-1){44}}
\put(150,-3){\line(0,-1){44}}
\put(153,0){\line(2,-1){45}}
\put(153,-50){\line(2,1){45}}
\put(50,3){\line(1,2){22.5}}
\put(100,3){\line(-1,2){22.5}}
\put(100,3){\line(1,2){22.5}}
\put(150,3){\line(-1,2){22.5}}
\end{picture}
\begin{picture}(-100,0)
\thicklines
\put(-45,-15){\vector(4,-1){20}}
\thinlines

\put(75,50){\circle*{6}}
\put(125,50){\circle{6}}
\put(50,0){\circle{6}}
\multiput(100,0)(50,0){2}{\circle*{6}}
\multiput(0,-25)(200,0){2}{\circle{6}}
\multiput(50,-50)(50,0){3}{\circle{6}}

\put(125,50){\circle*{4}}
\put(50,0){\circle*{4}}

\put(75,59){\makebox(0,0){$v_{10}$}}
\put(125,59){\makebox(0,0){$v_2$}}
\put(41,5){\makebox(0,0){$v_9$}}
\put(91,5){\makebox(0,0){$v_1$}}
\put(141,5){\makebox(0,0){$v_3$}}
\put(-10,-25){\makebox(0,0){$v_8$}}
\put(210,-25){\makebox(0,0){$v_4$}}
\put(50,-59){\makebox(0,0){$v_7$}}
\put(100,-59){\makebox(0,0){$v_6$}}
\put(150,-59){\makebox(0,0){$v_5$}}

\put(2.5,-23){\line(2,1){44.6}}
\put(2.5,-27){\line(2,-1){44.3}}
\multiput(56,0)(15,0){6}{\line(1,0){10}}
\put(53,-50){\line(1,0){44}}
\put(103,-50){\line(1,0){44}}
\put(50,-3){\line(0,-1){44}}
\multiput(100,-5)(0,-15){3}{\line(0,-1){10}}
\multiput(150,-5)(0,-15){3}{\line(0,-1){10}}
\multiput(156,-3)(14,-7){3}{\line(2,-1){10}}
\put(153,-50){\line(2,1){45}}
\multiput(53,6)(7,14){3}{\line(1,2){5}}
\multiput(97,6)(-7,14){3}{\line(-1,2){5}}
\multiput(103,6)(7,14){3}{\line(1,2){5}}
\multiput(147,6)(-7,14){3}{\line(-1,2){5}}
\put(65,40){\vector(-1,-2){10}}
\put(105,20){\vector(1,2){10}}
\end{picture}

\begin{picture}(500,150)
\thicklines
\put(-50,-25){\vector(1,0){20}}
\thinlines

\put(75,50){\circle*{6}}
\put(125,50){\circle{6}}
\put(50,0){\circle{6}}
\multiput(100,0)(50,0){2}{\circle*{6}}
\multiput(0,-25)(200,0){2}{\circle{6}}
\multiput(50,-50)(50,0){3}{\circle{6}}

\put(125,50){\circle*{6}}
\put(50,0){\circle*{6}}
\put(100,-50){\circle*{4}}

\put(75,59){\makebox(0,0){$v_{10}$}}
\put(125,59){\makebox(0,0){$v_2$}}
\put(41,5){\makebox(0,0){$v_9$}}
\put(91,5){\makebox(0,0){$v_1$}}
\put(141,5){\makebox(0,0){$v_3$}}
\put(-10,-25){\makebox(0,0){$v_8$}}
\put(210,-25){\makebox(0,0){$v_4$}}
\put(50,-59){\makebox(0,0){$v_7$}}
\put(100,-59){\makebox(0,0){$v_6$}}
\put(150,-59){\makebox(0,0){$v_5$}}

\multiput(6,-22)(14,7){3}{\line(2,1){10}}
\put(2.5,-27){\line(2,-1){44.3}}
\multiput(56,0)(15,0){6}{\line(1,0){10}}
\put(53,-50){\line(1,0){44}}
\put(103,-50){\line(1,0){44}}
\multiput(50,-5)(0,-15){3}{\line(0,-1){10}}
\multiput(100,-5)(0,-15){3}{\line(0,-1){10}}
\multiput(150,-5)(0,-15){3}{\line(0,-1){10}}
\multiput(156,-3)(14,-7){3}{\line(2,-1){10}}
\put(153,-50){\line(2,1){45}}
\multiput(53,6)(7,14){3}{\line(1,2){5}}
\multiput(97,6)(-7,14){3}{\line(-1,2){5}}
\multiput(103,6)(7,14){3}{\line(1,2){5}}
\multiput(147,6)(-7,14){3}{\line(-1,2){5}}
\put(94,-15){\vector(0,-1){20}}
\end{picture}
\begin{picture}(-100,0)
\thicklines
\put(-45,-15){\vector(4,-1){20}}
\thinlines
\put(75,50){\circle*{6}}
\put(125,50){\circle{6}}
\put(50,0){\circle{6}}
\multiput(100,0)(50,0){2}{\circle*{6}}
\multiput(0,-25)(200,0){2}{\circle{6}}
\multiput(50,-50)(50,0){3}{\circle{6}}

\put(125,50){\circle*{6}}
\put(50,0){\circle*{6}}
\put(100,-50){\circle*{6}}
\put(50,-50){\circle*{4}}
\put(150,-50){\circle*{4}}

\put(75,59){\makebox(0,0){$v_{10}$}}
\put(125,59){\makebox(0,0){$v_2$}}
\put(41,5){\makebox(0,0){$v_9$}}
\put(91,5){\makebox(0,0){$v_1$}}
\put(141,5){\makebox(0,0){$v_3$}}
\put(-10,-25){\makebox(0,0){$v_8$}}
\put(210,-25){\makebox(0,0){$v_4$}}
\put(50,-59){\makebox(0,0){$v_7$}}
\put(100,-59){\makebox(0,0){$v_6$}}
\put(150,-59){\makebox(0,0){$v_5$}}

\multiput(6,-22)(14,7){3}{\line(2,1){10}}
\put(2.5,-27){\line(2,-1){44.3}}
\multiput(56,0)(15,0){6}{\line(1,0){10}}
\multiput(55,-50)(15,0){3}{\line(1,0){10}}
\multiput(105,-50)(15,0){3}{\line(1,0){10}}
\multiput(50,-5)(0,-15){3}{\line(0,-1){10}}
\multiput(100,-5)(0,-15){3}{\line(0,-1){10}}
\multiput(150,-5)(0,-15){3}{\line(0,-1){10}}
\multiput(156,-3)(14,-7){3}{\line(2,-1){10}}
\put(153,-50){\line(2,1){45}}
\multiput(53,6)(7,14){3}{\line(1,2){5}}
\multiput(97,6)(-7,14){3}{\line(-1,2){5}}
\multiput(103,6)(7,14){3}{\line(1,2){5}}
\multiput(147,6)(-7,14){3}{\line(-1,2){5}}
\put(85,-55){\vector(-1,0){20}}
\put(115,-55){\vector(1,0){20}}
\end{picture}

\begin{picture}(500,150)
\thicklines
\put(-50,-25){\vector(1,0){20}}
\thinlines
\put(75,50){\circle*{6}}
\put(125,50){\circle{6}}
\put(50,0){\circle{6}}
\multiput(100,0)(50,0){2}{\circle*{6}}
\multiput(0,-25)(200,0){2}{\circle{6}}
\multiput(50,-50)(50,0){3}{\circle{6}}

\put(125,50){\circle*{6}}
\put(50,0){\circle*{6}}
\put(100,-50){\circle*{6}}
\put(50,-50){\circle*{6}}
\put(150,-50){\circle*{6}}
\put(0,-25){\circle*{4}}
\put(200,-25){\circle*{4}}

\put(75,59){\makebox(0,0){$v_{10}$}}
\put(125,59){\makebox(0,0){$v_2$}}
\put(41,5){\makebox(0,0){$v_9$}}
\put(91,5){\makebox(0,0){$v_1$}}
\put(141,5){\makebox(0,0){$v_3$}}
\put(-10,-25){\makebox(0,0){$v_8$}}
\put(210,-25){\makebox(0,0){$v_4$}}
\put(50,-59){\makebox(0,0){$v_7$}}
\put(100,-59){\makebox(0,0){$v_6$}}
\put(150,-59){\makebox(0,0){$v_5$}}

\multiput(6,-22)(14,7){3}{\line(2,1){10}}
\multiput(6,-27)(14,-7){3}{\line(2,-1){10}}
\multiput(56,0)(15,0){6}{\line(1,0){10}}
\multiput(55,-50)(15,0){3}{\line(1,0){10}}
\multiput(105,-50)(15,0){3}{\line(1,0){10}}
\multiput(50,-5)(0,-15){3}{\line(0,-1){10}}
\multiput(100,-5)(0,-15){3}{\line(0,-1){10}}
\multiput(150,-5)(0,-15){3}{\line(0,-1){10}}
\multiput(156,-3)(14,-7){3}{\line(2,-1){10}}
\multiput(156,-47)(14,7){3}{\line(2,1){10}}
\multiput(53,6)(7,14){3}{\line(1,2){5}}
\multiput(97,6)(-7,14){3}{\line(-1,2){5}}
\multiput(103,6)(7,14){3}{\line(1,2){5}}
\multiput(147,6)(-7,14){3}{\line(-1,2){5}}
\put(33,-48){\vector(-2,1){18}}
\put(170,-48){\vector(2,1){18}}
\end{picture}
\begin{picture}(-100,0)
\thicklines
\put(-45,-15){\vector(4,-1){20}}
\thinlines
\multiput(75,50)(50,0){2}{\circle*{6}}
\multiput(50,0)(50,0){3}{\circle*{6}}
\multiput(0,-25)(200,0){2}{\circle*{6}}
\multiput(50,-50)(50,0){3}{\circle*{6}}

\put(75,59){\makebox(0,0){$v_{10}$}}
\put(125,59){\makebox(0,0){$v_2$}}
\put(41,5){\makebox(0,0){$v_9$}}
\put(91,5){\makebox(0,0){$v_1$}}
\put(141,5){\makebox(0,0){$v_3$}}
\put(-10,-25){\makebox(0,0){$v_8$}}
\put(210,-25){\makebox(0,0){$v_4$}}
\put(50,-59){\makebox(0,0){$v_7$}}
\put(100,-59){\makebox(0,0){$v_6$}}
\put(150,-59){\makebox(0,0){$v_5$}}

\put(2,-23){\line(2,1){45}}
\put(2.5,-27){\line(2,-1){44.3}}
\put(53,0){\line(1,0){44}}
\put(53,-50){\line(1,0){44}}
\put(103,0){\line(1,0){44}}
\put(103,-50){\line(1,0){44}}
\put(50,-3){\line(0,-1){44}}
\put(100,-3){\line(0,-1){44}}
\put(150,-3){\line(0,-1){44}}
\put(153,0){\line(2,-1){45}}
\put(153,-50){\line(2,1){45}}
\put(50,3){\line(1,2){22.5}}
\put(100,3){\line(-1,2){22.5}}
\put(100,3){\line(1,2){22.5}}
\put(150,3){\line(-1,2){22.5}}
\end{picture}
\vspace{2cm}
\end{center}
\caption{Finding a PZFS}
\label{finding_a_pzfs}
\end{figure}
\end{example}
We easily deduce the following  relationship between $Z_+(G)$ and $Z(G)$.
\begin{obs}
Since every zero forcing set is a positive zero forcing set, $Z_+(G)\leq Z(G)$.\qed
\end{obs}

Next we calculate the positive zero forcing number of trees. This result has been shown in \cite{MR2645093}.
\begin{thm}\label{PZFS.tree}
Let $T$ be a tree. Then $Z_+(T)=1$ and any vertex $v\in V(T)$ forms a PZSF.
\end{thm}
\begin{proof}
The proof is by induction on $|V(T)|$. It is clearly true for $|V(T)|=1$. Assume that the statement is true for any tree with $|V(T)|<n$. Let $T$ be a tree on $n$ vertices and let $v$ be any vertex in $T$. If $v$ is a leaf, it forces its neighbour; if not, deleting $v$ creates  components such that $v$ has only one white neighbour in each component. In either case, smaller tree(s) are obtained for which the induction hypothesis is applied.
\end{proof}

The following theorem shows that the maximum positive semidefinite nullity of a graph is bounded above by the positive zero forcing number of the graph \cite[Theorem 3.5]{MR2645093}.
\begin{thm}\label{M_+(G)<=Z_+(G)}
For any graph $G$, $M_+(G)\leq Z_+(G)$.
\end{thm}
\begin{proof}
Let $A\in \mathcal{H}_+(G)$ with  $\nul(A)=M_+(G)$. Let $\bfx=[x_i]$ be a nonzero vector in $\nul(A)$. Define $B$ to be the set of indices $u$ such that $x_u=0$ and let $W_1,\ldots, W_k$ be the vertex sets of the $k$ components of $G\backslash B$. We claim that in the graph induced by $V(B\cup W_i)$, any vertex $w\in W_i$ can not be the unique neighbour of any vertex $u\in B$. Once the claim is established, if $X$ is a positive semidefinite zero forcing set for $G$, then the only vector in $\nul(A)$ with zeros in positions indexed by $X$ is the zero vector, and by Proposition~\ref{vanishing_at_k_positions}, $M_+(G)\leq Z_+(G)$.
To establish the claim, renumber the vertices so that the vertices of $B$ are last, the vertices of $W_1$ are first, followed by the vertices of $W_2$, etc. Then $A$ has the block form

\[
A =
\left( {\begin{array}{ccccc}
 A_1 & 0 & \ldots & 0 & C^*_1  \\
     0 & A_2 & \ldots & 0 & C^*_2  \\
\vdots & \vdots & \ddots & \vdots\\
0 & 0 & \ldots & A_k & C^*_k\\
C_1 & C_2 & \ldots & C_k & D\\
 \end{array} } \right).
\]
\\
Partition vector $\bfx$ as $\bfx=[\bfx^*_1,\ldots, \bfx^*_k,0]^{\top}$, and note that all entries of $\bfx_i$ are nonzero with $i=1,\ldots, k$. Then $A\bfx=0$ implies $A_i\bfx_i=0$ with $i=1,\ldots, k.$ Since $A$ is positive semidefinite, each column in $C^*_i$ is in the span of the columns of $A_i$ by the column inclusion property of Hermitian positive semidefinite matrices \cite{MR1642835}. That is, for $i=1,\ldots, k$, there exists $Y_i$ such that $C^*_i=A_iY_i$. Thus $C_i\bfx_i=Y^*_iA_i\bfx_i=0$, and $w\in W_i$ can not be the unique neighbour in $W_i$ of any vertex $u\in B$. Otherwise $\bfx_i$ would be a length $1$ vector so $A_i\bfx_i=0$ implies $\bfx_i=0$.
\end{proof}

It has long been known (see \cite{MR2350678}) that for any tree $T$, $M_+(T)=1$, but the use of the positive zero forcing number provides an easy proof of this result; indeed  by Theorem~\ref{M_+(G)<=Z_+(G)} $M_+(T)=1$ is an immediate consequence of $Z_+(T)=1$.

The following lemma shows how it is possible to replace a member of a PZFS by a new vertex which is not in the PZSF.
\begin{lem}\label{swap lemma1}
Let $B$ be a PZFS for a graph $G$ whose deletion leaves components $C_1, \ldots, C_k$. Then there is at least one vertex in each $C_i$ for $i=1,\ldots, k$, which is the only white neighbour in $C_i$ of at least one vertex of $B$. Also, if $v\in C_i$ is the only white neighbour of $u\in B$ in $C_i$, then $B'=\left(B\backslash \{u\}\right)\cup \{v\}$ is a PZFS for $G$.
\end{lem}
\begin{proof}
 There is at least one vertex in each $C_i$ with $i=1,\ldots, k$, that is the only white neighbour in $C_i$ of some vertex of  $B$ because if there is no such a vertex in $B$ for a component, the derived set won't include the vertices of the component which contradicts the fact that $B$ is a PZFS for $G$.
 
Assume $v\in C_i$ is the only white neighbour of $u\in B$ in $C_i$. Set $B'=\left(B\backslash \{u\}\right)\cup \{v\}$. Then two cases are possible:
\begin{enumerate}[(a)]
\item$u$ has no neighbour in $G\backslash B$ except $v$; in this case after deletion of $B'$ we get the component $\{u\}$ as an isolated vertex. Thus $v$ forces $u$ and the rest of the positive zero forcing process will be the same as the positive zero forcing process initiated by the vertices in $B$ after forcing $v$.
\item $u$ has some other neighbours in $G\backslash B$ other than $v$. Clearly these neighbours are not in $C_i$. Let $C_{n_i}$ be the components with
\[
\left(N(u)\backslash \left(B\cup\{v\}\right)\right)\subseteq \bigcup_{i=1}^k C_{n_i}.
\]
Thus in the forcing process initiated by the black vertices in $B'$, $u$ is the only white neighbour of $v$ in the component induced by $V( \cup_{i=1}^k C_{n_i})\cup\{u\}$. Hence $v$ forces $u$ and the rest of the positive zero forcing process will be the same as the positive zero forcing process initiated by the vertices in $B$ after forcing $v$.\qedhere
\end{enumerate}
\end{proof}

\section{Forcing trees}\label{Forcing Trees}
We start this section by defining a concept analogous to zero forcing chains. Let $G$ be a graph and $Z_p$ be a positive zero forcing set of $G$. Construct the derived set, recording the forces in the order in which they are performed; this is the chronological list of forces. Note that two or more vertices can perform forces at the same time by applying the colour change rule once and a vertex can perform two or more forces at the same time (for example see step $4$ in Figure~\ref{finding_a_pzfs}).

For  any chronological list of forces, as above,   a \txtsl{multi-forcing chain} is an induced rooted tree, $T_r$, formed by a sequence of sets of vertices $(r,X_1,\dots, X_k)$, where $r\in Z_P$ is the root and the vertices in $X_i$ are at  distance $i$ from $r$ in the tree. The vertices of $X_i$ for $i=1,\dots, k$,  are forced by applying the positive semidefinite colour change rule to the vertices in $X_{i-1}$; so for any $v \in X_i$ there is a  $u \in X_{i-1}$, such that $u$ forces $v$ if and only if $v$ is a neighbour of $u$.
A multi-forcing chain is also called a \txtsl{forcing tree}. Also, the set of all multi-forcing chains in a positive zero forcing process is called the \txtsl{set of forcing trees}.

In a forcing tree, the vertices in $X_i$ are said to be the vertices of the $i$-th level in the tree. Note that the vertices in a specific level may have been forced in different steps of  the positive semidefinite colour change procedure and they may also perform forces in different steps.

\begin{example} 
The positive zero forcing number of graph $G$ in Figure~\ref{ex.forcing trees} is $3$ and $\{1,3,10\}$ is a PZFS for the graph. The forcing trees in this colouring procedure (as depicted in Figure~\ref{ex1.forcing trees}) are as follows:
\newline $T_1=\{1,X_1,X_2,X_3\}$, where  
$X_1=\{2,6\}$, $X_2=\{5,7\}$, $X_3=\{4,8\}$,\\
$T_3=\{3\}$,\\
$T_{10}=\{10, X_1\}$, where  $X_1=\{9\}$.
\begin{figure}
\centering
\vspace{.5cm}
\includegraphics[width=.6\textwidth]{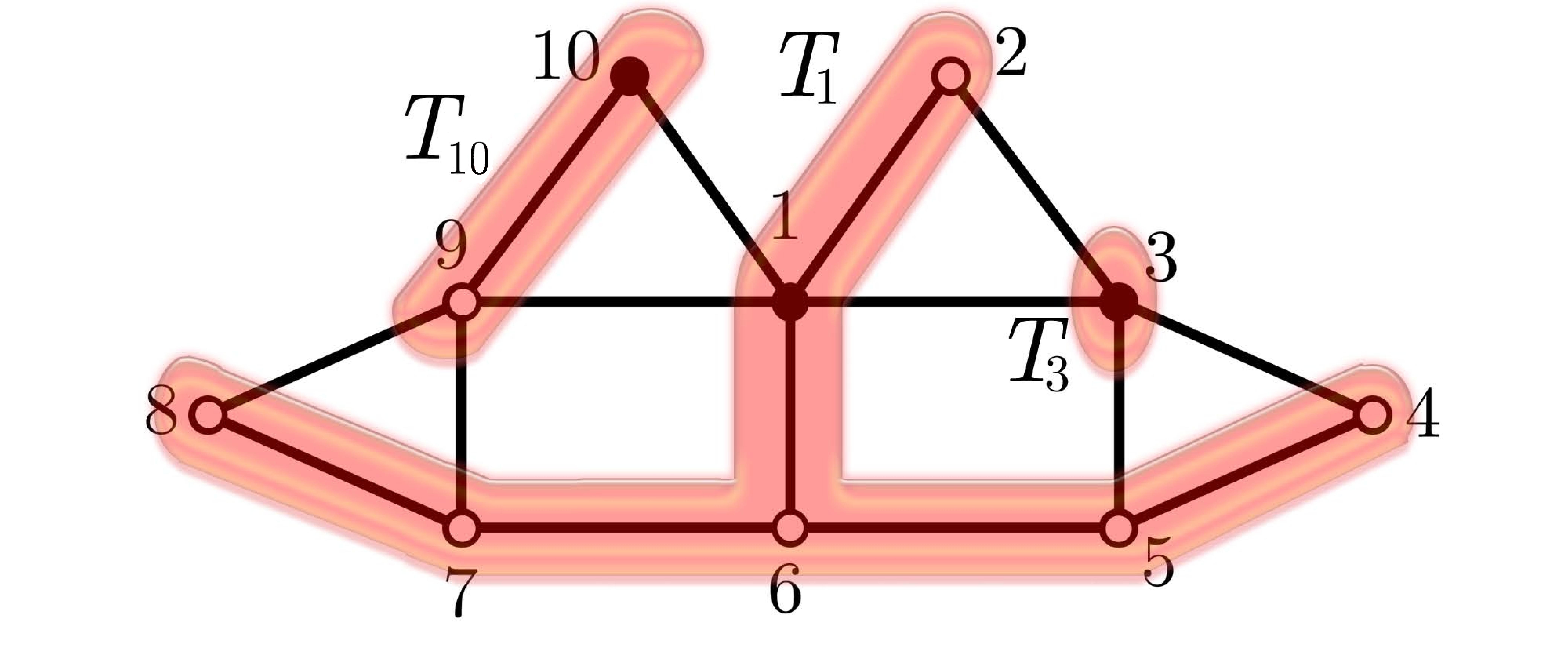}
\\
\caption{Graph $G$}
\label{ex.forcing trees}
\end{figure}

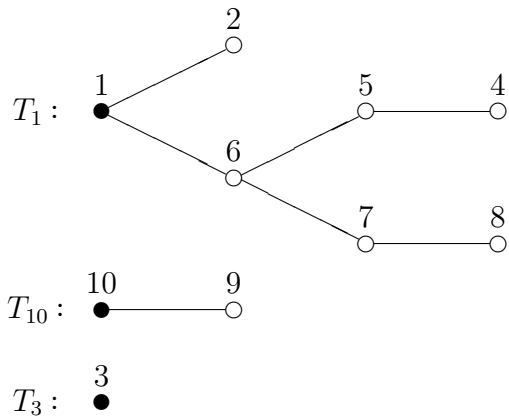
\begin{figure}
\begin{center}
\unitlength=1pt
\begin{picture}(150,50)
\put(0,0){\circle*{6}}
\put(50,25){\circle{6}}
\put(50,-25){\circle{6}}
\put(100,0){\circle{6}}
\put(150,0){\circle{6}}
\put(100,-50){\circle{6}}
\put(150,-50){\circle{6}}
\put(0,0){\line(2,1){47}}
\put(53,-24){\line(2,1){44.2}}
\put(0,0){\line(2,-1){47}}
\put(53,-26){\line(2,-1){44.2}}
\put(103,0){\line(1,0){44}}
\put(103,-50){\line(1,0){44}}
\put(0,-75){\circle*{6}}
\put(50,-75){\circle{6}}
\put(0,-75){\line(1,0){47}}
\put(0,-110){\circle*{6}}
\put(-25,0){\makebox(0,0){$T_1:$}}
\put(0,10){\makebox(0,0){$1$}}
\put(50,35){\makebox(0,0){$2$}}
\put(50,-15){\makebox(0,0){$6$}}
\put(100,-40){\makebox(0,0){$7$}}
\put(150,-40){\makebox(0,0){$8$}}
\put(100,10){\makebox(0,0){$5$}}
\put(150,10){\makebox(0,0){$4$}}
\put(-25,-75){\makebox(0,0){$ T_{10}:$}}
\put(0,-65){\makebox(0,0){$10$}}
\put(50,-65){\makebox(0,0){$9$}}
\put(-25,-110){\makebox(0,0){$ T_3:$}}
\put(0,-100){\makebox(0,0){$3$}}
\end{picture}\unitlength=1pt
\\
\vspace{4.5cm}
\end{center}
\caption{Forcing trees of $G$}
\label{ex1.forcing trees}
\end{figure}
\vspace{.5cm}
\end{example}

We now define the tree cover number which is a very interesting and useful graph parameter. Given a graph $G$, the \txtsl{tree cover number}, $\T(G)$, is the smallest positive integer $m$ such that there are $m$ vertex-disjoint induced trees  in $G$ that cover all the vertices of $G$.
In Example~\ref{ex.forcing trees}, the trees $T_1$, $T_{10}$ and $T_3$ form a tree covering of the graph. Also it can be shown that the tree cover number of the graph in Example~\ref{ex.forcing trees} is three. Note that any set of forcing trees corresponding to a minimal PZFS  is of size $Z_+(G)$  and covers all vertices of the graph. This implies the following.   
\begin{prop}\label{T(G)&Z_+(G)}
For any graph $G$, it holds that $\T(G)\leq Z_+(G)$.\qed
\end{prop}

Based on Proposition~\ref{T(G)&Z_+(G)}, the tree cover number is a lower bound for the positive zero forcing number. This bound is clearly tight for trees. But there are graphs, such as complete bipartite graphs, for which the discrepancy between these parameters is relatively large. It is, hence, an interesting  question to ask  for which families of graphs the equality between these two parameters holds. One way to approach this problem is  to find graph operations which preserve the equality in graphs for which these parameters agree.   

We have the following fact for the vertex-sum of any graph with any tree.
\begin{thm}\label{vertex-sum}
Let $G$ be any graph and $T$ be a tree both with a vertex labeled $v$ then the following equalities hold:
\begin{enumerate}[(a)]
\vspace{.2cm}
\item $Z_+(G\,\stackplus{v}\,T)=Z_+(G);$
\vspace{.2cm}
\item  $\T(G\,\stackplus{v}\,T)=\T(G).$
\vspace{.2cm}
\end{enumerate}
\end{thm}
\begin{proof}
For (a) assume that $B$ is a PZFS for $G$.  After applying the colour change rule with $B$ the set of  initial black vertices, all vertices in $G$ including $v$ get forced. According to Theorem~\ref{PZFS.tree}, $B$ is therefore a PZFS for $T$. Thus $Z_+(G\,\stackplus{v}\,T)\leq Z_+(G)$. To show that the reverse inequality also holds, let $B'$ be a PZFS for $Z_+(G\,\stackplus{v}\,T)$. If all members of $B'$ belong to $G$, then we are done. Otherwise at most one of the vertices of $B'$, say $u$,  is in $T$. Let $T_u$ be the forcing tree started by $u$ and $X_i$ be the first level of $T_u$ in which the vertices of $G$ appeared. Since $v$ is the only vertex in the intersection set of $V(G)$ and $V(T)$, $X_i\cap V(G)=\{v\}$. Thus $$Z_+(G\,\stackplus{v}\,T)\geq Z_+(G).$$
The statement (b) is trivial.
\end{proof}

\section{Positive zero forcing number and graph operations}

In this section we study the effects of the graph operations considered in Section~\ref{Simple operations} on the positive zero forcing number and show that it is not monotone on subgraphs. These results are in fact analogous to what we have already shown in Section~\ref{Simple operations} (for the proofs, see \cite[Proposition 5.3, Proposition 5.14, Proposition 5.22, Theorem 5.23]{ekstrand2013positive}).

First, we consider vertex deletion.
\begin{prop}
 Let $G$ be a graph and $v$ be a vertex in $V(G)$. Then 
 \[
Z_+(G)+\text{deg}(v)-1 \geq Z_+(G-v)\geq Z_+(G)-1.\qed
 \]
\end{prop}

Next, we consider edge deletion.

\begin{prop}
Let $G$ be a graph and $e$ be an edge in $G$. Then
\[
Z_+(G)-1\leq Z_+(G-e)\leq Z_+(G)+1.\qed
\]
\end{prop}

 Edge contraction has the following effect on the positive zero forcing number. 
\begin{prop}
 Let $e = uv$ be an edge in a graph $G$. Then  
\[
 Z_+(G)-1\leq Z_+(G/e)\leq Z_+(G)+1.\qed
 \]
\end{prop}
And finally the following result is about the effect of edge subdivision on the positive zero forcing number.  
\begin{prop}\label{subdivision and Z_+}
Let $H$ be the graph obtained from $G$ by subdividing an edge $e\in E(G)$. Then $Z_+(H) = Z_+(G)$ and any positive semidefinite zero forcing set for $G$ is a positive semidefinite zero forcing set for $H$.\qed 
\end{prop}

Because of the strong relationship between the tree cover number and the positive zero forcing number and  Proposition~\ref{subdivision and Z_+},  we turn our attention to the changes of the tree cover number of a graph after subdividing an edge.

\begin{thm}\label{subdivision and T(G)}
Let $H$ be the graph obtained from $G$ by subdividing an edge. Then
\[
T(G)=T(H).
\]
\end{thm} 
\begin{proof}
Assume $H$ is obtained by subdividing the edge $e=uv\in E(G)$. Hence $uw$ and $wv$ are the new edges in $H$. Let $\TT_G$ be a minimal tree covering for $G$. If $e$ is an edge of one of the trees, $T$, in  $\TT_G$, then subdividing $e$ in $T$ produces a new tree that covers $w$. Next consider when  $v$ and $u$ are covered by two different trees, say $T_1$ and $T_2$ respectively. In this case by extending $T_1$ to include the edge $vw$ we are able to cover $w$. Alternatively we could extend $T_2$ by $\{u,w\}$. Thus $T(H)\leq T(G)$. To establish the equality, suppose that $T(H)<T(G)$. Let $\TT_H$ be a minimal tree covering for $H$. By contracting $uw$ we obtain a tree covering for $G$ with the same size as $\TT_H$; this contradicts $T(G)$ being the tree cover number of $G$ and this completes the proof.  
\end{proof}
Combining Proposition~\ref{subdivision and Z_+} and Theorem~\ref{subdivision and T(G)} we have the following corollary. 
\begin{cor}\label{subdivision of a graph with Z_+=T}
Let $G$ be a graph with $Z_+(G)=T(G)$ and $H$ be a graph obtained from $G$ by subdividing some edges of $G$. Then
\[
Z_+(H)=T(H).\qed
\]
\end{cor}
Corollary~\ref{subdivision of a graph with Z_+=T} is a  useful tool to generate more families of graphs satisfying $Z_+(G)=T(G)$ out of graphs which satisfy this equality.  
\section{Outerplanar graphs}\label{Outerplanar Graphs}
Recall that a graph is a partial $2$-tree if it does not have a $K_4$ minor and that outerplanar graphs are exactly the graphs with no  $K_4$ and $K_{2,3}$ minors.
In \cite{barioli2011minimum} it is shown that $M_+(G) = T(G)$ for any outerplanar graph. This is extended to $Z_+(G) = M_+(G) = T(G)$ for any partial $2$-tree. This implies that  $Z_+(G) = M_+(G) = T(G)$ holds for any outerplanar graph. It is easy to see that every outerplanar graph is a partial $2$-tree.   In this section we give a different proof of the fact that the positive zero forcing number of an outerplanar simple graph agrees with its tree cover number. Moreover we show that any minimal tree covering of an outerplanar graph, $G$, coincides with a collection of forcing trees that contains exactly $Z_+(G)$ trees.
 
\subsection{Double paths and trees}
In this section we introduce two families of graphs: double paths and double trees. We establish some  properties for the PZFS of these graphs which we use to prove some important results in the next section.
If the vertices of an outerplanar graph with more than one face can be covered with two induced paths, then the graph is called a \txtsl{double path}. These paths are called the covering paths.
Similarly, if the vertices of an outerplanar graph with more than two faces can be covered with two induced trees, then the graph is called a \txtsl{double tree}.
Note that the subgraph $H$ of a double tree $G$ which is induced by the vertices on the boundary of the outer face of $G$ is a double path. Also $G\backslash H$ is a forest $\TT_H=\{T_1,T_2,\ldots, T_k\}$. In this case there is a subset $\{v_1,v_2,\ldots,v_k\}$ of $V(H)$, such that
\[
G=(\ldots((H\,\stackplus{v_1}\,T_1)\,\stackplus{v_2}\,T_2)\dots) \,\stackplus{v_k}\,T_k.
\] 
Since any vertex of a tree is a minimal PZFS of the tree, It is not hard to observe the following.
\begin{obs}\label{swap in double tree} 
Let $G$ and $H$ be as above and $u$ be any vertex in $V(T_i)\backslash \{v_i\}$ with, $1\leq i\leq k$. Assume $u\in B$. Then $B$ is a minimal PZFS of $G$ if and only if $\left(B\backslash \{u\}\right) \cup \{v_i\}$ is a minimal PZFS for $G$.\qed
\end{obs}

Now we show that for any double path starting with any vertex in one of the covering paths there is always another vertex in the other path such that these two vertices form a PZFS for the graph, and, using Theorem~\ref{PZFS.tree}, the positive zero forcing number of a double path is $2$.

In the following three lemmas we assume that $G$ is a double path with covering paths $P_1$ and $P_2$. Also we focus on a specific planar embedding of $G$. Once we have a planar embedding of $G$ we can refer to the end points of a covering path as the right end point and the left end point. 

\begin{lem}\label{lemma1} If $u$ and $v$ are both right (left) end points of  $P_1$ and $P_2$, respectively, then $\{u,v\}$ is a positive zero forcing set of $G$. 
\end{lem}
\begin{proof} 
 We prove this lemma by induction on the number of vertices. This is clearly true for $C_3$ which is a double path with fewest number of vertices.  Assume that it is true for all graphs $H$ with $|V(H)|<n$. Let $G$ be a graph of size $n$. Assume that $u$ and $v$ are left end points of $P_1$ and $P_2$, respectively. By assigning the colour black to each of these vertices we claim that $\{u,v\}$ is a PZFS of $G$. If $u$ is a pendant vertex, then it forces its only neighbour, say $w$, which is a left end point of a covering path in $G\backslash u$. Thus by the induction hypothesis $\{w,v\}$ is a PZSF of $G$. Similarly if $v$ is a pendant vertex, using a similar reasoning, the lemma follows. If neither $u$ nor $v$ are pendant then $u$ and $v$ are adjacent and since both are on the same side, at least one of them, say $u$, is of degree two. Let $w$ be the only neighbour of $u$ in $P_1$. Thus $u$ can force $w$ (its only white neighbour) and again by the induction hypothesis $\{w,v\}$ is a PZSF of $G$. The same reasoning applies when $u$ and $v$ are the right end points of $P_1$ and $P_2$, respectively.    
\end{proof}

\begin{lem}\label{lemma2}
 If $u$ and $v$ are two vertices of $P_1$ and $P_2$, respectively, which form a cut set for $G$, then $\{u,v\}$  is a positive zero forcing set for $G$.
\end{lem}
\begin{proof}
Let $W_1\subseteq V(G)$ and $W_2\subseteq V(G)$ be the vertices of the left hand side and the right hand side components of $G\backslash \{u,v\}$ respectively and let $G_1$ and $G_2$ be the subgraphs induced by $\{u,v\}\cup W_1$ and $\{u,v\}\cup W_2$ respectively. Then according to Lemma~\ref{lemma1}, $\{u,v\}$ is a PZFS for both $G_1$ and $G_2$ and this completes the proof.  
\end{proof}
\begin{lem}\label{lemma3}
 If $u$ is a vertex in $P_1$ which is not an end point, then there always is a vertex $v$ in $P_2$ such that $\{u,v\}$ is a cut set of $G$.
\end{lem}
\begin{proof}
Suppose there is a vertex $u$ in $P_1$ for which there is no vertex $v$ in $P_2$ such that $\{u,v\}$ is a cut set for $G$.
Since $G$ is an outerplanar graph, it has an embedding in the plane such that all vertices are on the same face.
 Let $v$ be any non pendant vertex of $P_2$. Obviously $u$ is a cut vertex of $P_1$ and $v$ is a cut vertex of $P_2$. Since $\{u,v\}$ is not a cut set of $G$, there is a vertex in the left hand side (or right hand side) of $u$ that is adjacent to a vertex in the right hand side (or left hand side) of $v$. Assume that $l$ is the farthest vertex from $v$ in $P_2$ having this property. Since $l$ is the farthest vertex from $v$ with the described property and $G$ is an outerplanar graph, $\{u,l\}$ is a cut set of $G$ which contradicts the fact that, there is no vertex in $P_2$ that forms a cut set along with $u$ for $G$.
\end{proof}

 Combining Lemmas \ref{lemma1}, \ref{lemma2} and \ref{lemma3} along with the fact that in the proof of all these three  lemmas forces are performed along covering paths, we have  the following.
\begin{cor}\label{PZFS of double path}
Let $G$ be a double path with covering paths $P_1$ and $P_2$. Then for any vertex $v$ in $P_1$, there is always another vertex $u$ in $P_2$ such that $\{u,v\}$ is a positive zero forcing set for $G$. Moreover the two paths $P_1$ and $P_2$ are a minimal set of forcing trees in $G$.\qed
\end{cor}

The following corollary is a consequence of Observation ~\ref{swap in double tree} and Corollary~\ref{PZFS of double path}. 
\begin{cor}\label{PZFS of double tree}
Let $G$ be a double tree with covering trees $T_1$ and $T_2$.  Then for any vertex $v$ in $T_1$, there is always another vertex $u$ in $T_2$ such that $\{u,v\}$ is a positive zero forcing set for $G$. Moreover $\{T_1,T_2\}$ coincides with a minimal collection of forcing trees in $G$.\qed
\end{cor}
\subsection{Pendant trees}
In this section we show that for any outerplanar graph there is a minimal tree covering in which there is a \textsl{pendant tree}.

In a tree covering of a graph, two trees, $T_1$ and $T_2$, are said to be \txtsl{adjacent} if there is at least one edge $uv\in E(G)$ such that  $v \in V(T_1)$ and $u \in V(T_2)$.
In Figure~\ref{ex.forcing trees},  $T_1$ is adjacent to $T_{10}$ and $T_3$ but $T_{10}$ and $T_3$ are not adjacent. A tree is called \txtsl{pendant} if it is  adjacent to only one other tree. In Example~\ref{ex.forcing trees}, $T_{10}$ and $T_3$ are two pendant trees.

Let $G$ be an outerplanar graph with a planar embedding in which all the vertices are on the same face. Let $\TT(G)$ be a minimal tree covering for $G$. Define $H_{\TT}$ to be the graph whose vertices correspond to the elements of $\TT(G)$ in which two vertices are adjacent if there is an outer edge between their corresponding trees in the graph $G$. Two trees of $\TT(G)$ are called \txtsl{consecutive}, if their corresponding vertices in $H_{\TT}$ are adjacent vertices each of degree two.

\begin{thm}\label{consecutive trees}
Let $G$ be an outerplanar graph and $\TT(G)$  a minimum tree covering for $G$. If there is no pendant tree in $\TT(G)$, then there are at least two consecutive trees in $\TT(G)$.
\end{thm}
\begin{proof}
Assume that $\TT(G)$ is a minimum tree covering for $G$ in which there is no pendant tree. Then, two cases are possible:
\newline
{\bf Case 1.} There is no tree in $\TT(G)$ with at least one of the inner edges of $G$ in its edge set (see Figure~\ref{fig_for_H_T} for an example of such a graph). Therefore $H_{\TT}$ is a cycle. Accordingly, any adjacent pair of trees in $\TT(G)$ are consecutive.
\begin{figure}[ht!]
\vspace{.5cm}
\centering
\includegraphics[width=.8\textwidth]{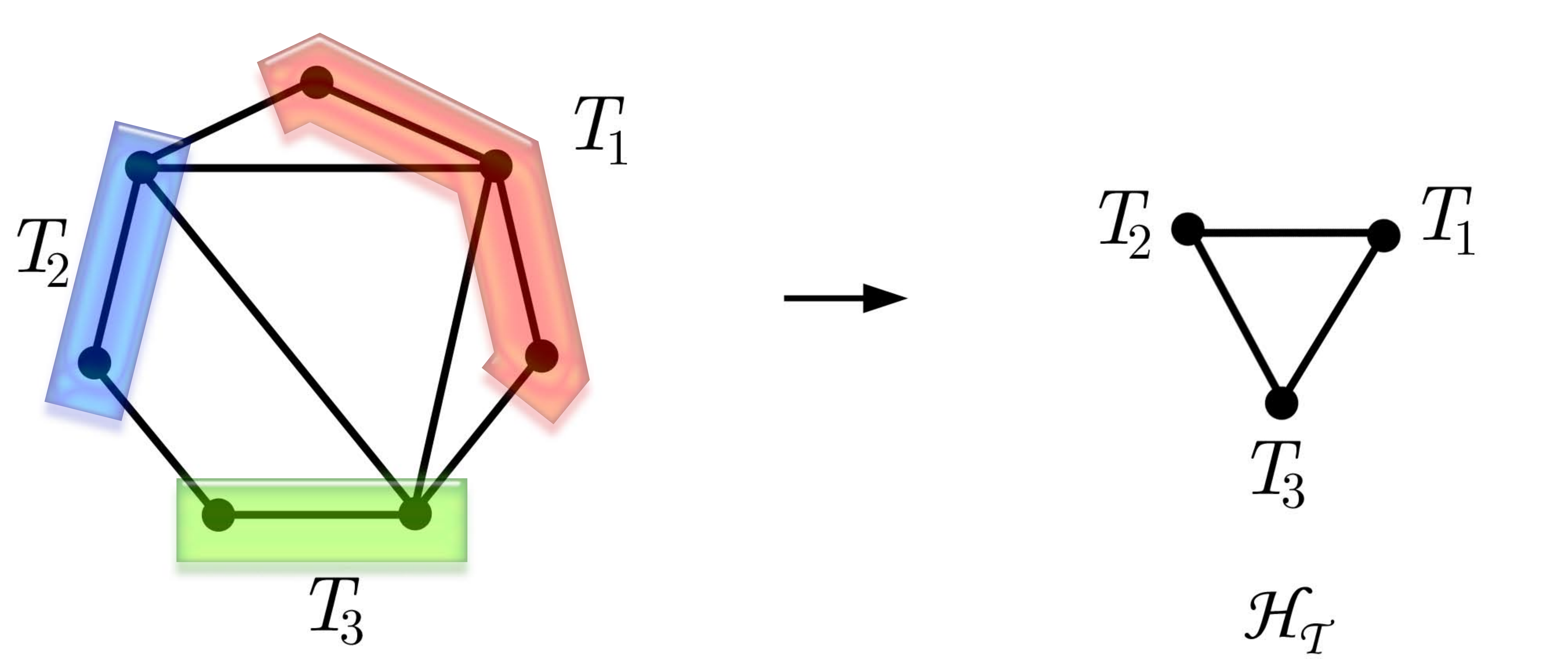}
\caption{Forcing trees with no inner edge in their edge set}
\label{fig_for_H_T}
\end{figure}

\begin{figure}[ht!]
\vspace{.5cm}
\centering
\includegraphics[width=.8\textwidth]{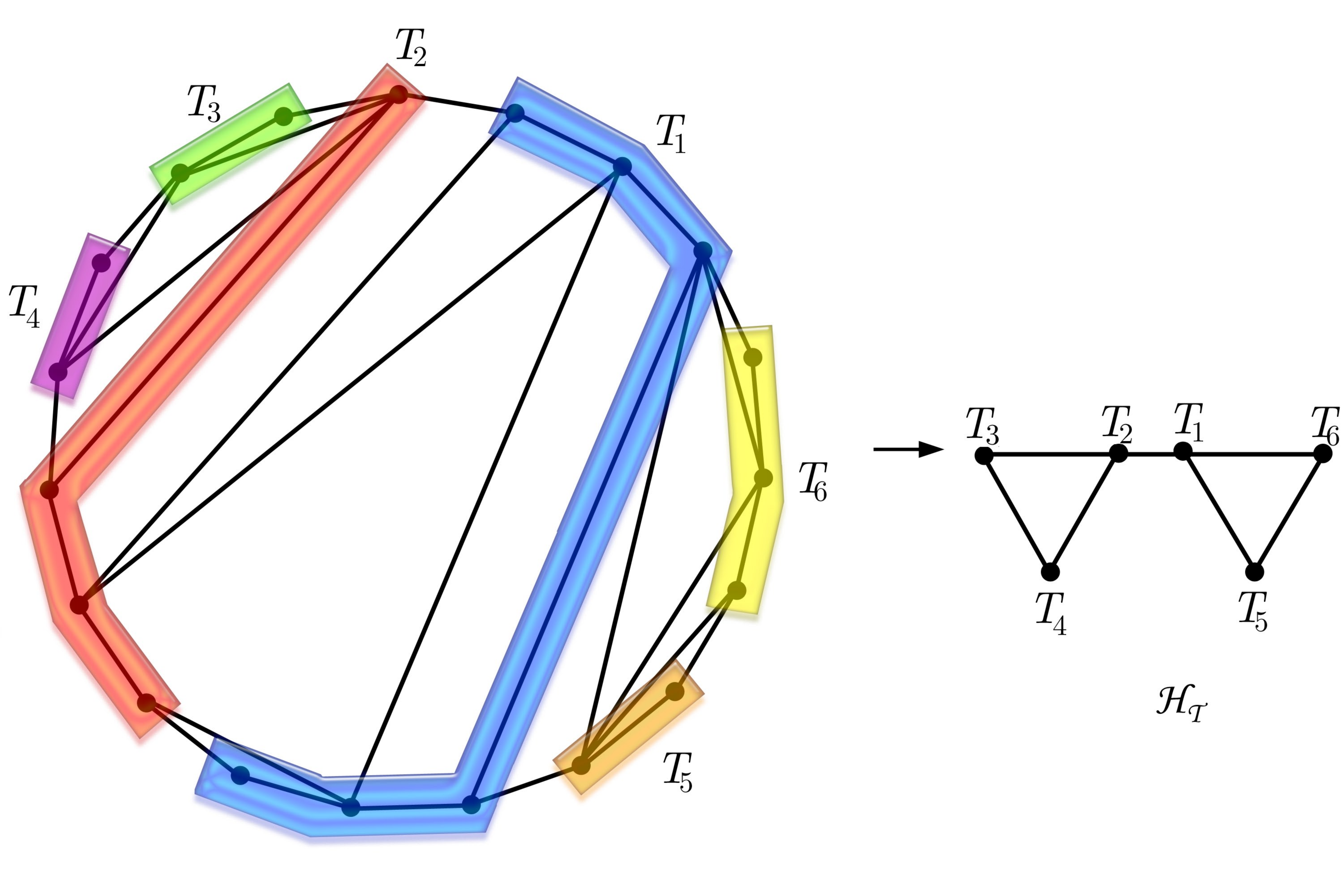}
\caption{$T_1$ is a forcing tree with an inner edge in its edge set}
\label{fig_for_H_T_case2}
\end{figure}
{\hspace{-.64cm}\bf Case 2.} There is at least one tree in $\TT(G)$ having at least one of the inner edges of $G$ in its edge set.  Any inner edge $e=uv\in E(G)$ partitions the plane into two parts and, consequently, the set $V(G)\backslash \{u,v\}$ into two subsets $V_{u}$ and $V_{v}$. For any inner edge $uv$, let $V_f$ be the subset of all vertices in $V(G)\backslash\{u,v\}$  which are not end-points of an inner edge that is included in a tree of $\TT(G)$. Since $G$ is outerplanar and finite, there exists an edge, $e=uv$, such that  at least one of $V_{u}$ or $V_{v}$ is in $V_f$. Therefore there are at least two consecutive trees in $\TT(G)$ covering the vertices of $V_f$. In Figure~\ref{fig_for_H_T_case2} the pairs $T_3,T_4$ and $T_5,T_6$ are examples of consecutive pairs of trees in $\TT(G)$.    
 \end{proof}
 
 The following theorem plays a key role in the proof of the fact that outerplanar graphs satisfy $Z_+(G)=T(G)$.  
\begin{thm}\label{pendant tree}
Let $G$ be an outerplanar graph. Then there is a minimum tree covering for $G$ in which there is a pendant tree.
\end{thm}
\begin{proof}
Since $G$ is an outerplanar graph, it has an embedding in the plane such that all of its vertices are on the same face. Assume that $\TT(G)$ is a minimum tree covering for $G$ in which there is no pendant tree. We use Theorem~\ref{consecutive trees} to construct a new tree covering $\TT'(G)$ of  $\TT(G)$  with $|\TT'(G)|=|\TT(G)|$ in which there is a pendant tree. 

 Consider two  trees $T_1$ and $T_2$ in $\TT(G)$, which are consecutive and assume that $T_1$ is in the right hand side of $T_2$. Let  $H$ be the graph induced by $V(T_1)\cup V(T_2)$ (Figure~\ref{pendant_tree_construction}). There are two outer edges in $H$ that have an end-point from each of trees $T_1$ and $T_2$. One of these outer edges, call it $e=uv$, is an inner edge in $G$ with  $u \in T_1$ and $v \in T_2$.
 If $v$ has any other neighbour in  $T_1$, then $u$ has no other neighbour in $T_2$;  otherwise there will be either two crossing edges in the planar embedding of $G$ or an edge beyond $e$; either case is a contradiction. Thus $u$ along with all vertices on its right hand side in $T_1$ and all the vertices in $T_2$ induce a new tree $T_i$ and the  vertices in the left hand side of $u$  in  $T_1$  induce another new tree $T_k$  which is a subtree of $T_1$. Now we have a new tree covering $\TT'(G)$ for $G$ with $|\TT'(G)|=|\TT(G)|$ in which $T_k$  is a pendant tree. 
 
A similar argument applies when $u$ has another neighbour in $T_2$.
If neither $u$ nor $v$ has any other neighbour in $T_1$ and $T_2$, respectively, then either cases mentioned above are applicable.
\end{proof}
\subsection{Outerplanar graphs satisfy $Z_+(G)=T(G)$}
We now turn our attention to the positive zero forcing number of outerplanar graphs and show that this parameter for any graph among this family of graphs is equal to the tree cover number of the graph, moreover any minimal tree covering of such a graph coincides with a minimal collection of forcing trees.

\begin{thm}\label{outerplanars satisfy Z_+=T}
Let $G$ be an outerplanar graph. Then
$$Z_+(G)=\T(G).$$
Also any minimal tree covering of the graph $\TT(G)$ coincides with a collection of forcing trees with $|\TT(G)|=Z_+(G)$. 
\end{thm}

\begin{proof}
  We prove the claim by induction on the tree cover number and using  Theorem~\ref{pendant tree}. It is clearly true for $\T(G)=1$. Assume that it is true for any outerplanar graph $G'$ with $\T(G') < k$. Now let $G$ be an outerplanar graph with $\T(G)=k$. By Proposition~\ref{T(G)&Z_+(G)}, we have  $ Z_+(G) \geq \T(G)$. According to Theorem~\ref{pendant tree} there is a minimum tree covering of $G$, $\TT(G)$, containing $T_1,T_2, \dots, T_k$  in which there is a pendant tree, say $T_k$. Let $T_{k-1}$ be the only neighbour of $T_k$. Let $G'=G\backslash V(T_k)$ then the induction hypothesis holds, so $T(G')=Z_+(G')=k-1$. Further,  $T_1,T_2, \dots, T_{k-1}$ are forcing trees in a positive zero forcing process whose initial set of black vertices, $Z'_p$, has a vertex from each tree in $\TT(G)\backslash T_k$. A zero forcing process in $G$ starting with the black vertices in $Z'_p$ can proceed as it does in $G'$ until the
 first vertex of $T_{k-1}$, say $x$, that is adjacent to some vertex in $T_k$ gets forced. Since the graph induced by $V(T_{k-1})\cup V(T_k)$ is a double tree, according to Corollary \ref{PZFS of double tree}, the vertex $x$ determines a vertex $y$ in $T_k$ such that $\{x,y\}$ is a PZFS for the subgraph, $W$, induced by $V(T_{k-1})\cup V(T_k)$. Since the induction hypothesis holds for $W$, $T_{k-1}$ is a forcing tree in $W$ too. Thus the vertices of $T_{k-1}$ get forced in the same order as they get forced in $G'$.   
 Therefore we can complete colouring the graph $G$ by adding the black vertex $y$ to the initial set of black vertices.
 Thus, $Z_p=Z'_p\cup \{y\}$ is a PZFS of $G$ with $T_1,T_2 \dots, T_k$ as the forcing trees in this positive zero forcing process. Thus, $Z_+(G)= \T(G)$.

Now we need to show that the original minimal tree covering of the graph $G$ also coincides with a collection of forcing trees. Note that in the procedure of constructing a pendant tree in the minimal tree covering of $G$, we modified two consecutive trees $T_1$ and $T_2$ to obtain a pendant tree $T_k$ in a new minimal forcing tree covering. Assume that $v\in T_2$ has a neighbour, other than $u$, in $T_1$ as  is shown in Figure~\ref{pendant_tree_construction}.

\begin{figure}[H]
\centering
\includegraphics[width=.6\textwidth]{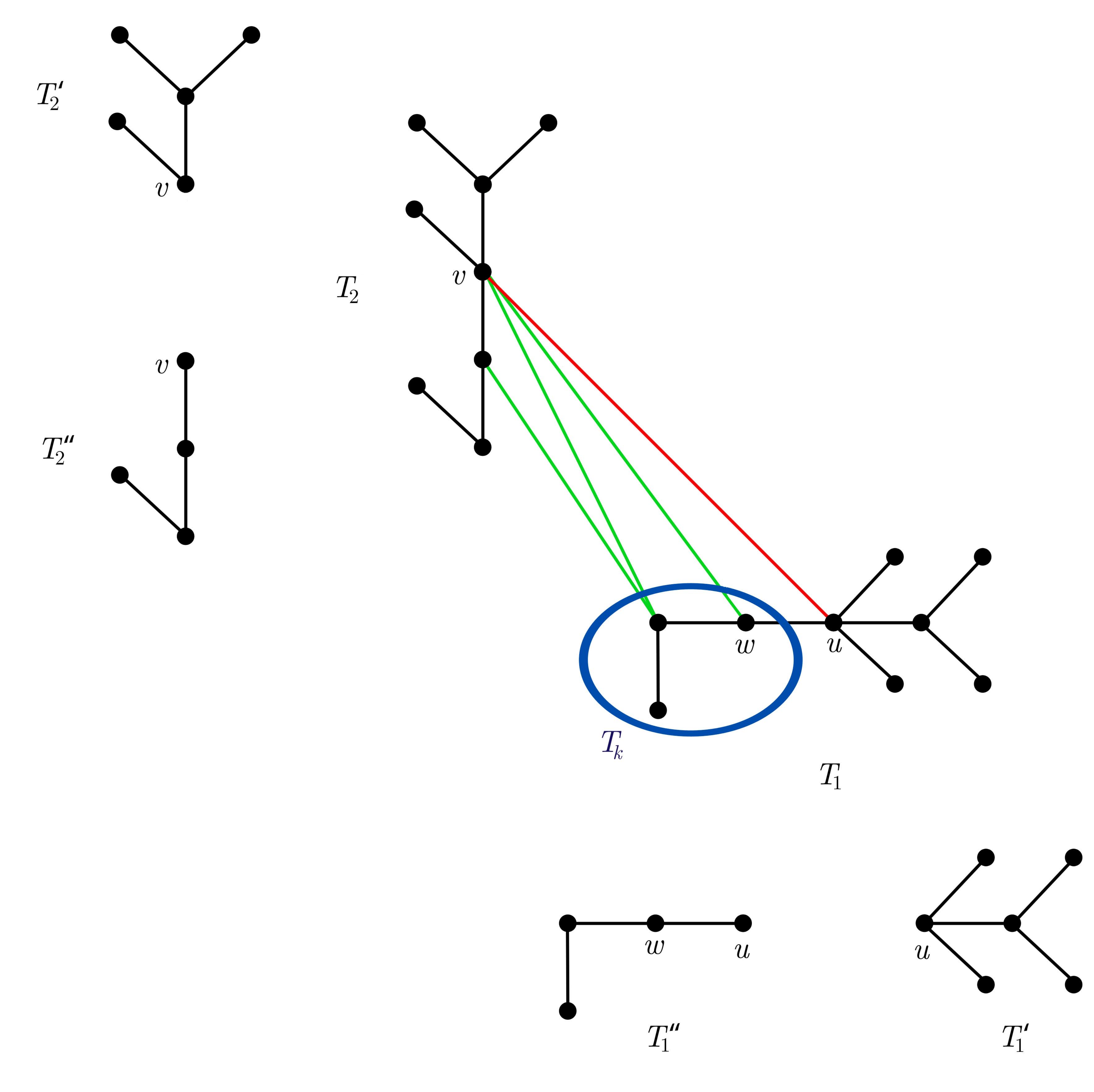}
\\
\caption{Construction of a pendant tree in an outerplanar graph}
\label{pendant_tree_construction}
\end{figure}

Let $T_1=T_1'\,\stackplus{u}\,T_1''$, where $T_1''$ involves all the vertices of $T_1$ that have some neighbours in $T_2$. Similarly, let $T_2=T_2'\,\stackplus{v}\,T_2''$, where $T_2''$ involves all the vertices of $T_2$ that have some neighbours in $T_1$.    
Then the tree adjacent to $T_k$ is the tree
\[
T_{k-1}=T_2\,\stackplus{vu}\,T_1'.
\]
 The tree $T_k$ is obtained by removing $u$ from $T_1''$. In other words $T_1=T_k\,\stackplus{wu}\, T_1'$ with $w\in V(T_k)$.

Now we show that, one can  change the direction of forces in this positive zero forcing process, so that instead of forcing along the trees $T_k$ and $T_{k-1}$, the forces are performed along the original trees $T_1$ and $T_2$; consequently, we show that the original tree covering of the graph is in fact a set of forcing trees for the graph.

Consider the positive zero forcing process starting with the vertices of $Z'_p$. Let $x$ be the first black vertex of  $T_{k-1}$ that is adjacent to a vertex in $T_k$.

 If $x\in T_2$, then  according to Corollary~\ref{PZFS of double tree}, the vertex $x$ determines a (not necessarily unique) vertex $y$ in $T_k$ such that $Z'_p\cup\{y\}$ is a PZFS for $G$. In this case the level of vertex $v$ is lower than the level of vertex $u$ in the forcing tree $T_{k-1}$. Thus $v$ forces $u$. The set $\{x,y\}$ is a PZFS for the graph $H=G[V(T_2'')\cup V(T_1'')]$ which is a double tree. Hence by the induction hypothesis $\{T_1'',T_2''\}$ is a set of forcing trees for $H$. Thus the colouring of vertex $u$ can be performed by vertex $w$ instead of vertex $v$. From this we have that $T_1'\,\stackplus{u}\,T_1''=T_1$ with the root $y$ and $T_2$ with the root $x$ are two forcing trees in a positive zero forcing process of $G$. 
 
Next, assume that $x$ belongs to $T_1'$. In this case according to Lemma~\ref{lemma1} and Observation~\ref{swap in double tree}, the set $Z'_p\cup\{w\}$ is a PZFS for $G$. Also in this case the level of vertex $u$ is lower than the level of vertex $v$ in the forcing tree $T_{k-1}$. Thus $u$ forces $v$. But since $u$ has only two white neighbours in the component $G[V(T_k)\cup V(T_2'')]$ which are $w$ and $v$, we can replace $w$ by $v$ in the PZFS of $G$. Then $u$ forces $w$ and the rest of the forces will be performed along $T_1$ and $T_2$. Thus $T_1$ with the root $x$ and $T_2$ with the root $v$ are two forcing trees in a positive zero forcing process of $G$.     
\end{proof}
The following is a consequence of combining Theorem~\ref{outerplanars satisfy Z_+=T} and Corollary~\ref{subdivision of a graph with Z_+=T}.
\begin{cor}
Any subdivision of an outerplanar graph (which may no longer be outerplanar) satisfies
\[
Z_+(G)=T(G).\qed
\]
\end{cor}

It seems that the scope of graphs satisfying $Z_+(G)=T(G)$ goes well beyond outerplanar graphs and even $2$-trees. The graph in Figure~\ref{non_outer_planar} is an example of a graph that is not a partial $2$-tree but still satisfies $Z_+(G)=T(G)$. This graph includes a $K_4$ as a subgraph therefore $Z_+(G)\geq 3$. But the set $\{1,2,3\}$ is a PZFS for $G$. Thus $Z_+(G)=3$. It is not hard to verify that the tree cover number of this graph is no less than $3$ and since the positive zero forcing number is an upper bound for the tree cover number we have $T(G)=3$ as well.    

\begin{figure}[H]
\centering
\unitlength=1pt
\begin{picture}(0,0)
\put(0,0){\circle*{6}}
\put(-25,-50){\circle*{6}}
\put(25,-50){\circle{6}}
\put(-75,-75){\circle*{6}}
\multiput(-25,-100)(50,0){2}{\circle{6}}
\put(0,-150){\circle{6}}

\put(-23.5,-52.5){\line(1,-1){45.7}}
\put(23.45,-52.55){\line(-1,-1){45.7}}
\put(-1.7,-2.7){\line(-1,-2){22.3}}
\put(-22.2,-50){\line(1,0){44}}
\put(23,-102.4){\line(-1,-2){22.25}}

\put(-75,-75){\line(2,-1){47.2}}
\put(-75,-75){\line(2,1){47.2}}

\put(1.6,-2.6){\line(1,-2){22.3}}
\put(-1.6,-2.6){\line(-1,-2){22.4}}

\put(-25,-50){\line(1,-4){24.25}}
\put(-22.2,-100){\line(1,0){44}}
\put(-25,-52.2){\line(0,-1){45}}
\put(25,-52.8){\line(0,-1){44}}

\put(-83,-75){\makebox(0,0){$1$}}
\put(-33,-45){\makebox(0,0){$2$}}
\put(8,0){\makebox(0,0){$3$}}
\put(33,-45){\makebox(0,0){$4$}}
\put(33,-105){\makebox(0,0){$5$}}
\put(-33,-105){\makebox(0,0){$6$}}
\put(8,-150){\makebox(0,0){$7$}}

\end{picture}
\vspace{6cm}
\caption{A non-outerplanar graph with $Z_+(G)=T(G)$}
\label{non_outer_planar}
\end{figure}

\section{Positive zero forcing and the tree cover number}
Since the positive zero forcing number and tree cover number are very nicely related it seems reasonable to study other families of graphs satisfying $Z_+(G)=T(G)$. In this section we try to find other graph operations that preserve the equivalency of these parameters for the graphs in which these parameters already agree. Moreover we will compare these parameters for $k$-trees. 

\subsection{Vertex-sum of two graphs}\label{Vertex-sum Of Two Graphs}
\begin{thm}\label{T&Z_+ of vertex sum}
For any graphs $G$ and $H$ with a vertex labeled $v$ 
\[
T(G\,\stackplus{v}\,H)=T(G)+T(H)-1,
\]
and
\[
Z_+(G\,\stackplus{v}\,H)=Z_+(G)+Z_+(H)-1.
\]
\end{thm}
\begin{proof}
In order to prove the first equality let $\TT_G$ and $\TT_H$ be the minimal tree coverings of $G$ and $H$, respectively, and $T_1\in \TT_G$ and $T_2\in \TT_H$ be the trees covering $v$. Let $T_v=T_1 \,\stackplus{v}\, T_2$. Observe that $T_v$ is an induced tree in $G\,\stackplus{v}\,H$ that covers $v$. Therefore, $\left(\left(\TT_G \cup \TT_H\right) \backslash \{T_1,T_2\}\right)\cup T_v$ is a tree covering of $G\,\stackplus{v}\,H$. 

Next we show, in fact, the vertices of $G\,\stackplus{v}\,H$ can not be covered with fewer trees.  Suppose that   the vertices of $G\,\stackplus{v}\,H$ can be covered with $T(G)+T(H)-2$ induced disjoint trees.  Let $T$ be the tree that covers $v$ in such a tree covering of $G\,\stackplus{v}\,H$. At least $T(G)-1$ trees are needed to cover the vertices of $G\backslash T$. Since the number of trees in the covering of $G\,\stackplus{v}\,H$ is  $T(G)+T(H)-2$, the vertices of $H\backslash T$ are covered by at most ($T(H)-2$) trees, but this is a contradiction with the fact that $T(H)$ is the least number of trees that cover the vertices of $H$. Thus  
\[
T(G\,\stackplus{v}\,H)=T(G)+T(H)-1.
\]
To prove the second equality, let $\TT_G$ and $\TT_H$ be the sets of forcing trees for a minimal PZFS in $G$ and $H$ respectively. Let $T_1\in \TT_G$ and $T_2\in \TT_H$ be the trees that contain $v$. Then $T_1\,\stackplus{v}\,T_2$  is a forcing tree in $G\,\stackplus{v}\,H$ covering $v$. Then a similar reasoning as above applies.       
\end{proof}

\begin{cor}
For any graphs $G$ and $H$ with a vertex labeled $v$ if $Z_+(G)=T(G)$ and $Z_+(H)=T(H)$, then 
\[
Z_+(G\,\stackplus{v}\,H)=T(G\,\stackplus{v}\,H).\qed
\]
\end{cor}
In \cite{booth2011minimum}, it has been shown that 
\[
M_+(G\,\stackplus{v}\,H)=M_+(G)+M_+(H)-1.
\]
Therefore we have the following as another consequence of Theorem~\ref{T&Z_+ of vertex sum}.   
\begin{cor}
Let $G$ and $H$ be two graphs both with a vertex labeled $v$. If $Z_+(G)=M_+(G)$, then 
\[
Z_+(G\,\stackplus{v}\,H)=M_+(G\,\stackplus{v}\,H).\qed
\]
\end{cor}
\subsection{Coalescence of graphs}
In Section~\ref{Vertex-sum Of Two Graphs}, the positive zero forcing number and the tree cover number of the vertex-sum of two graphs were studied. This provides the motivation for verifying the variations of these parameters in a more general way of merging two graphs. 

Let $H$ be a graph that is a subgraph of both the graphs $G_1$ and $G_2$, then  the \txtsl{coalescence} of $G_1$ and $G_2$ on $H$, denoted by $G_1\,\stackplus{H}\,G_2$, is obtained by the process of merging these graphs through the subgraph $H$. The coalescence of two graphs $G_1$ and $G_2$ on $H$ is obtained by identifying the subgraph $H$ in these two graphs. As an example see Figure~\ref{coalescence}.

\begin{figure}[ht!]
\centering
\unitlength=1pt
\begin{picture}(0,0)

\multiput(-150,0)(50,0){2}{\circle*{6}}
\multiput(-150,-50)(50,0){2}{\circle*{6}}
\put(-50,-25){\circle*{6}}

\put(-150,0){\line(1,0){50}}
\put(-150,0){\line(0,-1){50}}
\put(-150,-50){\line(1,0){50}}
\put(-100,0){\line(0,-1){50}}
\put(-100,0){\line(2,-1){48}}
\put(-100,-50){\line(2,1){48}}

\multiput(50,0)(50,25){2}{\circle*{6}}
\multiput(50,-50)(50,25){2}{\circle*{6}}
\put(50,0){\line(2,1){48}}
\put(50,0){\line(2,-1){48}}
\put(50,0){\line(0,-1){50}}
\put(50,-50){\line(2,1){48}}
\put(100,25){\line(0,-1){50}}

\put(0,-75){\circle*{6}}
\multiput(-25,-125)(50,0){2}{\circle*{6}}
\put(0,-75){\line(-1,-2){25}}
\put(0,-75){\line(1,-2){25}}
\put(-25,-125){\line(1,0){50}}

\multiput(-50,-225)(50,0){2}{\circle*{6}}
\multiput(-50,-275)(50,0){2}{\circle*{6}}
\put(50,-250){\circle*{6}}
\put(50,-200){\circle*{6}}

\put(-50,-225){\line(1,0){50}}
\put(-50,-225){\line(0,-1){50}}
\put(-50,-275){\line(1,0){50}}
\put(0,-225){\line(0,-1){50}}
\put(0,-225){\line(2,-1){48}}
\put(0,-275){\line(2,1){48}}
\put(50,-200){\line(0,-1){50}}
\put(0,-225){\line(2,1){48}}

\put(-125,-60){\makebox(0,0){$G_1$}}
\put(75,-60){\makebox(0,0){$G_2$}}
\put(0,-140){\makebox(0,0){$H$}}
\put(0,-295){\makebox(0,0){$G_1\,\, \,\stackplus{H}\,\, G_2$}}

\end{picture}
\vspace{11cm}
\caption{Coalescence of two graphs}
\label{coalescence}
\end{figure}
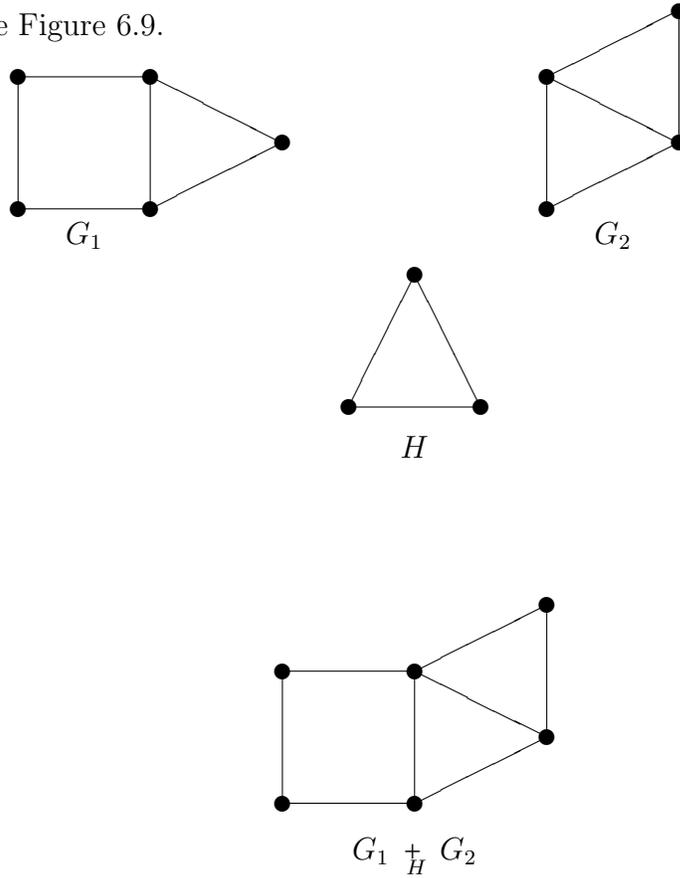

The following theorem gives a lower bound for the tree cover number of the coalescence of a graph $G$ and itself on an induced subgraph $H$.
\begin{thm}\label{T of coalesce}
Let $H$ be a subgraph of $G$. Then
\[
T(G\,\stackplus{H}\,G)\geq 2T(G)-k,
\]
where $k$ is the maximum number of trees that cover the vertices of $H$ in any minimal tree covering of $G$.
\end{thm}
\begin{proof}
Suppose that $T(G\,\stackplus{H}\,G)< 2T(G)-k$. Thus we can cover the vertices of $G\,\stackplus{H}\,G$ with $|\TT|=2T(G)-k-1$ trees. In this covering assume that the vertices of $H$ have been covered by $m$ ($m\leq k$) trees. To cover the vertices of $G\backslash H$  at least $T(G)-m$ trees are needed, where $m\leq k$. On the other hand since the entire graph has been covered by $2T(G)-k-1$ trees, $G\backslash H$ must be covered by at most $T(G)-k-1$ trees, which contradicts  the definition of $k$ and the fact that $T(G)$ is the least number of trees that cover the vertices of $G$.    
\end{proof}

Recall that the positive zero forcing number of a graph is in fact the number of forcing trees in a minimal positive zero forcing process of $G$. Using  similar reasoning as in Theorem~\ref{T of coalesce} we obtain the following.
\begin{thm}\label{Z_+ of coalesce}
Let $H$ be a subgraph of $G$. Then
\[
Z_+(G\,\stackplus{H}\,G)\geq 2Z_+(G)-k,
\]
where $k$ is the maximum number of forcing trees that cover the vertices of $H$ in any minimal positive zero forcing process of $G$.\qed
\end{thm}
The following is easy to prove.

\begin{lem}\label{forcing trees of K_n}
Let $K_n$ be a subgraph of $G$. In any positive zero forcing process of $G$ at most two of the vertices of $K_n$ are in the same forcing tree.
 \end{lem}

\begin{cor}\label{T of coalescence on K_n}
If $G$ is a graph that contains a complete graph on $n$ vertices as a subgraph, then
\[
T(G\,\stackplus{K_n}\,G)=2T(G)-k,
\]
 where $k$ is the maximum number of trees that cover the vertices of $K_n$ in any minimal tree covering of $G$. 
\end{cor}
\begin{proof}
To prove this, we show that, in fact, there is a tree covering of size $2T(G)-k$ for $G\,\stackplus{K_n}\,G$. Let $\TT$ be a minimal tree covering for $G$  with $|\TT|=T(G)$ in which $k$ trees cover all the vertices of $K_n$. 
Clearly, a vertex $v$ of $K_n$ is either covered in a tree in $T\in \TT$ that covers no other vertices of $K_n$ , or only one other vertex of $K_n$ say $u$ is covered in the same tree too. In the first case $T\,\stackplus{v}\,T$ is a tree and in the second case $T\,\stackplus{vu}\,T$ is a tree formed by coalescencening the covering trees of $\TT$.  Thus there is a tree covering of $G\,\stackplus{K_n}\,G$ of size $2T(G)-k$. 

 \end{proof}

If we replace trees in a minimal tree covering of the graph by forcing trees in a minimal  positive zero forcing process of the graph in Corollary~\ref{T of coalescence on K_n} and combine this with Lemma~\ref{forcing trees of K_n}, we obtain the following corollary.

\begin{cor}\label{Z_+ of coalescence on K_n}
If $G$ is a graph that has $K_n$ as a subgraph, then 
\[
Z_+(G\,\stackplus{K_n}\,G)=2Z_+(G)-k,
\]
 where $k$ is the maximum number of forcing trees that cover the vertices of $K_n$ in any minimal positive zero forcing process of $G$.\qed
\end{cor}

The following is a consequence of combining Corollaries~\ref{T of coalescence on K_n} and  \ref{Z_+ of coalescence on K_n}.
\begin{thm}
If  the graph $G$  satisfies $Z_+(G)=T(G)$ and has $K_n$ as a subgraph and any minimal family of forcing trees coincides with a collection of forcing trees, then
\[
Z_+(G\,\stackplus{K_n}\,G)=T(G\,\stackplus{K_n}\,G).\qed
\]
\end{thm}

\subsection{$k$-trees}\label{k_trees}
Being motivated by the fact that all partial $2$-trees (including $2$-trees) satisfy $Z_+(G)=T(G)$, in this section we try to track  the variations between the positive zero forcing number and the tree cover number in $k$-trees with $k>2$. Our results in this section demonstrate that $2$-trees are rather special when it comes to comparing $Z_+(G)$ and $T(G)$. In particular, we show that these two parameters need not agree for ``$k$-trees'' with $k>2$. 

A \textsl{$k$-tree}\index{k-tree@$k$-tree} is constructed inductively by starting with $K_{k+1}$ and  at each step a new vertex is added to the graph and this vertex is adjacent to the vertices of an existing $K_k$. Based on this definition we can deduce that the minimum degree of a $k$-tree is $k$ (the degree of the last vertex added). We will consider a special subfamily of $k$-trees that we call a \textsl{cluster}. 

Let $G$ be a $k$-tree that is constructed in the following way: 
 start with a $H=K_{k+1}$. Add new vertices that are adjacent  to a sub-clique  $K_k$ of $H$. Let $S$ be the set of all distinct subcliques $K_k$ of $H$ such that $K_k\cup v$ is a clique of size $k+1$ in $G$ for some $v\in V(G)\backslash V(H)$. Therefore for each vertex not in $H$ there is  exactly one vertex in $H$ that is not adjacent to it. We call this type of  $k$-tree a \txtsl{cluster}. Obviously the vertices of $H$ form a PZFS for a cluster.    

\begin{thm}\label{Z_+&T of k-trees}
Let $G$ be a $k$-tree that is also a cluster. 
\begin{enumerate}
\item If $|S|\geq3$, then $Z_+(G)=k+1$.
\item If $|S|<3$, then $Z_+(G)=k$.
\item If $|S|=k+1$ and $k$ is  even, then $T(G)=\lceil \frac{k+1}{2} \rceil+1$.
\item If $|S|<k+1$ and $k$ is even, then $T(G)=\lceil \frac{k+1}{2} \rceil$.
\end{enumerate}
\end{thm}   
\begin{proof}
 Since the minimum degree of $G$ is $k$ by Corollary~\ref{Z_+&delta}, $Z_+(G)\geq k$. Suppose $H$ is the first $K_{k+1}$.
 
  To prove the first statement suppose $B$ is a PZFS of $G$ with $|B|=k$. First assume that $B\subset V(H)$. Let $u$ and $v$ be a black and a white vertex of $H$, respectively. According to the definition of the graph $G$, along with the fact that $|S|\geq3$, $v$ and $u$ have at least one common neighbour, say $w$, in $V(G)\backslash V(H)$. Thus $u$ has two white neighbours $v$ and $w$ in the component consisting of $v$ and $w$. This is true for all vertices of $B$ which contradicts the fact that $B$ is a PZFS for $G$. 
Second assume $B\not\subseteq V(H)$. Then there is a vertex $z\in B$ such that $z\not\in H$. If $z$ has only one white neighbour, then $z$ can force it and we return to the first case. If $z$ has more white neighbours, then no vertex of $B$ can perform any force and again we reach a contradiction. Obviously, the set $V(H)$ is a PZFS for $G$. Thus $Z_+(G)=k+1$.

To verify the second statement 
 observe that if $|S|<3$, then there is at least one vertex in $V(H)$, call it $v$, that has at most one neighbour in $V(G)\backslash V(H)$, call it $u$. Thus $u$ along with all its neighbours except $v$ forms a PZFS for $G$. If $v$ has no neighbour in $V(G)\backslash V(H)$,  then any subset of size $k$ of the vertices of $H$ including $v$ is a PZFS for $G$.

For the third statement assume $k$ is even and $|S|=k+1$. Since $K_{k+1}$ is a subgraph of $G$, the tree cover number of $G$ is no less than $\lceil \frac{k+1}{2} \rceil$. Suppose that $\TT$ is a minimal tree covering for $G$ with $|\TT|=\lceil \frac{k+1}{2} \rceil$. No tree in $\TT$ can contain more than two vertices  of $H$. Also since  $T(G)= \lceil \frac{k+1}{2} \rceil$ there is no more than one tree in $\TT$ that covers only one vertex of $H$. Each tree in $\TT$ except one, say $T_1$, contains exactly two vertices of $H$ and $T_1$ contains only a single vertex of $H$. Since $|S|=k+1$, for any vertex, $w\in V(H)$, there is a corresponding vertex in $V(G)\backslash V(H)$ which is adjacent to all of the vertices of $H$ except $w$.  In particular,  if $v$ is the single vertex of $H$ in the tree $T_1$, then there is a vertex $u$ which is adjacent to all vertices in $H$ except $v$. Since $N_G(u)=V(H)\backslash \{v\}$, the vertex $u$ can not be covered by extending any of the trees of $\TT$ and as $u$ is adjacent to both vertices in any tree from $\TT$. This contradicts $\TT$ being a minimal tree covering for $G$. Thus  
\[
T(G)\geq\left\lceil \frac{k+1}{2}\right\rceil+1.
\] 
Finally we will show that we can construct a tree covering of this size.
 Let $\TT'$ be a minimal  tree covering for $H$. Thus $|\TT'|=\lceil \frac{k+1}{2} \rceil$. By considering the construction of $G$  for any $u\in V(G)\backslash V(H)$ except one, call it $v$, there is a tree $T$ in $\TT'$ that is an edge, say $xy$, such that $xu\in E(G)$ and $yu\not\in E(G)$. Hence we can extend $T$ to cover $u$. Thus $\TT'\cup \{v\}$ is a tree covering of size $\lceil \frac{k+1}{2} \rceil+1$ of $G$. Therefore $T(G)=\lceil \frac{k+1}{2} \rceil+1$.             

Finally for the last statement assume $|S|<k+1$ and $k$ is even. Thus there is a vertex, $v$, in  $H$ that is adjacent to all of the vertices in the graph $G$. If we consider a tree covering $\TT'$ for $K_{k+1}$ in which $T=\{v\}$ is a covering tree. we can extend $T$ to cover all the vertices of $V(G)\backslash V(H)$. Thus $T(G)=|\TT'|=\lceil \frac{k+1}{2} \rceil$, which proves the last statement.    
 \end{proof}
Note that if in Theorem~\ref{Z_+&T of k-trees} we have $k=2$, then the positive zero forcing number and the tree cover number coincide.  
 
\begin{thm}\label{k-tree with k odd}
Let $G$ be a $k$-tree for some odd $k$. Then $T(G)=\frac {k+1}{2}$. 
\end{thm}
\begin{proof}
The proof is by induction on the number of vertices. It is clearly true for $K_4$ (the smallest $3$-tree). Assume that the statement is true for all $G'$ with $|V(G')|<n$. Let $G$ be a $k$-tree with $|V(G)|=n$.  Let $v\in V(G)$ be a vertex of degree $k$. Thus by the induction hypothesis  $T(G\backslash{v})=\frac {k+1}{2}$. Let $\TT$ be a tree covering of $G\backslash{v}$ with $|\TT|=T(G\backslash{v})$. By the definition of a $k$-tree, the neighbours of $v$ form a clique on $k$ vertices.  The neighbours of $v$ form a clique, so any tree in $\TT$ covers at most two of the neighbours of $v$. Since $v$ has an odd number of neighbours, there is at least (in fact exactly) one tree, say $T$, in $\TT$ that covers only one of the neighbours of $v$. Therefore we can extend $T$ to cover $v$. Thus $T(G)=\frac {k+1}{2}$.        
\end{proof}

The equality between the positive zero forcing number and the tree cover number of $2$-trees along with our  approach in the proof of Theorem~\ref{Z_+&T of k-trees} and Theorem~\ref{k-tree with k odd} leads us to the following conjecture.
\begin{conj}\label{conj_k_trees}
If a $k$-tree $G$, has $t$ clusters as subgraphs, each with $|S|\geq3$, such that each pair of them share at most  $k-1$ cliques of size $k+1$ then $Z_+(G)=k+t$.   
\end{conj}

\section{The positive zero forcing number and other graph parameters} 
This section is devoted to comparing the positive zero forcing number and other combinatorial graph parameters such as the chromatic number and the independence number.


\subsection{The positive zero forcing number and the chromatic number}
Again it seems logical to ask if there is any relationship between the positive zero forcing number of a graph and the chromatic number of the graph since both parameters  arise from a type of graph colouring.
This, along with the last question of  Section~\ref{colin de} motivates us to prove an inequality between the positive zero forcing number and the chromatic number of a graph.
According to Conjecture~\ref{chi&nu} and inequalities (\ref{M_+(G)<=Z_+(G)}) and (\ref{nu&M}), a weaker comparison can be stated as follows:
\newline The following inequality holds for any graph $G$,
\[
\chi(G)\leq Z_+(G)+1.
\]
This section is devoted to verifying the above inequality for any graph $G$.

\begin{thm}\label{Z_+&max delta}
Let $G$ be a graph.
Then we have
\[
Z_+(G) \geq  \max \{ \delta(H')\,\,  | \,\, H'\,\, \text{is an induced subgraph of}\,\, G\}.
\]
\end{thm}
\begin{proof}
Let $H$ be an induced subgraph of $G$ such that
\[
\delta(H)= \max \{ \delta(H')\,\,  |\,\,  H'\,\,\text{ is an induced subgraph of}\,\, G\}.
\]

Also assume $Z_p$ is a PZFS of $G$  of size $Z_+(G)$. Having $Z_p$ as a PZFS results in a set of covering forcing trees $\mathcal{T}$, with $|\mathcal{T}|=Z_+(G)$, the roots of which are the distinct members of $Z_p$.

We construct a new graph $G'$ which is spanned by $H$ and for which $Z_+(G')\leq Z_+(G).$  Define $G'$ as follows. Fix a forcing tree $T$ in $\mathcal{T}$; assume that $l$ is the first level of $T$ having some vertices of $H$ and let  $S=\{v_1,\ldots,v_k\}$ be the set of vertices of $H$ appearing in level $l$ of $T$. Set  $X''_T$ to be the subgraph of $G$ induced by  $V(H)\cap V(T)$. Add an edge between $u,v\in V(X''_T)$, if $u$ is in a level in $T$ lower than the level of $v$ in $T$ and  the (unique) path in $T$ between $u$ and $v$, neither contains a vertex of $H$ nor a vertex in a  level lower  than the level of $u$, call this new graph $X'_T$. Finally define $X_T$ to be the graph obtained from $X'_T$ by adding the edges connecting $v_1$ to $v_i$, for $i=2,\ldots,k$. Then set $G'$ to be the subgraph of $G$ induced by
\[
\bigcup_{T\in \mathcal{T}} X_T.
\]
Note that $X_T$ will be a rooted tree with the root $v_1$. 

We claim that $\mathcal{T'}=\{X_T\,|\, T\in \mathcal{T}\}$ is a set of forcing trees in $G'$. It is enough to show that each force in $X_T$ is allowed in $G'$, for any $X_T\in \mathcal{T'}$. First, note that  $v_1$ can force all the vertices in $S \backslash v_1$ since after removing the vertices of $G\backslash H$ from $T$, the vertices of $S$ will be in different components. Also $u\in V(G')$ can force $v\in V(G')$ in $G'$ if they are adjacent in  $X_T$ and $u$ lies in a higher level than $v$. To see this, let  $u, u_1, \ldots, u_s, v$ be the only path connecting $u$ to $v$ in $T$. Since $u$ forces $u_1$ in $T$, $u_1$ must be the only white neighbour of $u$ in its component. Therefore  by removing the vertices $u_1, \ldots, u_s$ and adding the edge $uv$ to the graph, $v$ will be the only white neighbour of $u$ in its component. Thus $u$ forces $v$ in $X_T$.
We, therefore, have
\[
|\mathcal{T'}|\geq Z_+(G');
\]
and since $H$ spans $G'$, we have $\delta(G')\geq \delta(H)$. Thus
\[
Z_+(G)=|\mathcal{T}|\geq |\mathcal{T'}|\geq Z_+(G')\geq \delta(H),
\]
which proves the desired inequality.
\end{proof}
Using this result we deduce the following.
\begin{cor}\label{Z_+&delta}
For any graph $G$ we have
\[
Z_+(G)\geq \delta(G).\qed
\]
\end{cor}
Note that, the bound in the above inequality is tight for some graphs such as trees, complete graphs and even cycles.

Using Theorem~\ref{Z_+&max delta} and Lemma~\ref{chi&delta}, we can deduce the following.
\begin{cor}\label{Z_+& chi}
For any graph $G$ we have
\begin{equation}\label{Z_+(G) and x(G)}
\chi(G)\leq Z_+(G)+1.\qed
\end{equation}
\end{cor}
Observe that Corollary~\ref{Z_+& chi} is stronger than Corollary~\ref{Z& chi} as
\[
Z_+(G)\leq Z(G).
\]

In $1943$ Hadwiger proposed the following conjecture \cite{hadwiger1943uber}: 
\newline ``Every $n$-chromatic graph contains a $K_n$ minor, a complete graph on $n$ points''.

By establishing Inequality (\ref{Z_+(G) and x(G)}), it seems tempting to verify Hadwiger's conjecture for $Z_+(G)$ i.e., to verify the truth of the following statement:
``If $Z_+(G)\geq n-1$ then $G$ contains a $K_n$ minor.'' 
 It turns out that this equality fails for outerplanar graphs. According to Corollary~\ref{Z_+& chi} we have $Z_+(G)=T(G)$ if $G$ is an outerplanar graph. The tree cover number and consequently the positive zero forcing number of an outerplanar graph, $G$, can be arbitrarily large while $G$ does not contain even a $K_4$ minor.

\subsection{The positive zero forcing number and the independence number}
 An \txtsl{independent set}, or \txtsl{coclique}, is a set of vertices of a graph $G$ of which no pair is adjacent. The \txtsl{independence number}, which is denoted by $\alpha(G)$, of $G$ is the size of the largest independent set of $G$.
Based on the definition of the positive zero forcing number of a graph $G$ it seems reasonable that it should be related to the independence number of $G$. In this section we establish a lower bound and an upper bound for  the positive zero forcing number of a graph $G$ using the independence number of $G$. 
\begin{lem} \label{chi&alpha}
Let $G$ be a graph with chromatic number $\chi(G)$ and independence number $\alpha(G)$. Then
\[
\chi(G)\geq \left\lceil \frac{|G|}{\alpha(G)}\right\rceil.
\]
\end{lem} 
\begin{proof}
The inequality follows from the fact that in a proper colouring of $G$ every colour can be used at most $\alpha(G)$ times. 
\end{proof}

\begin{thm} \label{Z_+&alpha}
Let $Z_+(G)$ and $\alpha(G)$  be the positive zero forcing number and the independence number of the graph $G$ respectively. Then,
\[
\left\lceil \frac{|G|}{\alpha(G)}\right\rceil-1\leq Z_+(G)\leq |G|-\alpha(G).
\]
\end{thm}
\begin{proof}
The first inequality follows from combining Corollary~\ref {Z_+& chi} and Lemma~\ref {chi&alpha}.
To show the second inequality, let $H$ be a set of vertices in $G$, no two of which are adjacent with $|H|=\alpha(G)$. Let $L=V(G)\backslash V(H)$. Then $L$ is a positive zero forcing set for $G$, since removing $L$ from $G$ results in having $\alpha(G)$ isolated vertices which can be coloured black by the vertices in $L$. Thus $Z_+(G)\leq |L|$.
\end{proof}
 Note that in Theorem~\ref{Z_+&alpha} both upper and lower bounds can be tight. For example the lower bound holds with equality for  odd cycles and the graphs $C_{2k+1}\,\stackplus{v}\, P_{2k'+1}$, where $k,k'\geq 1$ and $v$ is an endpoint of the path.  On the other hand, complete graphs and complete bipartite graphs are examples for which the upper bound is tight. 
 
If $Z_+(G)=n-2$, then according to Theorem~\ref{Z_+&alpha}, $\alpha(G)\leq 2$. If $\alpha(G)=1$, then $G$ is a complete graph, hence $Z_+(G)=n-1$. Thus $\alpha(G)=2$. Therefore graphs with $Z_+(G)=n-2$ are other examples for which the second inequality in Theorem~\ref{Z_+&alpha} is tight.





\chapter{Conclusion}\label{conclusion}
The graph parameters zero forcing number and positive zero forcing number, offer a wide range of prospective research.
It is very appealing to try to find bounds or the exact value of the maximum nullity and minimum rank of graphs without going through  direct algebraic calculations. Therefore, discovering more families of graphs satisfying $M(G) = Z(G)$ or $M_+(G) = Z_+(G)$, is  an interesting project.
\vspace{.3cm}
\\
{\bf Problem 1.} \textsl{For which families of graphs $G$ do we have $M(G)=Z(G)$ or $M_+(G)=Z_+(G)$?}
\vspace{.3cm}

In Section~\ref{graphs_with_Z=P}, we introduced families of graphs for which the zero forcing number and the path cover number coincide. In fact, we showed that for the family of block-cycle graphs this is true. It seems, however, that  there are more  families for which the equality holds between these two parameters. For example, the graph $G=K_4-e$, where $e$ is an edge of $K_4$, has $Z(G)=P(G)$. It is, therefore, natural to propose the  following.
\vspace{.3cm}
\\
{\bf Problem 2.} \textsl{Characterize all the graphs $G$ for which $Z(G)=P(G)$.}
\vspace{.3cm}

We showed in Section~\ref{ZFS_basics} that for trees, any minimal path cover coincides with a collection of forcing chains. It seems that this is also the case for the block-cycle graphs. Note that, if this is true for any graph $G$ satisfying $Z(G)=P(G)$, then
Conjecture~\ref{conj_Z&P_vertex_sum} follows from Theorem~\ref{zero forcing of vertex-sum} and Observation~\ref{path covering of vertex-sum}. Hence, we state the following.
\vspace{.3cm}
\\
{\bf Problem 3.} \textsl{Let $G$ be a graph with $Z(G)=P(G)$. Is it true that any minimal path cover of $G$ coincides with a collection of forcing chains of $G$?}
\vspace{.3cm}

In Section~\ref{Forcing Trees}, we proved the equality $Z_{+}(G) = T(G)$ where $G$ is an outerplanar graph. The structure of a planar embedding of outerplanar graphs was the key point to establish the equality. It looks  like there are some non-outerplanar graphs whose pattern follows the same structure. Hence an approach consisting of discovering more graphs satisfying $Z_+(G)=T(G)$,  seems promising. We, therefore, propose the following problem.
\vspace{.3cm}
\\
{\bf Problem 4.} \textsl{Characterize all the graphs $G$ for which $Z_{+}(G) = T(G)$.}
\vspace{.3cm}

In Chapter~\ref{ZFS} we discussed the zero forcing number of the corona of  graphs as well as the join of graphs. In addition, in Chapter~\ref{PZFS} we studied the problem of finding the positive zero forcing number for the coalescence of graphs. Thus the problem of evaluating the (positive) zero forcing number for some other  graph operations such as  the Cartesian product, the direct product, the strong product or the complement of graphs may also be interesting. 
\vspace{.3cm}
\\
{\bf Problem 5.} \textsl{What can we say about the (positive) zero forcing number of graph operations such as different  products and  complements of graphs?}
\vspace{.3cm}

Recall from Section~\ref{graph_with_Z=M} that the equality $M(G)=Z(G)$ does not hold for all chordal graphs. On the other hand, however,  we showed that for some families of chordal graphs, the equality does indeed hold. It might be possible to  describe the conditions under which a chordal graph satisfies the equality $M(G)=Z(G)$.  We phrase this problem as follows.  
\vspace{.3cm}  
\\
{\bf Problem 6.} \textsl{For which chordal graphs $G$ does the equality $M(G)=Z(G)$ hold?}
\vspace{.3cm}

Recall from Subsection~\ref{k_trees} that the equality $Z_+(G)=T(G)$ holds for $2$-trees but fails to hold for general $k$-trees. Recall, also, that we evaluated the positive zero forcing number and the tree cover number for some families of $k$-trees, namely the clusters. We are, therefore, interested in finding the positive zero forcing number and the tree cover number of the general $k$-trees. Hence solving Conjecture~\ref{conj_k_trees} is one of the most interesting future works in this area. 
\vspace{.3cm}
\\
{\bf Problem 7.} \textsl{If a $k$-tree $G$, has $t$ clusters as subgraphs, each with $|S|\geq3$, such that each pair of them share at most $k-1$ clique(s) of size $k+1$ then does the equality $Z_+(G)=k+t$ hold}?
\vspace{.3cm}

Finally, note that the parameters $\mu$ and $\nu$ discussed in Chapter~\ref{ZFS} are of great importance because of their minor monotonicity property.   Studying Conjecture~\ref{chi&mu} and Conjecture~\ref{chi&nu} is yet another interesting direction to follow. In particular, a follow up project could consist of investigating the Hadwiger conjecture which is one of the deepest unsolved problems in graph theory and would directly verify Conjecture~\ref{chi&mu} and Conjecture~\ref{chi&nu}.

\printindex
\addcontentsline{toc}{chapter}{References}
\bibliographystyle{plain}
\bibliography{bibliography}
\end{document}